\documentclass[11pt]{amsart}

\usepackage[applemac]{inputenc}
\usepackage[T1]{fontenc}
\usepackage{mathtools}
\usepackage{amssymb}
\usepackage{amsbsy}
\usepackage{bm}
\usepackage{bbm} 
\usepackage{mathrsfs}
\usepackage{fge}
\usepackage[all]{xy}

\usepackage[margin=2.5cm]{geometry} 

\usepackage{enumitem}

\numberwithin{equation}{section}

\newtheorem{theorem}[equation]{Theorem}
\newtheorem{lemma}[equation]{Lemma}
\newtheorem{proposition}[equation]{Proposition}
\newtheorem{corollary}[equation]{Corollary}

\theoremstyle{definition}
\newtheorem{definition}[equation]{Definition}
\theoremstyle{remark}
\newtheorem{remark}[equation]{Remark}

\newtheorem{example}[equation]{Example}
\newtheorem*{acknowledgments}{Acknowledgments}

\let\oldmathcal\mathcal
\renewcommand{\mathcal}{\mathscr}
\renewcommand{\setminus}{\mathbin{\fgebackslash}}

\def\invlim{\mathop{\vtop{\ialign{##\crcr$\hfill{\lim}\hfil$\crcr
  \noalign{\kern1pt\nointerlineskip}\leftarrowfill\crcr\noalign
  {\kern -3pt}}}}\limits}
\def\dirlim{\mathop{\vtop{\ialign{##\crcr$\hfill{\lim}\hfil$\crcr
  \noalign{\kern1pt\nointerlineskip}\rightarrowfill\crcr\noalign
  {\kern -3pt}}}}\limits}
\def\lomapr#1{\smash{\mathop{\relbar\joinrel\longrightarrow}\limits^{#1}}}

\let\emptyset\varnothing

\newcommand{\jcdot}{{\scriptscriptstyle\bullet}}

\newcommand{\breveqp}{\breve{\mathbf{Q}}_p}

\newcommand{\Qp}{\mathbf{Q}_p}

\newcommand{\R}{\mathrm {R} }

\newcommand{\LL}{\mathrm {L} }

\newcommand{\qp}{{\mathbf Q}_{p}}

\newcommand{\ad}{\operatorname{ad} }

\newcommand{\Ind}{\operatorname{Ind} }

\newcommand{\cont}{\operatorname{cont} }

\newcommand{\eet}{\operatorname{\acute{e}t} }

\newcommand{\Spa}{\operatorname{Spa} }

\newcommand{\Hom}{\operatorname{Hom} }

\newcommand{\Ext}{\operatorname{Ext} }

\newcommand{\Gal}{\operatorname{Gal} }

\newcommand{\Cone}{\operatorname{Cone} }

\newcommand{\gr}{\operatorname{gr} }

\newcommand{\sff}{{\mathcal{F}}}

\newcommand{\sg}{{\mathcal{G}}}

\newcommand{\sbb}{{\mathcal{B}}}
\newcommand{\scc}{{\mathcal{C}}}

\newcommand{\sll}{{\mathcal{L}}}
\newcommand{\sn}{{\mathcal N}}
\newcommand{\so}{{\mathcal O}}

\newcommand{\sm}{{\mathcal{M}}}
\newcommand{\spp}{{\mathcal{P}}}

\newcommand{\wt}{\widetilde}
\newcommand{\wh}{\widehat}

\newcommand{\Z}{ {\mathbf Z} }

\newcommand{\Q}{ {\mathbf Q}}
\newcommand{\N}{{\mathbf N}}
\newcommand{\rg}{\R\Gamma}

\DeclareMathOperator{\Ord}{Ord}
\DeclareMathOperator{\car}{char}
\DeclareMathOperator{\supp}{supp}
\DeclareMathOperator{\ev}{ev}

\mathchardef\mhyphen="2D
\newcommand{\ci}{c\mhyphen{}i}
\renewcommand{\twoheadrightarrow}{\rightarrow\mathrel{\mspace{-14mu}}\rightarrow}
\newcommand{\xtwoheadrightarrow}[2][]{\xrightarrow[#1]{#2}\mathrel{\mkern-14mu}\rightarrow}
\newcommand{\iso}{\stackrel{\sim}{\to}}
\newcommand{\HOrd}[1][*]{H^{#1}\!\Ord}
\newcommand{\F}{\mathbf{F}}
\newcommand{\Hb}{\bm{H}}
\newcommand{\Gb}{\bm{G}}
\newcommand{\Gm}{\mathbf{G}_\mathrm{m}}
\newcommand{\Zb}{\bm{Z}}
\newcommand{\Sb}{\bm{S}}
\newcommand{\Tb}{\bm{T}}
\newcommand{\Bb}{\bm{B}}
\newcommand{\Ub}{\bm{U}}
\newcommand{\Pb}{\bm{P}}
\newcommand{\Mb}{\bm{M}}
\newcommand{\Nb}{\bm{N}}
\newcommand{\Zc}{\oldmathcal{Z}}
\newcommand{\Nc}{\oldmathcal{N}}
\newcommand{\Zcb}{\bm{\Zc}}
\newcommand{\Ncb}{\bm{\Nc}}
\newcommand{\Fbar}{\overline{F}}
\newcommand{\Bbar}{\overline{B}}
\newcommand{\Ubar}{\overline{U}}

\newcommand{\ubar}{\bar{u}}

\newcommand{\Pbar}{\overline{P}}
\newcommand{\Nbar}{\overline{N}}
\newcommand{\nbar}{\bar{n}}
\newcommand{\Bbbar}{\bm{\Bbar}}
\newcommand{\Ubbar}{\bm{\Ubar}}

\newcommand{\Pbbar}{\bm{\Pbar}}
\newcommand{\Nbbar}{\bm{\Nbar}}
\newcommand{\Wt}{\widetilde{W}}
\newcommand{\wtw}{\widetilde{w}}
\newcommand{\Wh}{\widehat{W}}
\newcommand{\whw}{\widehat{w}}
\newcommand{\varepsbar}{\bar{\varepsilon}}
\newcommand{\alphat}{\widetilde{\alpha}}
\newcommand{\mub}{\bm{\mu}}

\newcommand{\LC}{\mathrm{LC}}
\newcommand{\LCu}{\LC_\mathrm{u}}
\newcommand{\LCc}{\LC_\mathrm{c}}
\newcommand{\Arm}{\mathrm{A}}
\newcommand{\Erm}{\mathrm{E}}
\newcommand{\GL}{\mathrm{GL}}
\newcommand{\SL}{\mathrm{SL}}

\newcommand{\St}{\mathrm{St}}
\newcommand{\der}{\mathrm{der}}
\newcommand{\ssc}{\mathrm{sc}}
\newcommand{\is}{\mathrm{is}}
\newcommand{\anis}{\mathrm{an}}
\newcommand{\abs}{\mathrm{abs}}
\usepackage{graphics}

\title[$p$-adic \'etale cohomology of period domains]{$p$-adic \'etale cohomology of period domains}
\author[P. Colmez]{Pierre Colmez}
\address{CNRS, IMJ-PRG, Sorbonne Universit\'e, 4 place Jussieu, 75005 Paris, France}
\email{pierre.colmez@imj-prg.fr}
\author[G. Dospinescu]{Gabriel Dospinescu}
\address{CNRS, UMPA, \'Ecole Normale Sup\'erieure de Lyon, 46 all\'ee d'Italie, 69007 Lyon, France}
\email{gabriel.dospinescu@ens-lyon.fr}
\author[J. Hauseux]{Julien Hauseux}
\address{Univ.\ Lille, CNRS, UMR 8524 - Laboratoire Paul Painlev\'e, F-59000 Lille, France}
\email{julien.hauseux@univ-lille.fr}
\author[W. Nizio{\l}]{Wies{\l}awa Nizio{\l}}
\address{CNRS, IMJ-PRG, Sorbonne Universit\'e, 4 place Jussieu, 75005 Paris, France}
\email{wieslawa.niziol@ens-lyon.fr}
\thanks{The research of J.H. was partially supported by the projects ANR-11-LABX-0007-01 CEMPI and ANR-16-IDEX-0004 ULNE.
The research of P.C., G.D., and W.N. was partially supported by the projects ANR-14-CE25-0002-01 PERCOLATOR and ANR-19-CE40-0015-02 COLOSS}
\date{\today}

\begin{document}

\begin{abstract}
We compute the $p$-torsion and $p$-adic \'etale cohomologies with compact support of period domains over local fields in the case of basic isocrystals for quasi-split reductive groups.
As in the cases of $\ell$-torsion or $\ell$-adic coefficients, $\ell\neq p$, considered by Orlik, the results involve generalized Steinberg representations.

For the $p$-torsion case, we follow the method used by Orlik in his computations of the $\ell$-torsion \'etale cohomology using as a key 
new ingredient the computation of $\Ext$ groups between mod $p$ generalized Steinberg representations of $p$-adic groups.
For the $p$-adic case, we don't use Huber's definition of \'etale cohomology with compact support as Orlik did since it seems to give spaces that are much too big; instead we use continuous \'etale cohomology with compact support.
\end{abstract}

\makeatletter
\def\@@and{\unskip}
\makeatother

\maketitle

\setcounter{tocdepth}{1}
\tableofcontents

\section{Introduction}

   Let $p$ be a prime number.
   One of the main results of \cite{CDN3} and \cite{CDN4} is the computation of the geometric
   $p$-adic \'etale cohomology of Drinfeld $p$-adic symmetric spaces in arbitrary dimension. The final result is analogous to the one in the  case of $\ell$-adic \'etale cohomology with $\ell\neq p$, which was known by the work of Schneider and Stuhler \cite{SS}. The Drinfeld symmetric spaces are among the most classical examples of 
   {\it $p$-adic period domains} but it is well-known that they are very special\footnote{See \cite[Sec.\ 3]{Rfin} for a list of such properties.}.  In fact, the proofs in \cite{CDN3} and \cite{CDN4} use these unique properties of Drinfeld spaces
  hence it was not clear to us whether the results of loc.\ cit.\ would extend to more general $p$-adic period domains.
   
    The purpose of this paper is to show that, for
   {\it compactly supported} $p$-torsion \'etale cohomology, it is possible to treat  fairly general $p$-adic period domains. Moreover, the result is similar to the one for $\ell$-torsion, 
$\ell\neq p$, cohomology with compact support  
obtained\footnote{The Euler characteristic of period domains was known before, thanks to Kottwitz and Rapoport, see \cite{DOR} for a beautiful presentation.} by Dat \cite{Dat} (for the Drinfeld spaces) and  by Orlik \cite{Ol} (in general).  
That this is the case is a little surprising since, as we will explain below, 
the $p$-adic \'etale cohomology with compact support (in the sense of Huber \cite{H2}) of $p$-adic period domains is not at all similar 
to its $\ell$-adic counterpart, $\ell\neq p$,  computed by Orlik \cite{Oc},
and seems to produce not very useful objects.  On the other hand, the continuous
compactly supported cohomology that we define gives reasonable objects
(at least in the situation we consider or in the case of the complement of a
subvariety in a proper analytic variety as considered\footnote{
In both cases this continuous compactly supported cohomology coincides with the naive one.}
 in~\cite{chinois}). 
   
   While the arguments in \cite{CDN3} and \cite{CDN4} are based on $p$-adic Hodge theory (via the syntomic method) and its integral versions \cite{BMS1,BMS2,CK},
 this paper combines a beautiful geometric construction due to Orlik \cite{Ol} with a vanishing result for extensions between mod $p$ representations of $p$-adic reductive groups.
 The proof of the second result is the main difference with the $\ell\neq p$ case.
We are not able to recover the results of \cite{CDN3} and \cite{CDN4} using the methods used here, and conversely the methods in loc.\ cit.\ do not seem to 
give the results obtained in this paper for Drinfeld spaces 
(Poincar\'e duality with $p$-torsion coefficients holds for ``almost proper''
analytic varieties~\cite{chinois}, but probably does not hold for general 
analytic varieties, at least in a naive sense).
 Orlik did recover in \cite{Opro} the computation of $p$-adic pro-\'etale cohomology of Drinfeld spaces from \cite{CDN3} using his method -- which is the one of this paper as well -- but one encounters  considerable technical 
difficulties\footnote{For example,
 the rational $p$-adic pro-\'etale cohomology of an open ball has
a simple description in terms of differential forms~\cite{CN2}, 
but the integrality conditions coming from 
the $p$-adic \'etale cohomology make the computations subtler.}  
when working with the \'etale cohomology instead of the pro-\'etale one. 

 \subsection{Notation}   In order to state the main results of this paper we need to introduce some notation. 
      Let $C$ be the  completion of an algebraic closure  of $\Q_p$ and let 
$(G,[b],\{\mu\})$ be a {\it local Shtuka datum} over 
   $\Q_p$. Here $G$ is a connected reductive group over $\Q_p$, 
   $[b]$ is an element of the Kottwitz set $B(G)$
     of $\sigma$-conjugacy classes in\footnote{ 
   $\breve{\Q}_p$ is the completion of the maximal unramified extension of $\Q_p$ in $C$.} $G(\breve{\Q}_p)$, i.e., an isomorphism class $N_b$ of isocrystals with $G$-structure over $\breve{\Q}_p$, and
   $\{\mu\}$ is a conjugacy class of geometric cocharacters of $G$. Moreover, we ask that 
   $[b]$ lies in the Kottwitz set\footnote{For the main result of the paper it would be enough to assume that 
   $[b]$ belongs to the larger set $A(G,\mu)$, since all we need is that the period domain is nonempty, which is equivalent to $[b]\in A(G,\mu)$ by a result of Fontaine and Rapoport \cite{FR}.} $B(G,\mu)$, a certain finite subset of $B(G)$ defined roughly by a comparison between the 
   Hodge polygon attached to $\mu$ and the Newton polygon attached to $N_b$ (see \cite[Ch.\ 6]{K2} for the precise definition of the set $B(G, \mu)$). This assumption is made so that 
   the period spaces whose cohomology we want to compute are not empty.  
   
   The pair $(G,\{\mu\})$ gives rise to a generalized flag variety\footnote{If $G$ is quasi-split over $\Q_p$, which will be the case in our main result, we can choose 
   $\mu\in \{\mu\}$ defined over $E$ and then $\mathcal{F}=\mathcal{F}(G, \{\mu\})$ is the quotient of $G_E$ by the parabolic subgroup 
$P(\mu)$ associated to $\mu$.}
$\mathcal{F}=\mathcal{F}(G,\{\mu\})$ defined over 
   the field of definition $E$ of $\{\mu\}$, a finite extension of $\Q_p$ and a local analogue of the reflex field in the theory of Shimura varieties. We will consider $\mathcal{F}$ as an adic space over ${\rm Spa}(E,\mathcal{O}_E)$. Letting  
   $\breve{E}=E\breve{\Q}_p$, the 
    {\it $p$-adic period domain} introduced by Rapoport and Zink \cite{RZ} $$\mathcal{F}^{\rm wa}=\mathcal{F}^{\rm wa}(G, [b], \{\mu\})$$ is a partially proper 
open subset of $\mathcal{F}\otimes_E {\breve{E}}$, 
classifying the weakly admissible filtrations of type $\{\mu\}$ on the isocrystal $N_b$. Basic examples of 
$p$-adic period domains are the adic affine spaces, the projective spaces, and  the Drinfeld symmetric spaces (complements of the union of all $\Q_p$-rational hyperplanes in the projective spaces). 

    As we have already mentioned, 
 Orlik computed in \cite{Ol} the $\ell$-adic compactly supported \'etale cohomology of 
these period domains when $G$ is quasi-split over $\Q_p$, 
$[b]$ is a basic class, and $\ell\ne p$ is a sufficiently generic prime number.
We will also assume that $G$ is quasi-split over $\Q_p$ and that 
    $b\in G(\breve{\Q}_p)$ is basic and $s$-decent\footnote{The hypothesis that $b$ is decent is harmless, since any $\sigma$-conjugacy class in $G(\breve{\Q}_p)$ contains an $s$-decent element for some positive integer $s\geq 1$.}. We refer the reader to the main body of the article for these notions, introduced by Kottwitz (for the first one) and Rapoport-Zink (for the second one). 
    This implies, for instance, that $b\in G(\Q_{p^s})$ and that the period domain  $\mathcal{F}^{\rm wa}=\mathcal{F}^{\rm wa}(G, [b], \{\mu\})$ 
 has a canonical model (still denoted $\mathcal{F}^{\rm wa}$) over $E_s=E\Q_{p^s}\subset \overline{\mathbf{Q}}_p$. 
Let $J_b$ be the automorphism group of $N_b$.
It is a connected reductive group over 
$\Q_p$, which is an inner form of $G$ (this is equivalent to  $b$ being basic). The natural action of 
$G(\breve{\Q}_p)$ on the flag variety $\mathcal{F}\otimes_E \breve{E}$ induces an action of 
$J_b(\Q_p)$ on the period domain $\mathcal{F}^{\rm wa}$.   In particular, we obtain an action of 
$J_b(\Q_p)\times \sg_{E_s}$, $\sg_{E_s}=\Gal(\overline{\Q}_p/E_s)$,  on $H^*_{\eet, c}(\mathcal{F}^{\rm wa}_C, \Z/\ell^n)$ and $H^*_{\eet, c}(\mathcal{F}^{\rm wa}_C, \Z_{\ell})$, for any prime $\ell$. 
The main theorem of this paper gives a simple description of these representations in the case $\ell=p$.

Let $T$ be a maximal torus of $G$ such that $\mu$ factors through $T$, and let
$W=N(T)/T$ be the (absolute) Weyl group of $G$ with respect to $T$, which acts naturally on 
$X_*(T)$. Let 
$W^{\mu}$ be the set of Kostant representatives with respect to 
$W/{\rm Stab}(\mu)$, i.e., the representatives of shortest length in their cosets. 
The group $\sg_{E_s}$ acts on $W$ and preserves $W^{\mu}$ since 
$\mu$ is defined over $E_s$. One can associate to each $\sg_{E_s}$-orbit 
$[w]\in W^{\mu}/\sg_{E_s}$ the following objects:

$\bullet$ An integer $l_{[w]}$, the length of any element of $[w]$. 
  
$\bullet$ For any prime $\ell$, a $\Z/\ell^n[\sg_{E_s}]$-module $\rho_{[w]}(\Z/\ell^n)$, which is simply the 
$\Z/\ell^n$-module of $\Z/\ell^n$-valued functions on 
$[w]$, with the obvious $\sg_{E_s}$-action twisted (\`a la Tate) by $-l_{[w]}$.
    
We will simply write $J$ instead of $J_b$ from now on.
Choose a maximal $\qp$-split torus
$S$ of $J_{\rm der}$ contained in $T$ and a minimal $\qp$-parabolic subgroup $P_0$ of $J$ containing $S$. Let 
$\Delta\subset X^*(S)$ be the associated set of relative simple roots. For each subset 
$I$ of $\Delta$, we denote by $P_I$ the corresponding standard $\qp$-parabolic subgroup of 
     $J$, so that $P_{\emptyset}=P_0$ and $P_{\Delta}=J$. Consider the compact $p$-adic manifold 
$$X_I=J(\qp)/P_I(\qp).$$
If $R$ is an abelian group, let 
$$v_{P_I}^J(R)=i_{P_I}^J(R)/\sum_{I\subsetneq I'}i_{P_I}^J(R),\quad i_{P_I}^J(R)=\LC(X_I, R),$$
     be the corresponding generalized Steinberg representation of $J(\qp)$, with coefficients in 
     $R$ (here $\LC(?, R)$ is the space of locally constant functions on 
     $?$ with values in $R$).
    
     Finally, choose an invariant inner product $(-,-)$ on $G$, i.e., an inner product on 
     $X_*(T')\otimes \Q$, for all maximal tori $T'$ of $G$, compatible with the adjoint action of 
     $G(\overline{\mathbf{Q}}_p)$ and the natural action of $\sg_{\qp}$ on maximal tori of $G$. It induces 
     an invariant inner product on $J$ as well. For each $\sg_{E_s}$-orbit $[w]$ 
     of $w\in W^{\mu}$, define ($\nu:=\nu_b$, the Newton map of $b$)
       $$I_{[w]}=\{\alpha\in \Delta\mid (w\mu-\nu, \omega_{\alpha})\leq 0\},
\quad P_{[w]}:= P_{I_{[w]}},$$
       where $\omega_{\alpha}\in X_*(S)\otimes \Q$, $\alpha\in \Delta$,  form  the dual basis of $\Delta$.

\subsection{The main result}Recall that, for $\ell\neq p$, we have the following computation of Orlik.

\begin{theorem}[Orlik \cite{Ol, Oc}] \label{main-Orlik}
 Let $(G,[b], \{\mu\})$ be a local Shtuka datum with $G/\qp$  quasi-split, $b\in G(\breve{\Q}_p)$ basic and $s$-decent. Let $\ell\neq p$ be  sufficiently generic\footnote{See \cite[Sec.\ 1]{Ol} for the definition.} with respect to $G$. 

 There are isomorphisms of $\sg_{E_s}\times J(\qp)$-modules
\begin{align*}
  H^*_{\eet,c}(\sff^{\rm wa}_C,\Z/\ell^n) & \simeq \bigoplus_{[w]\in W^{\mu}/\sg_{E_s}} v^J_{{P_{[w]}}}(\Z/\ell^n)\otimes \rho_{[w]}(\Z/\ell^n)[-n_{[w]}],\\
 H^*_{\eet,{\rm c, Hu}}(\sff^{\rm wa}_C,\Z_\ell) & \simeq \bigoplus_{[w]\in W^{\mu}/\sg_{E_s}} v^J_{{P_{[w]}}}(\Z_\ell)\otimes \rho_{[w]}(\Z_{\ell})[-n_{[w]}],
\end{align*}
where $n_{[w]}=2l_{[w]}+|\Delta\setminus I_{[w]}|$ and   $H^*_{\eet,{\rm c, Hu}}$ denotes  Huber's compactly supported cohomology.
In particular, the action of $J(\Q_p)$ 
on $H^*_{\eet,{\rm c, Hu}}(\sff^{\rm wa}_C,\Z_\ell)$ is smooth.
\end{theorem}

 Our main result is the following computation.
\begin{theorem}\label{main0} Let $(G,[b], \{\mu\})$ be a local Shtuka datum with $G/\qp$ quasi-split, $b\in G(\breve{\Q}_p)$ basic and $s$-decent. 
Assume that 
$p \neq 2$.

 There are isomorphisms of $\sg_{E_s}\times J(\qp)$-modules
\begin{align*}
H^*_{\eet,c}(\sff^{\rm wa}_C,\Z/p^n) & \simeq \bigoplus_{[w]\in W^{\mu}/\sg_{E_s}} v^J_{{P_{[w]}}}(\Z/p^n)\otimes \rho_{[w]}(\Z/p^n)[-n_{[w]}],\\
H^*_{\eet,{\rm c}}(\sff^{\rm wa}_C,\Z_p) & \simeq \bigoplus_{[w]\in W^{\mu}/\sg_{E_s}} v^{J,\cont}_{{P_{[w]}}}(\Z_p)\otimes \rho_{[w]}(\Z_p)[-n_{[w]}],
\end{align*}
where $H^*_{\eet,{\rm c}}$ denotes the continuous compactly supported cohomology,  $v^{J,\cont}_{P_I}(\Z_p)=\varprojlim_{n} v_{P_I}^J(\Z/p^n)$ denotes {\em continuous} Steinberg representations, and $\rho_{[w]}(\Z_p)=\varprojlim_{n} \rho_{[w]}(\Z/p^n)$.
\end{theorem}
\begin{remark} \label{pathology} The result for torsion coefficients in Theorem \ref{main0} is analogous to the one of Orlik quoted above. The analog of Orlik's second isomorphism is false:
in the case of the adic affine space ${\mathbb A}^1_{\Q_p}$, which is a period domain for the group $G={\mathbb G}_{m,\Q_p}\times {\mathbb G}_{m,\Q_p}$, 
we obtain (in the appendix) the isomorphism
$$
H^2_{\eet,{c},{\rm Hu}}({\mathbb A}^1_C,\Z_p(1))
\simeq (\so_{{\mathbb P}^1_C,\infty}/C)\oplus\Z_p\,,
$$
where $\so_{{\mathbb P}^1_C,\infty}$ is the stalk of analytic functions
at $\infty$. This result is to be compared
with the isomorphism 
$H^2_{\eet,{c},{\rm Hu}}({\mathbb A}^1_C,\Z_{\ell}(1))\simeq \Z_{\ell}$, for $\ell\neq p$. Note, moreover, that the 
action of $G(\Q_p)$ on $H^2_{\eet,{c},{\rm Hu}}({\mathbb A}^1_C,\Z_p(1))$ is not smooth. 
\end{remark}
\begin{remark}
The case of ${\mathbb A}^1_{\Q_p}$ suggests that 
 Huber's definition is not the right one for $p$-adic coefficients.
On the other hand, the continuous compactly supported cohomology\footnote{Instead of requiring a proper support for a compatible sequence of global sections we just take sequences of properly supported global sections.}
$$\R\Gamma_{\eet, c}(X, \Z_p):=\R\varprojlim_n\R\Gamma_{\eet,c}(X,\Z/p^n),$$
gives sensible results,
as  Theorem \ref{main0} shows.   
In this particular case,  we have an isomorphism
$$
H^i_{\eet,{\rm c}}(\sff^{\rm wa}_C,\Z_p)\simeq H^i_{\eet,c, {\rm naive}}(\sff^{\rm wa}_C,\Z_{p}):=\varprojlim_nH^i_{\eet,c}(\sff^{\rm wa}_C,\Z/p^n), \quad i\geq 0,
$$
with the naive version of compactly supported cohomology. We note that in a recent preprint \cite{chinois}, Lan-Liu-Zhu  prove a rational Poincar\'e duality for $p$-adic \'etale cohomology of "almost proper" adic spaces  and the compactly supported cohomology that they use is the naive one, which is equal to the continuous  one in their setting because their torsion cohomology groups are finite (hence satisfy the Mittag-Leffler condition).

Note that one could use also the continuous compactly supported cohomology in the $\ell$-adic case, $\ell\neq p$,
 instead of
Huber's version.  One would get continuous generalized Steinberg representations instead of smooth ones
in Theorem~\ref{main-Orlik}, which would fit better with the objects appearing in the $p$-adic Langlands program
such as Emerton's completed cohomology.  That would also make a (topological) Poincar\'e duality possible for
the spaces that we consider. 
\end{remark}

\begin{remark}
Moreover: 
\begin{enumerate}
\item We expect that the hypothesis $p \neq 2$ in Theorem \ref{main0} is not needed.
This hypothesis is made so that we can use Theorem \ref{vanishing-intro} below.
When $G=\mathbb{GL}_{d+1,\Q_p}$, Theorem \ref{main0} holds true for $p=2$ as well (see Remark \ref{vanishing-GLn} below).

\item Let ${\mathbb H}^d_{\Q_p}$ be  the Drinfeld symmetric space of dimension 
  $d$ over $\Q_p$.     Recall that 
  ${\mathbb H}^d_{\Q_p}=\mathbb{P}^{d}_{\Q_p}\setminus \cup_{H\in \mathcal{H}} H$, where $\mathcal{H}$ is the set of $\Q_p$-rational hyperplanes. Set $G:=\mathbb{GL}_{d+1,\Q_p}$.
Theorem \ref{main0} yields  an isomorphism of 
           $\sg_{\Q_p}\times G(\qp)$-modules 
          \begin{equation*}
 H^i_{\eet,c}({\mathbb H}^d_{C},\Z/p^n)\simeq {\rm Sp}_{2d-i}(\Z/p^n)(d-i),
 \end{equation*}
where the generalized Steinberg representations ${\rm Sp}_{j}(\Z/p^n)$  are  as defined in Section \ref{main-ex}.
  Comparing this isomorphism with that of \cite{CDN4}:
$$
 H^i_{\eet}({\mathbb H}^d_{C},\Z/p^n)\simeq{\rm Sp}_{i}(\Z/p^n)^*(-i)
$$
 one finds an abstract
duality of ${\rm GL}_{d+1}(\Q_p)\times\sg_{\Q_p}$-representations:
$$
H^i_{\eet,c}({\mathbb H}^d_{C},\Z/p^n)(d)\simeq H^{2d-i}_{\eet}({\mathbb H}^d_{C},\Z/p^n)^{*}.
$$
It seems likely that this abstract 
duality is induced by the cup-product with values
in $H^{2d}_{\eet,c}({\mathbb H}^d_{C},\Z/p^n(d))\simeq \Z/p^n$ but we did not
verify this.

 This suggests that 
Poincar\'e duality holds for $\sff^{\rm wa}_C$ and that one can
deduce from Theorem \ref{main0}
a description of the \'etale cohomology $H^*_{\eet}(\sff^{\rm wa}_C, \Z/p^n)$
as $J(\Q_p)\times \sg_{\Q_p}$-modules. 

\item Suppose moreover that $\mu$ is minuscule. Thanks to the work of Fargues-Fontaine \cite{FF0}, Kedlaya-Liu \cite{KL}, and Scholze \cite{SW}, 
we can define the {\it admissible locus} $\sff^{a}\subset \sff^{\rm wa}$, a partially proper open subset of 
$\sff$ having the same classical points as $\sff^{\rm wa}$, and a $p$-adic local system over it  interpolating the Galois representations associated to these classical points by the theorem of Colmez-Fontaine. 
In some remarkable situations (which can be completely classified thanks to the work of Chen-Fargues-Shen \cite{CFS} and Goertz-He-Nie \cite{GHN}) we have 
$\sff^{\rm wa}=\sff^a$ and so the above theorem describes the $p$-adic \'etale cohomology with compact support of the admissible locus. For instance, this is the case for the quasi-split group $G={\rm SO}(V,q)$, where $V=\Q_p^n$ endowed with the quadratic form 
$q(x_1,\ldots,x_n)=x_1x_n+x_2x_{n-1}+\cdots+x_nx_1$, the minuscule cocharacter $\mu(z)={\rm diag}(z,1,\ldots,1, z^{-1})$, 
and the basic class $[b]=[1]\in B(G,\mu)$, for which $J=G$. The flag variety is then the quadric $\sff$ over $\breve{\Q}_p$ with equation 
$q(x)=0$ in projective space and we have $\sff^{\rm wa}=\sff^{a}=\sff\setminus G(\Q_p) S$, where 
$S$ is the Schubert variety with equations $x_{\lceil n/2\rceil+1}=\cdots=x_n=0$ inside $\sff$ (we learnt this example from Fargues). 
For $n=21$, we obtain a very concrete description of the $p$-adic period domain for polarized K3 surfaces with supersingular reduction and the previous theorem yields its $p$-adic \'etale cohomology with compact support. 
In general, we do not know how to describe the $\ell$-adic \'etale 
cohomology (with compact support) of $\sff^a$, even for $\ell\ne p$. 

\end{enumerate}
\end{remark}

\subsection{The proof of the main result} 
We will sketch the proof of Theorem \ref{main0} in the torsion case; the continuous case follows by taking limits. 

\subsubsection{The geometric part} As we have already mentioned, the geometric part of the proof (which, however, does use representation theory as well) is analogous to Orlik's proof of the corresponding 
    result with $\ell$-torsion coefficients, for $\ell\neq p$. Our contribution here lies solely in the verification that all $\ell$-torsion statements in Orlik's proof work in the $p$-torsion setting as well. That this was  not guaranteed is shown by the fact that it fails in the $\ell$-adic setting: Orlik's $\ell$-adic proof for $\ell\neq p$ breaks down $p$-adically (as we have already seen in Remark \ref{pathology},  Huber's compactly supported 
    $l$-adic cohomology behaves rather badly for $l=p$, while it behaves as expected for $l\ne p$, and this is crucial for Orlik's argument to work). 
    
    The argument goes as follows. 
    One starts with the distinguished triangle
(associated to the triple $(\sff^{\rm wa},\sff,\partial\sff^{\rm wa})$, $\partial\sff^{\rm wa}:=\sff\setminus \sff^{\rm wa}$)
 \begin{equation*}
  \rg_{\eet,c}(\sff^{\rm wa}_C,\Z/p^n)\longrightarrow \rg_{\eet}(\sff_C,\Z/p^n)\longrightarrow \rg_{\eet}(\partial\sff^{\rm wa}_C,\Z/p^n).
  \end{equation*}
This reduces the computation of $H^*_{\eet,c}(\sff^{\rm wa}_C,\Z/p^n)$  to that of $H^*_{\eet}(\partial\sff^{\rm wa}_C,\Z/p^n)$ and the boundary map 
$\partial: H^*_{\eet}(\partial\sff^{\rm wa}_C,\Z/p^n)\to   H^{*+1}_{\eet,c}(\sff^{\rm wa}_C,\Z/p^n)$: one needs to prove an isomorphism
(we omit the coefficients $\Z/p^n$ in the formula):
 \begin{align}
 \label{reduction11}
 H^*_{\eet}(\partial\sff^{\rm wa}_C)\simeq\Bigg\{
\xymatrix{*\txt{ 
$\bigoplus_{|\Delta\setminus I_{[w]}|=1}\big(i^J_{{P_{[w]}}}\otimes\rho_{[w]} [-2l_{[w]}]\big )$\\ $\bigoplus$ \\
 $\bigoplus_{|\Delta\setminus I_{[w]}|>1}\big(\rho_{[w]}[-2l_{[w]}]
\oplus \big(v^J_{{P_{[w]}}}\otimes\rho_{[w]}[-2l_{[w]}-|\Delta\setminus I_{[w]}|+1]\big) \big)$}}
\end{align}
To do it, one stratifies the complement $\partial\sff^{\rm wa}$ by Schubert varieties whose cohomology is easy to compute.  More precisely, one uses  the Faltings and Totaro  description of weak admissibility as a semistability condition: the period domain $\sff^{\rm wa}$ is the locus of semistability in $\sff$ and the complement $\partial\sff^{\rm wa}$  is the locus in $\sff$, where semistability fails. 
To test semistability one   applies   the Hilbert-Mumford criterion: for a field extension $K/\check{E}$, $x\in \sff(K)$  is semistable (hence $x\in \sff^{\rm wa}(K)$) if and only if $\mu(x,\lambda)\geq 0$, for all $\lambda\in X_*(J)^{\sg_F}$. Here $\mu(-,-)$ is the slope function associated to a  linearization  of the action of $J$.   
    
      The slope function, a priori convex on  each chamber of the spherical building $\sbb(J_{\rm der})$, is actually affine. This  implies that, in the Hilbert-Mumford criterion, it is enough to test the 1-parameter subgroups
 associated to the relative simple roots and their conjugates.  This leads to the stratification $$ \partial\sff^{\rm wa}=Z_1\supset \cdots\supset Z_{i-1}\supset Z_i\supset Z_{i+1}\supset \cdots$$
      that is defined in the following way. For $\lambda\in X_*(J)_{\Q},$ let $Y_{\lambda}$ be   the locus in $\sff$, where $\lambda$ damages the semistability condition. For $I\subset \Delta$, let $Y_I:=\cap_{\alpha\notin I}Y_{\omega_{\alpha}}$ be the associated Schubert variety. Then the locus $Z_i$ of $\partial\sff^{\rm wa}$, where the semistability fails to the degree at least $i$, 
   can be described as 
   $$Z_i=\bigcup_{|\Delta\setminus I|=i} Z_{I},\quad 
Z_{I}:=J(\Q_p)\cdot Y^{\rm ad}_I.$$
We note that  $Z_I$ is a closed pseudo-adic subspace of $\partial\sff^{\rm wa}$.  In particular, so is $\partial\sff^{\rm wa}=Z_1=\bigcup_{|\Delta\setminus I|=1} Z_{I}.$

 Having this stratification, by a procedure akin to a closed Mayer-Vietoris,  one obtains an {\em acyclic} complex
   of sheaves on $\partial\sff^{\rm wa}_C$, called {\em the fundamental complex}, 
   \begin{equation*}
            0\to \Z/p^n\to \bigoplus_{|\Delta\setminus I|=1} (\Z/p^n)_I\to \bigoplus_{|\Delta\setminus I|=2} (\Z/p^n)_I\to\cdots
              \to \bigoplus_{|\Delta\setminus I|=|\Delta|-1} (\Z/p^n)_I\to (\Z/p^n)_{\emptyset}\to 0,
              \end{equation*} 
             where $(\Z/p^n)_I$ denotes  the constant sheaf $\Z/p^n$ evaluated\footnote{We simplify for the sake of the introduction; see Section \ref{strat} for details.} on  $Z_{I,C}$.
        This complex yields      
             a spectral sequence 
              \begin{equation}
              \label{CDG11}
 E_1^{i,j}=\bigoplus_{|\Delta\setminus I|=i+1} H^i_{\eet}(\partial \sff^{\rm wa}_C, (\Z/p^n)_I)
 \Rightarrow H^{i+j}_{\eet}(\partial \sff^{\rm wa}_{C},\Z/p^n).
 \end{equation}
Using the fact that $P_I(\Q_p)$ is the stabilizer of $Y_I$
in $J(\Q_p)$ and $X_I=J(\Q_p)/P_I(\Q_p)$ one computes that
            \begin{align*}
            H^i_{\eet}(\partial \sff^{\rm wa}_C, (\Z/p^n)_I) & \simeq \LC(X_I, H^i_{\eet}(Y_{I,C},\Z/p^n))\simeq i^J_{P_I}(\Z/p^n)\otimes H^i_{\eet}(Y_{I,C},\Z/p^n)\\
             & \simeq  i^J_{P_I}(\Z/p^n)\otimes (\bigoplus_{[w]\in\Omega_I}\rho_{[w]}(\Z/p^n)[-2l_{[w]}]).
             \end{align*}
          Here $\Omega_I$ is a subset of $W^{\mu}/\sg_{E_s}$ (see Section \ref{Schubert1}).   The third isomorphism is obtained by the classical 
               computation of the cohomology of Schubert varieties. Via a simple Galois-theoretic weight argument, this computation implies that the above spectral sequence  degenerates at $E_2$.   Using results of 
 Grosse-Kl\"onne \cite{GK}, Herzig \cite{Her}, and Ly \cite{LySt} on generalized Steinberg representations mod $p$, one can also compute the 
 $E_2$ terms: they are equal to the terms on the right hand side of the formula (\ref{reduction11}).
 
\subsubsection{The group-theoretic part}

It follows from the above section that the grading of $H^{*}_{\eet}(\partial \sff^{\rm wa}_{C},\Z/p^n)$ associated to the filtration induced by the spectral sequence \eqref{CDG11} is isomorphic to the right hand side of \eqref{reduction11}.
It remains to show that this filtration splits.
And this is where things get much harder for $p$-torsion coefficients than for the $\ell$-torsion ones.
Splitting this filtration essentially comes down to understanding $\Ext$ groups between generalized Steinberg representations with $p$-torsion coefficients.
Fortunately, it suffices to deal with $\Ext^1$'s, which are the only ones we can handle, contrary to the usual theory with complex coefficients (adapted to the $\ell$-adic setting by Orlik \cite{OrlikExt} and Dat \cite{Dat}).
It is indeed a well-known phenomenon in the theory of smooth mod $p$ representations of $p$-adic reductive groups that $\Ext$ groups can be very hard to compute, since most of the techniques for complex or $\ell$-adic coefficients fail.

Before stating the key result that allows us to split the filtration, let us briefly explain the argument for complex or $\ell$-torsion coefficients and point out the difficulties occurring for $p$-torsion coefficients.
Let $R$ be one of the rings $\mathbf{C}, \Z/\ell^n, \Z/p^n$ ($\ell$ being sufficiently generic with respect to $J$).
One can construct an acyclic complex
\begin{equation*}
0 \to i_{\Delta}^J(R) \to \bigoplus_{\substack{I \subset I' \subset \Delta \\ |\Delta \setminus I'|=1}} i_{P_{I'}}^J(R) \to \cdots \to \bigoplus_{\substack{I \subset I' \subset \Delta \\ |I' \setminus I| = 1}} i_{P_{I'}}^J \to i_{P_I}^J(R) \to v_{P_I}^J(R) \to 0.
\end{equation*}
For $R = \mathbf{C}$ or $\Z/\ell^n$ this is a rather standard result, and it also works for $\Z/p^n$ thanks to the above-mentioned work of Grosse-Kl\"onne, Herzig, and Ly (the acyclicity of this complex is also crucial in computing the $E_2$ terms of the above spectral sequence).
Suppose that $R\ne \Z/p^n$.
A spectral sequence argument reduces the computation of $\Ext^*_{J(\Q_p)}(v_{P_I}^J(R), v_{P_{I'}}^J(R))$ to the computation of $\Ext^*_{J(\Q_p)}(i_{P_I}^J(R), i_{P_{I'}}^J(R))$ for all $I, I' \subset \Delta$.
The exactness of the Jacquet functor (which fails when $R=\Z/p^n$) reduces the problem to understanding extensions between the Jacquet module of $i_{P_I}^J(R)$ (which can be understood by the Bernstein-Zelevinsky geometric lemma) and the trivial representation.
After several other relatively standard but technical arguments one reduces everything to the computation of $H^*(J(\Q_p), i_{P_I}^J(R)) = H^*(M_I(\Q_p), R)$, where $M_I$ is the Levi quotient of the standard parabolic $P_I$.
Thus we are reduced to computing $H^*(G(\Q_p), R)$ for a reductive group $G$ over $\Q_p$, which can be done using the contractibility of the Bruhat-Tits building of $G$ \emph{and} the fact that $i_K^G(1)$ is injective as smooth representation whenever $K$ is a compact open subgroup of $G$ (since passage to $K$-coinvariants is exact).
This again fails when $R = \Z/p^n$.
Actually, it is an interesting problem to compute $H^*(G(\Q_p), \Z/p^n)$ for a reductive group $G$ over $\Q_p$.
Unfortunately we don't have much to say about this except to mention that the computation of $H^*(\GL_n(\Z_p), \Z/p)$ seems rather complicated: Lazard's theory allows one to compute $H^*(1+p\mathrm{M}_n(\Z_p), \Z/p)$ (at least when $1+p\mathrm{M}_n(\Z_p)$ is a uniform pro-$p$ group, i.e.\ $p \neq 2$), so one is reduced to the computation of $H^*(\GL_n(\Z/p), \Z/p)$, a well-known open problem.

The previous paragraph makes it clear that a new idea is needed in order to compute extensions between generalized Steinberg representations modulo $p$.
We will only focus on the computation of $\Ext^1$'s, which is enough for our needs.
All $\Ext$ groups below are computed in the category of smooth $J(\Q_p)$-representations with coefficients in $\Z/p^n$.
The key result needed to split the abutment filtration of the spectral sequence (\ref{CDG11}) is then:

\begin{theorem} \label{vanishing-intro}
Assume that $p \geq 5$ and let $I, I' \subset \Delta$.
If $|(I \cup I') \setminus (I \cap I')| \geq 2$, then
\begin{equation*}
\Ext_{J(\Q_p)}^1(v_{P_I}^J(\Z/p^n), v_{P_{I'}}^J(\Z/p^n)) = 0.
\end{equation*}
The result holds true when $p = 3$ under the stronger assumption $||I| - |I'|| \geq 2$.
\end{theorem}

\begin{remark} \label{vanishing-GLn}
The result also holds true when $p = 2$ and $J$ is an inner form of $\mathbb{GL}_{d+1,\Q_p}$ under the stronger assumption $||I| - |I'|| \geq 2$.
\end{remark}

We refer the reader to Section \ref{stratege} for an overview of the rather technical proof of the theorem.
(Actually we compute the $\Ext^1$'s between generalized Steinberg representations, with coefficients in an artinian commutative ring in which $p$ is nilpotent, for any reductive group over a local field of residue characteristic $p$.)
The most difficult part is to compute $H^1(G(\Q_p), \St)$, where $\St$ denotes the ordinary Steinberg representation, for a reductive group $G$ over $\Q_p$ and one can actually reduce the theorem to this computation by rather painful devissage arguments involving Emerton's ordinary parts functor and its derived functors.
The computation of higher $\Ext$ groups seems much more involved: extending our method would require at least to compute $H^*(G(\Q_p), \St)$ and to prove a conjecture of Emerton on the higher ordinary parts functors (see \cite[Conj.\ 3.7.2]{Em2}).

\begin{acknowledgments}
We would like to thank Sascha Orlik for patiently explaining to us the details of his work.
We also thank Laurent Fargues for helpful discussions concerning the content of this paper.
G.D. would like to thank Shanwen Wang and the Fudan University, where parts of the paper were written, for the wonderful working conditions.
\end{acknowledgments}

\section{Extensions between generalized Steinberg representations} \label{extensions}

Let $F$ be a local field of residue characteristic $p$, $G$ be the group of $F$-points of a connected reductive algebraic $F$-group, and $R$ be an artinian commutative ring in which $p$ is nilpotent.
We compute the $\Ext^1$ groups between generalized Steinberg representations in the category of smooth $G$-representations with coefficients in $R$.
In particular, we will prove Theorem \ref{vanishing-intro} from the introduction.

\subsection{Notation}

Let us fix the notation for this section.
We fix a separable closure $\Fbar$ of $F$ and let $\sg_F = \Gal(\Fbar/F)$.
Let $\varepsilon : \Q_p^* \to \Z_p^*$ denote the $p$-adic cyclotomic character and let $\varepsbar : \Q_p^* \to \F_p^*$ denote its reduction mod $p$.
We let $D$ denote a finite-dimensional central division algebra over $F$.

\subsubsection{Linear algebraic $F$-groups}

A linear algebraic $F$-group will be written with a boldface letter like $\Hb$ and its group of $F$-points will be denoted by the corresponding ordinary letter $H = \Hb(F)$.
We will write $H^1(F, \Hb)$ for the set of isomorphism classes of $\Hb$-torsors over $F$.
If $\Hb$ is smooth (this is always the case when $\car(F) = 0$ by Cartier's theorem), then $H^1(F, \Hb)$ is isomorphic to the Galois cohomology group $H^1(\sg_F, \Hb(\Fbar))$.
We will write $\Zb_{\Hb}$ for the center of $\Hb$.

Let $\Gb$ be a connected reductive algebraic $F$-group.
We fix a maximal split torus $\Sb \subset \Gb$ and a minimal parabolic subgroup $\Bb \subset \Gb$ containing $\Sb$.
Let $\Zcb$ be the centralizer of $\Sb$ in $\Gb$, that is the Levi factor of $\Bb$ containing $\Sb$, and $\Ub$ be the unipotent radical of $\Bb$, so that $\Bb = \Zcb \Ub$.
We write $\Bbbar = \Zcb \Ubbar$ for the opposite minimal parabolic subgroup.
Let $\Ncb$ be the normalizer of $\Sb$ in $\Gb$ and let
\begin{equation*}
W = \Ncb / \Zcb = \Nc / \Zc
\end{equation*}
be the relative Weyl group of $\Gb$.
For $w \in W$, we let
\begin{equation*}
\Ubbar_w = \Ubbar \cap w^{-1} \Ubbar w \quad \text{and} \quad \Bbbar_w = \Zcb \Ubbar_w.
\end{equation*}

Let $X^*(\Sb)$ be the group of characters of $\Sb$, let $X_*(\Sb)$ be the group of cocharacters of $\Sb$, and let $\langle -, - \rangle : X^*(\Sb) \times X_*(\Sb) \to \Z$ denote the natural pairing.
Let
\begin{equation*}
\Phi \supset \Phi^+ \supset \Delta
\end{equation*}
be the subsets of relative roots, positive roots, simple roots in $X^*(\Sb)$.
We let $\Phi^- = -\Phi^+ = \Phi \setminus \Phi^+$.
For $\alpha \in \Delta$, we let $\alpha^\vee \in X_*(\Sb)$ be the corresponding coroot, $s_\alpha \in W$ be the corresponding simple reflection, and $\Ub_\alpha \subset \Ub$ be the corresponding root subgroup.
If $\Ub_\alpha$ is one-dimensional, then $\alpha$ extends to a character of $\Zcb$ which will be denoted $\alphat$.

For $I \subset \Delta$, we let $\Pb_I = \Mb_I \Nb_I$ be the corresponding parabolic subgroup of $\Gb$ containing $\Bb$, where $\Mb_I$ is the Levi factor of $\Pb_I$ containing $\Sb$ and $\Nb_I$ is the unipotent radical of $\Pb_I$, and we let $\Bb_I = \Mb_I \cap \Bb$ and $\Ub_I = \Mb_I \cap \Ub$, so that $\Bb_I = \Zcb \Ub_I$, and we let $\Zb_I$ denote the center of $\Mb_I$.\footnote{
In the case $I = \emptyset$ we have $\Pb_I = \Bb$, $\Mb_I = \Zcb$, $\Nb_I = \Ub$, $\Bb_I = \Zcb$, $\Ub_I = 1$, and $\Zb_I = \Zb_{\Zcb}$.
In the case $I = \Delta$ we have $\Pb_I = \Gb$, $\Mb_I = \Gb$, $\Nb_I = 1$, $\Bb_I = \Bb$, $\Ub_I = \Ub$, and $\Zb_I = \Zb_{\Gb}$.
}
We write $\Pbbar_I = \Mb_I \Nbbar_I$ and $\Bbbar_I = \Zcb \Ubbar_I$ for the opposite parabolic subgroups.
We let $\Ncb_I$ be the normalizer of $\Sb$ in $\Mb_I$ and we let
\begin{equation*}
W_I = \Ncb_I / \Zcb = \Nc_I / \Zc
\end{equation*}
be the relative Weyl group of $\Mb_I$.
We let $\Wt_I \subset W$ be the set of representatives of minimal length of the right cosets $W_I \backslash W$, we let $w_{I, 0} \in W_I$ be the longest element, and we let
\begin{equation*}
\Wh_I = w_{I, 0} \Wt_I \setminus \bigcup_{J \supsetneq I} w_{J, 0} \Wt_J.
\end{equation*}
For $w_I \in W_I$, we let
\begin{equation*}
\Ubbar_{I, w_I} = \Ubbar_I \cap w_I^{-1} \Ubbar_I w_I \quad \text{and} \quad \Bbbar_{I, w_I} = \Zcb \Ubbar_{I, w_I}
\end{equation*}
When $I = \{\alpha\}$ (resp.\ $I = \Delta \setminus \{\alpha\}$), we rather write $\Pb_\alpha$, $\Mb_\alpha$, $\Nb_\alpha$, etc.\ (resp.\ $\Pb^\alpha$, $\Mb^\alpha$, $\Nb^\alpha$, etc.).

\subsubsection{Smooth representations}

All representations will be smooth with coefficients in $R$ and all maps between $R$-modules will be $R$-linear.

Given a locally profinite space $X$, we let $\LC(X)$ be the $R$-module of locally constant functions on $X$ with coefficients in $R$,
we let $\supp(f) = X \setminus f^{-1}(\{0\})$ denote the (open and closed) support of a function $f \in \LC(X)$,
and we let $\LCc(X) \subset \LC(X)$ be the $R$-submodule consisting of those functions with compact support.

Given a closed subgroup $H$ of $G$ and a smooth $H$-representation $\sigma$, we define a smooth $G$-representation by letting $G$ act by right translation on the $R$-module
\begin{equation*}
\Ind_H^G(\sigma) = \{ f : G \to \sigma \mid \text{$\exists K_f \subset G$ open subgroup s.t.\ $f(hgk) = h \cdot f(g)$ $\forall h \in H, g \in G, k \in K_f$} \}.
\end{equation*}
Let $1$ be the trivial representation of any locally profinite group.
For any subset $I \subset \Delta$, let
\begin{equation*}
i_{P_I}^G = \Ind_{P_I}^G(1) \simeq \LC(P_I \backslash G).
\end{equation*}
If $J \supset I$ is another subset, then there is an injection $i_{P_J}^G \hookrightarrow i_{P_I}^G$ which is induced by the natural surjection $P_I \backslash G \twoheadrightarrow P_J \backslash G$.
The generalized Steinberg representation with respect to $I$ is the quotient
\begin{equation*}
v_{P_I}^G = i_{P_I}^G / \sum_{J \supsetneq I} i_{P_J}^G.
\end{equation*}
In the case $I = \emptyset$ we obtain the ordinary Steinberg representation denoted $\St$.
In the case $I = \Delta$ we obtain the trivial representation $1$.

Given a closed subgroup $\Ubar{}' \subset \Ubar$ stable under conjugation by $\Zc$, we endow the $R$-module $\LCc(\Ubar{}')$ with a smooth action of the group $\Bbar{}' = \Zc\Ubar{}'$ defined by
\begin{equation*}
(z \ubar \cdot f)(\ubar') = f(z^{-1} \ubar' z \ubar)
\end{equation*}
for $z \in \Zc$, $\ubar, \ubar' \in \Ubar{}'$, and $f \in \LCc(\Ubar{}')$.

\subsection{The main results}

We prove the following result (cf.\ Theorem \ref{vanishing-intro} in the introduction).

\begin{theorem} \label{theo:vanishing}
Assume that $p \geq 5$ and let $I, J \subset \Delta$.
If $|(I \cup J) \setminus (I \cap J)| \geq 2$, then $\Ext_G^1(v_{P_I}^G, v_{P_J}^G) = 0$.
The result holds true when $p = 3$ under the stronger assumption $||I| - |J|| \geq 2$.
\end{theorem}

\begin{remark}
We expect Theorem \ref{theo:vanishing} to hold true for all $p$.
Actually, we prove the result in almost all cases when $p = 3$ and in some cases when $p = 2$.
Moreover, it follows from our computations that the above $\Ext^1$ is $3$-torsion when $p = 3$ and $8$-torsion when $p = 2$ (see Remarks \ref{rema:H1Ord} and \ref{rema:base}).
In particular, if $\car(F) = 0$ and $E / F$ is a finite extension, then the analogous $\Ext^1$ in the category of admissible unitary continuous $G$-representations on $E$-Banach spaces vanishes for all $p$ (see \cite[Prop.\ 5.3.1]{JHD} and \cite[Lemme 3.3.3]{JHC}).
\end{remark}

When $I = J$, the $R$-module $\Ext_G^1(v_{P_I}^G, v_{P_J}^G)$ has been computed in \cite{HSS2} without any assumption (see Proposition 8 in loc.\ cit.).
In the course of the proof of Theorem \ref{theo:vanishing}, we also treat the case $J = I \sqcup \{\alpha\}$ under a very mild assumption (which is always satisfied if $p \geq 5$) and we reduce the remaining case $I = J \sqcup \{\alpha\}$ to the special case where $\Delta = \{\alpha\}$.
We treat the latter case under some assumption on $G$.
In particular, we obtain the following result.

\begin{theorem} \label{theo:GLn}
Assume that $G = \GL_n(D)$ with $D \neq \Q_2$.
Let $I, J \subset \Delta$.
\begin{enumerate}
\item If $J = I \sqcup \{\alpha\}$, then the $R$-module $\Ext_G^1(v_{P_I}^G, v_{P_J}^G)$ is free of rank $1$.
\item If $I = J \sqcup \{\alpha\}$, then there is an $R$-linear isomorphism
\begin{equation*}
\Ext_G^1(v_{P_I}^G, v_{P_J}^G) \simeq \Hom(F^*, R).
\end{equation*}
\item If $|(I \cup J) \setminus (I \cap J)| \geq 2$, then $\Ext_G^1(v_{P_I}^G, v_{P_J}^G) = 0$.
\end{enumerate}
\end{theorem}

In contrast, we do not know how to compute the $R$-module $\Ext_G^1(\St, 1)$ when $G = \GL_2(\Q_2)$, or $\Ext_G^1(v_{P_1}^G, v_{P_2}^G)$ when $G = \GL_3(\Q_2)$ and $P_1, P_2$ denote the two maximal proper standard parabolic subgroups of $G$.

\begin{remark}
A locally analytic version of part (2) of Theorem \ref{theo:GLn} (when $D = F$ and $\car(F) = 0$) is established in the work of Ding \cite{Ding} and generalized to split reductive groups by Gehrmann \cite{Lennart}.
Higher $\Ext$ groups are computed by Orlik and Strauch \cite{Orlik-Strauch}, for split reductive groups and in a suitable category of locally analytic representations (but not in the category of admissible locally analytic representations).
We note that a vanishing result for $\Ext^1$ in the locally analytic world gives a corresponding vanishing result in the context of admissible Banach representations, since the continuous generalized Steinberg representations are the universal unitary completions of their locally analytic vectors.
We thank Lennart Gehrmann for pointing out the references above.
Let us mention though that the vanishing mod $p$ in Theorem \ref{vanishing-intro} is crucial for the proof of Theorem \ref{main0}: the corresponding result for Banach representations is not sufficient for our needs.
\end{remark}

\subsection{The proof of the main results} \label{stratege}

We fix two subsets $I, J \subset \Delta$.
First, we recall the computation of the $R$-module $\Hom_G(v_{P_I}^G, v_{P_J}^G)$.

\begin{proposition}[Grosse-Kl\"onne, Herzig, Ly] \label{prop:HomSt}
There is an $R$-linear isomorphism
\begin{equation*}
\Hom_G(v_{P_I}^G, v_{P_J}^G) \simeq
\begin{cases}
R &\text{if $I = J$,} \\
0 &\text{otherwise.}
\end{cases}
\end{equation*}
\end{proposition}

\begin{proof}
If $I = J$, then the result is a special case of \cite[Cor.\ 5]{HSS2}.
If $I \neq J$, then by devissage the result reduces to the case where $R$ is a field of characteristic $p$, which is proved by Grosse-Kl\"onne \cite{GK}, Herzig \cite{Her}, and Ly \cite{LySt}.
\end{proof}

Now, using the results of the next sections, we compute the $R$-module $\Ext_G^1(v_{P_I}^G, v_{P_J}^G)$ when $I \neq J$.
In particular, we prove Theorems \ref{theo:vanishing} and \ref{theo:GLn}.
We treat the cases $J \not \subset I$ and $J \subsetneq I$ separately.

\subsubsection{The case $J \not \subset I$}

In this case we can compute directly the $R$-module $\Ext_G^1(v_{P_I}^G, v_{P_J}^G)$ using the results of Section \ref{reduction}.
Let $\alpha \in (\Delta \setminus I) \cap J$.
If $F = \Q_p$, $\dim \Ub_\alpha = 1$, and $J \setminus \{\alpha\} = I \cap \{\alpha\}^\perp$ assume that $\varepsbar \circ \alphat \neq 1$.
This assumption is satisfied in the following cases: $p \geq 5$; $p = 3$ and $||I| - |J|| \geq 2$; $p = 2$, $||I| - |J|| \geq 2$, and $G = \GL_n(D)$ (see Remark \ref{rema:assumption}).
Then there is an $R$-linear isomorphism (see \eqref{isoExt})
\begin{equation*}
\Ext_G^1(v_{P_I}^G, v_{P_J}^G) \simeq
\begin{cases}
R & \text{if $J = I \sqcup \{\alpha\}$,} \\
0 & \text{otherwise.}
\end{cases}
\end{equation*}
In the case $J = I \sqcup \{\alpha\}$ the $R$-module $\Ext_G^1(v_{P_I}^G, v_{P_J}^G)$ is generated by the class of $\Ind_{P^\alpha}^G(v_{M^\alpha \cap P_I}^{M^\alpha})$ (see \eqref{dev}).
This proves Theorem \ref{theo:vanishing} and part (3) of Theorem \ref{theo:GLn} in the case $J \not \subset I$, as well as part (1) of Theorem \ref{theo:GLn}.

\subsubsection{The case $J \subsetneq I$}

In this case we can reduce the computation of the $R$-module $\Ext_G^1(v_{P_I}^G, v_{P_J}^G)$ to the computation of the $R$-modules $\Ext_{M_K}^1(1, \St_{M_K})$ for certain subsets $K \subset \Delta$, where $\St_{M_K}$ denotes the Steinberg representation of $M_K$, using the results of Section \ref{reduction}:
for all $\alpha \in \Delta \setminus I$ there is an exact sequence of $R$-modules (see \eqref{redI})
\begin{equation*}
0 \to \Ext_G^1(v_{P_I}^G, v_{P_J}^G) \to \Ext_{M^\alpha}^1(v_{M^\alpha \cap P_I}^{M^\alpha}, v_{M^\alpha \cap P_J}^{M^\alpha}) \to \Ext_G^1(v_{P_{I \sqcup \{\alpha\}}}^G, v_{P_J}^G),
\end{equation*}
and for all $\alpha \in J$ there is an exact sequence of $R$-modules (see \eqref{redJ})
\begin{equation*}
0 \to \Ext_G^1(1, v_{P_J}^G) \to \Ext_{M^\alpha}^1(1, v_{M^\alpha \cap P_{J \setminus \{\alpha\}}}^{M^\alpha}) \to \Ext_G^1(1, v_{P_{J \setminus \{\alpha\}}}^G).
\end{equation*}
Using these exact sequences recursively (i.e.\ for $\Gb$ and its Levi subgroups $\Mb_K$) yields an $R$-linear injection
\begin{equation} \label{Ext1inj}
\Ext_G^1(v_{P_I}^G, v_{P_J}^G) \hookrightarrow \Ext_{M_{I \setminus J}}^1(1, \St_{M_{I \setminus J}}).
\end{equation}
Moreover, in the case $I = J \sqcup \{\alpha\}$ the above injection is an $R$-linear isomorphism
\begin{equation} \label{Ext1iso}
\Ext_G^1(v_{P_I}^G, v_{P_J}^G) \iso \Ext_{M_\alpha}^1(1, \St_{M_\alpha})
\end{equation}
provided that $\Ext_{M_{\{\alpha, \beta\}}}^1(1, \St_{M_{\{\alpha, \beta\}}}) = 0$ for all $\beta \in \Delta \setminus \{\alpha\}$.

In Section \ref{conclusion}, using the results of Sections \ref{isogenies}, \ref{comparison}, and \ref{vanishing}, we compute the $R$-module $\Ext_G^1(1, \St)$.
Replacing $\Gb$ by its Levi subgroups $\Mb_K$ we obtain the following results.
Assume that $p \neq 2$.
\begin{itemize}
\item If $K = \{\alpha\}$ and the adjoint action of $\Zc$ on $\Ubar_\alpha \setminus \{1\}$ is transitive, then there is an $R$-linear isomorphism (see Proposition \ref{prop:basealpha})
\begin{equation} \label{basealpha}
\Ext_{M_\alpha}^1(1, \St_{M_\alpha}) \iso \Ext_{\Zc}^1(1, 1)^{s_\alpha = -1}.
\end{equation}
(Here $\Ext_{\Zc}^1(1, 1)^{s_\alpha = -1}$ denotes the $R$-submodule of $\Ext_{\Zc}^1(1, 1)$ consisting of those extensions on which $s_\alpha$ acts by multiplication by $-1$.)
\item If $|K| \geq 2$, then $\Ext_{M_K}^1(1, \St_{M_K}) = 0$ (see Proposition \ref{prop:base}).
\end{itemize}
When $G = \GL_n(D)$, these results hold true for $p = 2$ as well, and \eqref{basealpha} induces an $R$-linear isomorphism
\begin{equation*}
\Ext_{M_\alpha}^1(1, \St_{M_\alpha}) \simeq \Hom(F^*, R).
\end{equation*}
(see Remarks \ref{rema:basealpha} and \ref{rema:base}).

Finally, we can compute the $R$-module $\Ext_G^1(v_{P_I}^G, v_{P_J}^G)$.
Assume that either $p \neq 2$ or $G = \GL_n(D)$.
If $I = J \sqcup \{\alpha\}$, then $\Ext_{M_{\{\alpha, \beta\}}}^1(1, \St_{M_{\{\alpha, \beta\}}}) = 0$ for all $\beta \in \Delta \setminus \{\alpha\}$, so that composing \eqref{Ext1iso} and \eqref{basealpha} yields an $R$-linear isomorphism
\begin{equation*}
\Ext_G^1(v_{P_I}^G, v_{P_J}^G) \simeq \Ext_{\Zc}^1(1, 1)^{s_\alpha = -1}.
\end{equation*}
When $G = \GL_n(D)$, this isomorphism induces an $R$-linear isomorphism
\begin{equation*}
\Ext_G^1(v_{P_I}^G, v_{P_J}^G) \simeq \Hom(F^*, R).
\end{equation*}
This proves part (2) of Theorem \ref{theo:GLn}.
If $|I \setminus J| \geq 2$, then $\Ext_{M_{I \setminus J}}^1(1, \St_{M_{I \setminus J}}) = 0$, so that \eqref{Ext1inj} yields $\Ext_G^1(v_{P_I}^G, v_{P_J}^G) = 0$.
This proves Theorem \ref{theo:vanishing} and part (3) of Theorem \ref{theo:GLn} in the case $J \subsetneq I$.

\subsection{Reduction to the case $I = \Delta$ and $J = \emptyset$} \label{reduction}

First we recall some results on the parabolic induction functor and its right adjoint.
The parabolic induction functor $\Ind_{P_I}^G$, from the category of smooth $M_I$-representations to the category of smooth $G$-representations, is exact and preserves admissibility.
When $\car(F) = 0$, Emerton \cite{Em1} constructed a functor $\Ord_{\Pbar_I}$ (the ordinary part), from the category of smooth $G$-representations to the category of smooth $M_I$-representations, which is left-exact, preserves admissibility, and is the right adjoint of $\Ind_{P_I}^G$ when restricted to admissible representations.
His construction was generalized by Vign\'eras \cite{VigAdj} to include the case $\car(F) = p$.
When $\car(F) = 0$, Emerton \cite{Em2} extended the functor $\Ord_{\Pbar_I}$ (which is not exact as soon as $I \neq \Delta$) to a sequence of functors $\HOrd[n]_{\Pbar_I}$ which preserve admissibility and form a $\delta$-functor when restricted to admissible representations.
When $\car(F) = p$, the restriction of the functor $\Ord_{\Pbar_I}$ to admissible representations is exact (this is due to one of us, see \cite[Th.\ 1]{JHE}) and we let $\HOrd[n]_{\Pbar_I} = 0$ for $n \geq 1$.

\subsubsection{Reduction to the case $I = \Delta$}

Assume that $I \neq \Delta$ and let $\alpha \in \Delta \setminus I$.
By exactness and transitivity of parabolic induction, there is a short exact sequence of $G$-representations
\begin{equation} \label{dev}
0 \to v_{P_{I \sqcup \{\alpha\}}}^G \to \Ind_{P^\alpha}^G(v_{M^\alpha \cap P_I}^{M^\alpha}) \to v_{P_I}^G \to 0
\end{equation}
which induces an exact sequence
\begin{align*}
0 & \to \Hom_G(v_{P_I}^G, v_{P_J}^G) \to \Hom_G(\Ind_{P^\alpha}^G(v_{M^\alpha \cap P_I}^{M^\alpha}), v_{P_J}^G) \to \Hom_G(v_{P_{I \sqcup \{\alpha\}}}^G, v_{P_J}^G) \\
& \to \Ext_G^1(v_{P_I}^G, v_{P_J}^G) \to \Ext_G^1(\Ind_{P^\alpha}^G(v_{M^\alpha \cap P_I}^{M^\alpha}), v_{P_J}^G) \to \Ext_G^1(v_{P_{I \sqcup \{\alpha\}}}^G, v_{P_J}^G).
\end{align*}
The adjunction between $\Ind_{P^\alpha}^G$ and $\Ord_{\Pbar^\alpha}$ yields an isomorphism
\begin{equation*}
\Hom_G(\Ind_{P^\alpha}^G(v_{M^\alpha \cap P_I}^{M^\alpha}), v_{P_J}^G) \simeq \Hom_{M^\alpha}(v_{M^\alpha \cap P_I}^{M^\alpha}, \Ord_{\Pbar^\alpha}(v_{P_J}^G)).
\end{equation*}
Moreover, there is a short exact sequence
\begin{align*}
0 & \to \Ext_{M^\alpha}^1(v_{M^\alpha \cap P_I}^{M^\alpha}, \Ord_{\Pbar^\alpha}(v_{P_J}^G)) \to \Ext_G^1(\Ind_{P^\alpha}^G(v_{M^\alpha \cap P_I}^{M^\alpha}), v_{P_J}^G) \\
& \to \Hom_{M^\alpha}(v_{M^\alpha \cap P_I}^{M^\alpha}, \HOrd[1]_{\Pbar^\alpha}(v_{P_J}^G))
\end{align*}
(see \cite[(3.7.6)]{Em2} if $\car(F) = 0$ and \cite[Cor.\ 2]{JHE} if $\car(F) = p$).
By \cite[Th.\ 6.1(ii)]{AHV}, there is an $M^\alpha$-equivariant isomorphism
\begin{equation*}
\Ord_{\Pbar^\alpha}(v_{P_J}^G) \simeq
\begin{cases}
v_{M^\alpha \cap P_J}^{M^\alpha} & \text{if $\alpha \notin J$,}\\
0 & \text{otherwise.}
\end{cases}
\end{equation*}
Using Proposition \ref{prop:HomSt} and Lemma \ref{lemm:H1Ord} below, we obtain the following results.
\begin{itemize}
\item Assume that $\alpha \in J$.
If $F = \Q_p$, $\dim \Ub_\alpha = 1$, and $J \setminus \{\alpha\} = I \cap \{\alpha\}^\perp$ assume moreover that $\varepsbar \circ \alphat \neq 1$.
Then there is an isomorphism
\begin{equation} \label{isoExt}
\Ext_G^1(v_{P_I}^G, v_{P_J}^G) \simeq
\begin{cases}
R & \text{if $J = I \sqcup \{\alpha\}$,} \\
0 & \text{otherwise.}
\end{cases}
\end{equation}
In the case $J = I \sqcup \{\alpha\}$ the $R$-module $\Ext_G^1(v_{P_I}^G, v_{P_J}^G)$ is generated by the class of \eqref{dev}.
\item Assume that $\alpha \notin J$.
Then there is an exact sequence
\begin{equation} \label{redI}
0 \to \Ext_G^1(v_{P_I}^G, v_{P_J}^G) \to \Ext_{M^\alpha}^1(v_{M^\alpha \cap P_I}^{M^\alpha}, v_{M^\alpha \cap P_J}^{M^\alpha}) \to \Ext_G^1(v_{P_{I \sqcup \{\alpha\}}}^G, v_{P_J}^G).
\end{equation}
\end{itemize}

\begin{remark} \label{rema:assumption}
Let us discuss the assumption in the case $\alpha \in J$.
\begin{enumerate}
\item Assume that $F = \Q_p$ and $\dim \Ub_\alpha = 1$.
We have $\varepsbar \circ \alphat \neq 1$ if $p \geq 5$.
More precisely, $\varepsbar \circ \alpha = 1$ if and only if either $p = 3$ and $\alpha \in 2X^*(\Sb)$, in which case the irreducible component of $\Phi$ containing $\alpha$ must be of type $\mathrm{A}_1$ or $\mathrm{C}_n$ ($n \geq 2$) with $\alpha$ being the long root, or $p = 2$.\footnote{
Indeed: if $p \geq 5$ then $\varepsbar \circ \alpha \circ \alpha^\vee = \varepsbar^2 \neq 1$; if $p = 3$ and $\alpha \notin 2X^*(\Sb)$ then there exists $\omega_\alpha \in X_*(\Sb)$ such that $\langle \alpha, \omega_\alpha \rangle = 1$ hence $\varepsbar \circ \alpha \circ \omega_\alpha = \varepsbar \neq 1$; if $p = 3$ and $\alpha \in 2X^*(\Sb)$ then $\varepsbar \circ \alpha = \varepsbar^2 \circ (\frac12 \alpha) = 1$; if $p = 2$ then $\varepsbar = 1$.}
\item Assume that $\alpha \in J$ and $J \setminus \{\alpha\} = I \cap \{\alpha\}^\perp$, i.e.\ $J \setminus I = \{\alpha\}$, $I \setminus J \subset \Delta \setminus \{\alpha\}^\perp$, and $I \cap J \subset \{\alpha\}^\perp$.
Then either $J = I \sqcup \{\alpha\}$ with $I \subset \{\alpha\}^\perp$, or $0 \leq |I| - |J| \leq |\Delta \setminus \{\alpha\}^\perp| - 1$.
In particular, we have $||I| - |J|| \leq 1$ unless the irreducible component of $\Phi$ containing $\alpha$ is of type $\mathrm{D}_n$ ($n \geq 4$), $\Erm_6$, $\Erm_7$, or $\Erm_8$ with $\alpha$ being the root with $3$ neighbors.
\end{enumerate}
\end{remark}

Now we turn to the key computation of this section.

\begin{lemma} \label{lemm:H1Ord}
\begin{enumerate}
\item We have $\HOrd[1]_{\Pbar^\alpha}(v_{P_J}^G) \neq 0$ if and only if $F = \Q_p$, $\dim \Ub_\alpha = 1$, $\alpha \in J$, and $J \setminus \{\alpha\} \subset \{\alpha\}^\perp$.
\item Assume that $F = \Q_p$, $\dim \Ub_\alpha = 1$, $\alpha \in J$, and $J \setminus \{\alpha\} \subset \{\alpha\}^\perp$.
Then
\begin{equation*}
\Hom_{M^\alpha}(v_{M^\alpha \cap P_I}^{M^\alpha}, \HOrd[1]_{\Pbar^\alpha}(v_{P_J}^G)) \neq 0
\end{equation*}
only if $J \setminus \{\alpha\} = I \cap \{\alpha\}^\perp$ and $\varepsbar \circ \alphat = 1$.
\end{enumerate}
\end{lemma}

\begin{proof}
First, we review a filtration on $v_{P_J}^G$.
As explained in \cite[Sec.\ 2.1]{JH}\footnote{
In loc.\ cit.\ $\Gb$ is split but the results extend verbatim to any $\Gb$.
},
the Bruhat decomposition
\begin{equation*}
G = \bigsqcup_{w \in W} B w \Bbar
\end{equation*}
induces a filtration on $i_B^G$ by $\Bbar$-subrepresentations with graded pieces $I_w$ indexed by $W$, and there is a $\Bbar_w$-equivariant isomorphism $I_w \simeq \LCc(\Ubar_w)$ for all $w \in W$.
Using the $G$-equivariant morphisms
\begin{equation*}
i_B^G \hookleftarrow i_{P_J}^G \twoheadrightarrow v_{P_J}^G,
\end{equation*}
the filtration on $i_B^G$ induces a filtration on $v_{P_J}^G$ by $\Bbar$-subrepresentations whose graded pieces are the $\Bbar$-representations $I_{\whw_J}$ for $\whw_J \in \Wh_J$ (see the proof of \cite[Prop.\ 2.3.6]{JHC}).

Now, the filtration on $v_{P_J}^G$ induces a filtration on $\HOrd[1]_{\Pbar^\alpha}(v_{P_J}^G)$ by $\Bbar^\alpha$-subrepresentations whose graded pieces are exactly the $\Bbar^\alpha$-representations $\HOrd[1]_{\Pbar^\alpha}(I_{\whw_J})$ for all $\whw_J \in \Wh_J$ (see the proofs of \cite[Prop.\ 2.3.6, Th.\ 2.3.7]{JHC}).
We partially compute the latter using \cite[Th.\ 2.6]{JHB}\footnote{
In loc.\ cit.\ $\Gb$ is split but the results extend to any $\Gb$ if one replaces $\ell(w)$ by $\dim(\Ubbar / \Ubbar_w)$.
Alternatively, these results can be recovered from \cite[Th.\ 3.3.3]{JHD}.
}.
If $F = \Q_p$ and $\dim \Ub_\alpha = 1$, then for any $\whw_J \in \Wh_J$, which we write $\whw_J = (\wtw^\alpha)^{-1} w^\alpha$ with $\wtw^\alpha \in \Wt^\alpha$ and $w^\alpha \in W^\alpha$, there is a $\Bbar_{w^\alpha}^\alpha$-equivariant isomorphism
\begin{equation} \label{H1Ord}
\HOrd[1]_{\Pbar^\alpha}(I_{\whw_J}) \simeq
\begin{cases}
\LCc(\Ubar^\alpha_{w^\alpha}) \otimes (\varepsilon \circ \alphat) & \text{if $\wtw^\alpha = s_\alpha$,} \\
0 & \text{otherwise.}
\end{cases}
\end{equation}
In particular, $\HOrd[1]_{\Pbar^\alpha}(v_{P_J}^G) \neq 0$ if and only if $\Wh_J \cap s_\alpha W^\alpha \neq \emptyset$, and the latter condition is equivalent to $\alpha \in J$ and $J \setminus \{\alpha\} \subset \{\alpha\}^\perp$ by Lemma \ref{lemm:combinatoire}(1) below.
If either $F \neq \Q_p$ or $\dim \Ub_\alpha \neq 1$, then $\HOrd[1]_{\Pbar^\alpha}(I_{\whw_J}) = 0$ for all $\whw_J \in \Wh_J$ so that $\HOrd[1]_{\Pbar^\alpha}(v_{P_J}^G) = 0$.
This proves (1).

We turn to (2).
Assume that $F = \Q_p$, $\dim \Ub_\alpha = 1$, $\alpha \in J$, and $J \setminus \{\alpha\} \subset \{\alpha\}^\perp$.
The $G$-equivariant surjection $i_{M^\alpha \cap P_I}^{M^\alpha} \twoheadrightarrow v_{M^\alpha \cap P_I}^{M^\alpha}$ induces an injection
\begin{equation*}
\Hom_{M^\alpha}(v_{M^\alpha \cap P_I}^{M^\alpha}, \HOrd[1]_{\Pbar^\alpha}(v_{P_J}^G)) \hookrightarrow \Hom_{M^\alpha}(i_{M^\alpha \cap P_I}^{M^\alpha}, \HOrd[1]_{\Pbar^\alpha}(v_{P_J}^G))
\end{equation*}
and the adjunction between $\Ind_{M^\alpha \cap P_I}^{M^\alpha}$ and $\Ord_{M^\alpha \cap \Pbar_I}$ yields an isomorphism
\begin{equation*}
\Hom_{M^\alpha}(i_{M^\alpha \cap P_I}^{M^\alpha}, \HOrd[1]_{\Pbar^\alpha}(v_{P_J}^G)) \simeq \Hom_{M_I}(1, \Ord_{M^\alpha \cap \Pbar_I}(\HOrd[1]_{\Pbar^\alpha}(v_{P_J}^G))).
\end{equation*}
We partially compute the $\Bbar_I$-representation $\Ord_{M^\alpha \cap \Pbar_I}(\HOrd[1]_{\Pbar^\alpha}(I_{\whw_J}))$ using the isomorphism \eqref{H1Ord} and, once again, \cite[Th.\ 2.6]{JHB}.
For any $\whw_J \in \Wh_J$ such that $\whw_J = s_\alpha w^\alpha$ with $w^\alpha \in W^\alpha$, which we write $w^\alpha = \wtw_I^{-1}w_I$ with $\wtw_I \in \Wt_I$ and $w_I \in W_I$, there is a $\Bbar_{I, w_I}$-equivariant isomorphism
\begin{equation*}
\Ord_{M^\alpha \cap \Pbar_I}(\HOrd[1]_{\Pbar^\alpha}(I_{\whw_J})) \simeq
\begin{cases}
\LCc(\Ubar_{I, w_I}) \otimes (\varepsilon \circ \alphat) & \text{if $\wtw_I = 1$,} \\
0 & \text{otherwise.}
\end{cases}
\end{equation*}
For any $w_I \in W_I$, we have $\LCc(\Ubar_{I, w_I})^{\Ubar_{I, w_I}} \neq 0$ if and only if $\Ubar_{I, w_I} = \{1\}$, i.e.\ $w_I = w_{I, 0}$, hence an isomorphism
\begin{equation*}
\Hom_{\Bbar_{I, w_I}}(1, \LCc(\Ubar_{I, w_I})) \otimes (\varepsilon \circ \alphat)) \simeq
\begin{cases}
\Hom_\Zc(1, \varepsilon \circ \alphat) & \text{if $w_I = w_{I, 0}$,} \\
0 & \text{otherwise.}
\end{cases}
\end{equation*}
Since $\Ord_{M^\alpha \cap \Pbar_I}$ and $\Hom_{M_I}(1, -)$ are left-exact, we deduce that
\begin{equation*}
\Hom_{M^\alpha}(v_{M^\alpha \cap P_I}^{M^\alpha}, \HOrd[1]_{\Pbar^\alpha}(v_{P_J}^G)) \neq 0
\end{equation*}
only if $s_\alpha w_{I,0} \in \Wh_J$, and the latter condition is equivalent to $\alpha \in J$ and $J \setminus \{\alpha\} = I \cap \{\alpha\}^\perp$ by Lemma \ref{lemm:combinatoire}(2) below.
In this case, there is an injection
\begin{equation} \label{HomH1Ord}
\Hom_{M^\alpha}(v_{M^\alpha \cap P_I}^{M^\alpha}, \HOrd[1]_{\Pbar^\alpha}(v_{P_J}^G)) \hookrightarrow \Hom_\Zc(1, \varepsilon \circ \alphat).
\end{equation}
Finally, we see by devissage that $\Hom_\Zc(1, \varepsilon \circ \alphat) \neq 0$ if and only if $\varepsbar \circ \alphat \neq 1$.
This proves (2).
\end{proof}

\begin{remark} \label{rema:H1Ord}
Assume that $F = \Q_p$, $\dim \Ub_\alpha = 1$, and $\varepsbar \circ \alphat = 1$.
\begin{enumerate}
\item We expect the $R$-module
\begin{equation*}
\Hom_{M^\alpha}(v_{M^\alpha \cap P_I}^{M^\alpha}, \HOrd[1]_{\Pbar^\alpha}(v_{P_J}^G))
\end{equation*}
to be non-zero if and only if $J \setminus I = \{\alpha\}$, $I \setminus J = \Delta \setminus \{\alpha\}^\perp$, and $I \cap J \subset \{\alpha\}^\perp$.
\item In general, this $R$-module is killed by $(\varepsilon \circ \alphat)(\Zc) - 1$ (see \eqref{HomH1Ord}), hence by $(\varepsilon \circ \alpha)(S) - 1$, which is equal to $\Z_p^* - 1$ if $\alpha \notin 2X^*(\Sb)$ and to ${\Z_p^*}^2 - 1$ if $\alpha \in 2X^*(\Sb)$.
We deduce that it is $2$-torsion when $p = 2$ and $\alpha \notin 2X^*(\Sb)$, $8$-torsion when $p = 2$ and $\alpha \in 2X^*(\Sb)$, $3$-torsion when $p = 3$ and $\alpha \in 2X^*(\Sb)$, and $0$ otherwise.
\end{enumerate}
\end{remark}

\begin{lemma} \label{lemm:combinatoire}
\begin{enumerate}
\item $\Wh_J \cap s_\alpha W^\alpha \neq \emptyset$ if and only if $\alpha \in J$ and $J \setminus \{\alpha\} \subset \{\alpha\}^\perp$.
\item $s_\alpha w_{I,0} \in \Wh_J$ if and only if $\alpha \in J$ and $J \setminus \{\alpha\} = I \cap \{\alpha\}^\perp$.
\end{enumerate}
\end{lemma}

\begin{proof}
We will use the following characterization:
\begin{equation*}
\Wh_J = \{w \in W \mid w^{-1}(J) \subset \Phi^- \ \text{and} \ w^{-1}(\Delta \setminus J) \subset \Phi^+\}.
\end{equation*}
Let $w^\alpha \in W^\alpha$.
For $\beta \in \Delta$, we have
\begin{equation*}
(s_\alpha w^\alpha)^{-1}(\beta) =
\begin{cases}
(w^\alpha)^{-1}(-\alpha) & \text{if $\beta = \alpha$,} \\
(w^\alpha)^{-1}(\beta) & \text{if $\beta \perp \alpha$,} \\
(w^\alpha)^{-1}(\beta - \langle \beta, \alpha^\vee \rangle \alpha) \ \text{with} \ \langle \beta, \alpha^\vee \rangle < 0 & \text{otherwise.}
\end{cases}
\end{equation*}
Since $(w^\alpha)^{-1}(\alpha) \in \Phi^+$, we deduce that
\begin{equation*}
(s_\alpha w^\alpha)^{-1}(\beta) \in \Phi^- \Leftrightarrow \beta = \alpha \text{ or } \beta \in (\Delta \cap w^\alpha(\Phi^-)) \cap \{\alpha\}^\perp.
\end{equation*}
Therefore, $s_\alpha w^\alpha \in \Wh_J$ if and only if $\alpha \in J$ and $J \setminus \{\alpha\} = (\Delta \cap w^\alpha(\Phi^-)) \cap \{\alpha\}^\perp$.
We deduce that $\Wh_J \cap s_\alpha W^\alpha \neq \emptyset$ only if $\alpha \in J$ and $J \setminus \{\alpha\} \subset \{\alpha\}^\perp$.
Conversely, if $\alpha \in J$ and $J \setminus \{\alpha\} \subset \{\alpha\}^\perp$, then $w_{J,0} = s_\alpha w_{J \setminus \{\alpha\}, 0} \in \Wh_J \cap s_\alpha W^\alpha$.
This proves (1).
Now $w_{I, 0} \in W^\alpha$ since $\alpha \notin I$, and $\Delta \cap w_{I, 0}(\Phi^-) = I$, hence (2).
\end{proof}

\subsubsection{Reduction to the case $J = \emptyset$}

Assume that $J \neq \emptyset$ and let $\alpha \in J$.
Taking into account Proposition \ref{prop:HomSt}, the short exact sequence \eqref{dev} with $J \setminus \{\alpha\}$ instead of $I$ induces an exact sequence
\begin{equation*}
0 \to \Ext_G^1(1, v_{P_J}^G) \to \Ext_G^1(1, \Ind_{P^\alpha}^G(v_{M^\alpha \cap P_{J \setminus \{\alpha\}}}^{M^\alpha})) \to \Ext_G^1(1, v_{P_{J \setminus \{\alpha\}}}^G).
\end{equation*}
Using Lemma \ref{lemm:Shapiro} below with $\pi = v_{M^\alpha \cap P_{J \setminus \{\alpha\}}}^{M^\alpha}$, we obtain an exact sequence
\begin{equation} \label{redJ}
0 \to \Ext_G^1(1, v_{P_J}^G) \to \Ext_{M^\alpha}^1(1, v_{M^\alpha \cap P_{J \setminus \{\alpha\}}}^{M^\alpha}) \to \Ext_G^1(1, v_{P_{J \setminus \{\alpha\}}}^G).
\end{equation}

\begin{lemma} \label{lemm:Shapiro}
Let $\pi$ be a smooth $M^\alpha$-representation.
There is an isomorphism
\begin{equation*}
\Ext_{M^\alpha}^1(1, \pi) \simeq \Ext_G^1(1, \Ind_{P^\alpha}^G(\pi)).
\end{equation*}
\end{lemma}

\begin{proof}
By a straightforward generalization (to any $\Gb$) of a special case of \cite[Lemma 4.3.3]{Em2}, there is an isomorphism
\begin{equation*}
\Ext_{P^\alpha}^1(1, \pi) \simeq \Ext_G^1(1, \Ind_{P^\alpha}^G(\pi))
\end{equation*}
where $P^\alpha$ acts on $\pi$ by inflation.
Thus it remains to prove that there is an isomorphism
\begin{equation*}
\Ext_{M^\alpha}^1(1, \pi) \simeq \Ext_{P^\alpha}^1(1, \pi).
\end{equation*}
By \cite[Prop.\ 2.2.2]{Em2}, such an isomorphism can be rewritten in terms of group cohomology, computed using continuous cochains, as follows:
\begin{equation*}
H^1(M^\alpha, \pi) \simeq H^1(P^\alpha, \pi).
\end{equation*}
Using the fact that $N^\alpha$ acts trivially on $\pi$, the inflation-restriction exact sequence for continuous group cohomology yields an exact sequence
\begin{equation*}
0 \to H^1(M^\alpha, \pi) \to H^1(P^\alpha, \pi) \to \Hom(N^\alpha, \pi)^{M^\alpha}.
\end{equation*}
We prove that the last term is zero.
Let $\phi : N^\alpha \to \pi$ be a continuous group homomorphism.
The action of $m \in M^\alpha$ on $\phi$ is given by $m \cdot \phi : n \mapsto m \cdot \phi(m^{-1}nm)$.
Since $\phi$ is continuous, $\ker(\phi)$ is open in $N^\alpha$.
If $\phi$ is $M^\alpha$-invariant, then $\ker(\phi) = \ker(m \cdot \phi) = m \ker(\phi) m^{-1}$ for all $m \in M^\alpha$, hence $\ker(\phi) = N^\alpha$, i.e.\ $\phi = 0$.
\end{proof}

\emph{We assume from now on that $I = \Delta$, $J = \emptyset$, and $\Delta \neq \emptyset$.
The goal of the next sections is to compute the $R$-module $\Ext_G^1(1, \St)$.
Using \cite[Prop.\ 2.2.2]{Em2}, the latter can be expressed in terms of group cohomology, computed using continuous cochains, as follows:
\begin{equation*}
\Ext_G^1(1, \St) \simeq H^1(G, \St).
\end{equation*}}

\subsection{Vanishing and isogenies} \label{isogenies}

In this section we give two technical lemmas.

\begin{lemma} \label{lemm:product}
Assume that $\Gb = \Gb_1 \times \Gb_2$ for some connected reductive algebraic $F$-groups $\Gb_1$ and $\Gb_2$ such that $\Gb_1$ has positive relative semisimple rank.
\begin{enumerate}
\item If $\Gb_2$ also has positive relative semisimple rank, then $\Ext_G^1(1, \St) = 0$.
\item If $\Gb_2$ has relative semisimple rank $0$, then $\St$ is isomorphic to the Steinberg representation of $G_1$, with $G_2$ acting trivially on it, and there is an isomorphism
\begin{equation*}
\Ext_G^1(1, \St) \iso \Ext_{G_1}^1(1, \St).
\end{equation*}
\end{enumerate}
\end{lemma}

\begin{proof}
Let $\St_i$ denote the Steinberg representation of $G_i$ for $i \in \{1, 2\}$.
We write $\Bb = \Bb_1 \times \Bb_2$.
There is a $G$-equivariant isomorphism $i_B^G \simeq i_{B_1}^{G_1} \otimes_R i_{B_2}^{G_2}$ which induces a $G$-equivariant isomorphism $\St \simeq \St_1 \otimes_R \St_2$.
Note that the $R$-module $\St_2$ is free by \cite[Cor.\ 5.6]{LySt}.
Thus $\St^{G_1} \simeq (\St_1)^{G_1} \otimes_R \St_2 = 0$ by Proposition \ref{prop:HomSt}.
Therefore, since the $R$-module $\St_1$ is also free, the inflation-restriction exact sequence for continuous group cohomology yields an isomorphism
\begin{equation*}
H^1(G, \St) \iso H^1(G_1, \St_1 \otimes_R (\St_2)^{G_2}).
\end{equation*}
If $\Gb_2$ has positive semisimple rank, then $(\St_2)^{G_2} = 0$ by Proposition \ref{prop:HomSt}, hence (1).
If $\Gb_2$ has relative semisimple rank $0$, then $\St_2 = 1$ so that $\St_1 \otimes_R (\St_2)^{G_2} \simeq \St_1$, hence (2).
\end{proof}

\begin{lemma} \label{lemm:isogeny}
Let $\varphi : \Gb' \to \Gb$ be a central isogeny of connected reductive algebraic $F$-groups.
Let $\St'$ denote the Steinberg representation of $G'$.
\begin{enumerate}
\item If $\Ext_{G'}^1(1, \St') = 0$, then $\Ext_G^1(1, \St) = 0$.
\item The converse holds true if $|\Delta| \geq 2$.
\end{enumerate}
\end{lemma}

\begin{proof}
Proceeding as in the proof of \cite[II.8 Lemma]{AHHV}, we see that the $R$-module $\St$ endowed with the action of $G'$ via $\varphi$ is isomorphic to $\St'$.
In particular, $\St^{\varphi(G')} \simeq (\St')^{G'} = 0$ by Proposition \ref{prop:HomSt} and $\varphi$ induces an isomorphism
\begin{equation} \label{iso-varphi}
H^1(\varphi(G'), \St) \iso H^1(G', \St').
\end{equation}
(Here we also use the fact that $\varphi$ is central and that the center of $G'$ acts trivially on $\St'$.)
Therefore, the inflation-restriction exact sequence for continuous group cohomology yields an isomorphism
\begin{equation*}
H^1(G, \St) \iso H^1(\varphi(G'), \St)^{G / \varphi(G')},
\end{equation*}
hence (1).
For (2), it is enough to show that the action of $G$ on $H^1(\varphi(G'), \St)$ is trivial when $|\Delta| \geq 2$.
We defer the proof of this fact to the end of Subsection \ref{injectivity}.
\end{proof}

\subsection{Comparison of $\Ext^1$ for $G$ and for $\Bbar$} \label{comparison}

In this section we construct a map $\Ext_G^1(1, \St) \to \Ext_\Zc^1(1, 1)$ which factors through $\Ext_{\Bbar}^1(1, \St)$;
under some assumptions, we prove that it is injective and we identify its image.

\subsubsection{Construction of a map}

We construct a map
\begin{equation} \label{comp}
\Ext_G^1(1, \St) \to \Ext_\Zc^1(1, 1).
\end{equation}
Recall that the $R$-module $\LCc(\Ubar)$ is endowed with a smooth action of $\Bbar$ defined by
\begin{equation} \label{LCc}
(z \ubar \cdot f)(\ubar') = f(z^{-1} \ubar' z \ubar)
\end{equation}
for $z \in \Zc$, $\ubar, \ubar' \in \Ubar$, and $f \in \LCc(\Ubar)$.
We deduce from the Bruhat filtration that there is a $\Bbar$-equivariant isomorphism $\jmath : \LCc(\Ubar) \iso \St$ (see the proof of Lemma \ref{lemm:H1Ord} and use the fact that $\Wh_\emptyset = \{1\}$), which is the composite
\begin{equation} \label{jmath}
\LCc(\Ubar) \hookrightarrow i_B^G \twoheadrightarrow \St
\end{equation}
where the first map is induced by the open immersion $\Ubar \hookrightarrow B \backslash G$ and the second map is the $G$-equivariant surjection.
The evaluation at $1 \in \Ubar$ yields a $\Zc$-equivariant surjection $\ev_1 : \LCc(\Ubar) \twoheadrightarrow 1$.
The $\Zc$-equivariant surjection $\ev_1 \circ \jmath^{-1} : \St \twoheadrightarrow 1$ induces a map
\begin{equation} \label{ev1}
\Ext_{\Bbar}^1(1, \St) \to \Ext_\Zc^1(1, 1).
\end{equation}
Finally, we define \eqref{comp} as the map obtained by composing \eqref{ev1} with the restriction map
\begin{equation} \label{res}
\Ext_G^1(1, \St) \to \Ext_{\Bbar}^1(1, \St).
\end{equation}
In terms of group cohomology, \eqref{comp} corresponds to the map
\begin{equation} \label{comp-coh}
H^1(G, \St) \to H^1(\Zc, 1) \simeq \Hom(\Zc, R)
\end{equation}
defined by $\phi \mapsto \ev_1 \circ \jmath^{-1} \circ \phi_{|\Zc}$.

\subsubsection{Injectivity of the map} \label{injectivity}

First, the map \eqref{res} is injective by a straightforward generalization (to any $\Gb$) of a special case of \cite[Lemma 4.3.5]{Em2}.
In order to prove that the map \eqref{ev1} is also injective, we generalize the results of \cite[Sec.\ A.2]{Em2}.

We define an action of $\Bbar$ on $\LC(\Ubar)$ using the formula \eqref{LCc}.
This action is not smooth and we let $\LCu(\Ubar) \subset \LC(\Ubar)$ denote the $R$-submodule consisting of those functions on which the action of $\Bbar$ is smooth.
We keep the notation $\LCu(\Ubar)$ from \cite[Sec.\ A.2]{Em2} although it is \emph{not} the $R$-module of uniformly locally constant functions on $\Ubar$.\footnote{
A function $f \in \LC(\Ubar)$ is uniformly locally constant if and only if the action of $\Ubar$ on $f$ is smooth, but, contrary to what is claimed in \cite[Sec.\ A.2]{Em2}, the action of $\Zc$ on such a function is not necessarily smooth.
}
The homeomorphism $\Ubar \iso \Zc \backslash \Bbar$ induces a $\Bbar$-equivariant isomorphism
\begin{equation} \label{isoInd}
\LCu(\Ubar) \simeq \Ind_\Zc^{\Bbar}(1).
\end{equation}
We consider the long exact sequence
\begin{align*}
0 & \to \Hom_{\Bbar}(1, \LCc(\Ubar)) \to \Hom_{\Bbar}(1, \LCu(\Ubar)) \to \Hom_{\Bbar}(1, \LCu(\Ubar) / \LCc(\Ubar)) \\
& \to \Ext_{\Bbar}^1(1, \LCc(\Ubar)) \to \Ext_{\Bbar}^1(1, \LCu(\Ubar)).
\end{align*}
Clearly $\Hom_{\Bbar}(1, \LCc(\Ubar)) = 0$, while Frobenius reciprocity and \eqref{isoInd} induce isomorphisms
\begin{gather*}
\Hom_{\Bbar}(1, \LCu(\Ubar)) \simeq \Hom_\Zc(1, 1), \\
\Ext_{\Bbar}^1(1, \LCu(\Ubar)) \simeq \Ext_\Zc^1(1, 1).
\end{gather*}
(For the latter see the proof of \cite[Lemma 4.3.3]{Em2}.)
By a straightforward generalization (to any $\Gb$) of a special case of \cite[Lemma 4.3.4]{Em2}, there is an isomorphism
\begin{equation*}
\Hom_{\Bbar}(1, \LCu(\Ubar) / \LCc(\Ubar)) \simeq \Hom_\Zc(1, \LCu(\Ubar) / \LCc(\Ubar)).
\end{equation*}
Finally, the $\Bbar$-equivariant injection $\LCc(\Ubar) \hookrightarrow \Ind_\Zc^{\Bbar}(1)$ induced by the inclusion $\LCc(\Ubar) \subset \LCu(\Ubar)$ and \eqref{isoInd} corresponds via Frobenius reciprocity to the $\Zc$-equivariant surjection $\ev_1 : \LCc(\Ubar) \twoheadrightarrow 1$.
Therefore, using the isomorphism $\jmath$ (which is the composite \eqref{jmath}), the above long exact sequence can be rewritten as
\begin{align} \label{LES}
0 & \to \Hom_\Zc(1, 1) \to \Hom_\Zc(1, \LCu(\Ubar) / \LCc(\Ubar)) \to \Ext_{\Bbar}^1(1, \St) \xrightarrow{\eqref{ev1}} \Ext_\Zc^1(1, 1).
\end{align}

\begin{lemma} \label{lemm:inj}
Assume that the adjoint action of $\Zc$ on $\Ubar \setminus \{1\}$ is transitive or that $|\Delta| \geq 2$.
Then the $R$-module $\Hom_\Zc(1, \LCu(\Ubar) / \LCc(\Ubar))$ is free of rank $1$.
\end{lemma}

\begin{proof}
Let $f \in \LCu(\Ubar)$ such that $z \cdot f - f \in \LCc(\Ubar)$ for all $z \in \Zc$.
We will prove that $f \in R + \LCc(\Ubar)$, where we identify $R$ with the $R$-submodule of $\LCu(\Ubar)$ consisting of the constant functions.

We fix a compact open subgroup $\Ubar_0 \subset \Ubar$ and an element $z_+ \in Z_\Zc$ strictly contracting $\Ubar_0$, i.e.\ such that $z_+ \Ubar_0 z_+^{-1} \subset \Ubar_0$ and $\bigcap_{i \geq 1} z_+^i \Ubar_0 z_+^{-i} = \{1\}$, or equivalently, $\bigcup_{i \geq 1} z_+^{-i} \Ubar_0 z_+^i = \Ubar$.
We have a decomposition of $\Ubar$ into a disjoint union of open subsets
\begin{equation*}
\Ubar = \Ubar_0 \sqcup \bigsqcup_{i \geq 1} z_+^{-i}(\Ubar_0 \setminus z_+ \Ubar_0 z_+^{-1})z_+^i.
\end{equation*}
Correspondingly, we can write
\begin{equation*}
f = f_0 + \sum_{i \geq 1} z_+^{-i} \cdot f_i
\end{equation*}
with $\supp(f_0) \subset \Ubar_0$ and $\supp(f_i) \subset \Ubar_0 \setminus z_+ \Ubar_0 z_+^{-1}$ for all $i \geq 1$.
We have
\begin{equation*}
z_+ \cdot f - f = z_+ \cdot f_0 - f_0 + f_1 + \sum_{i \geq 1} z_+^{-i} \cdot(f_{i+1} - f_i).
\end{equation*}
Since $z_+ \cdot f - f \in \LCc(\Ubar)$, there exists $j \geq 1$ such that $f_{i+1} = f_i$ for all $i \geq j$.
Thus we can rewrite
\begin{equation*}
f = f_0 + \sum_{i \geq 1} z_+^{-i} \cdot f_\infty
\end{equation*}
for some $f_0, f_\infty \in \LCc(\Ubar)$ with $\supp(f_\infty) \subset \Ubar_0 \setminus z_+ \Ubar_0 z_+^{-1}$.
Note that $f \in R + \LCc(\Ubar)$ if and only if $f_\infty$ is constant on $\Ubar_0 \setminus z_+ \Ubar_0 z_+^{-1}$.

We prove a key identity.
Let $\ubar \in \Ubar_0 \setminus z_+ \Ubar_0 z_+^{-1}$ and $z \in \Zc$ such that $z \ubar z^{-1} \in \Ubar_0 \setminus z_+ \Ubar_0 z_+^{-1}$.
For any $j \geq 1$ large enough so that $z_+^{-j} \ubar z_+^j \notin \supp(z^{-1} \cdot f - f)$ and $z_+^{-j} \ubar z_+^j \notin \supp(z^{-1} \cdot f_0 - f_0)$, we have
\begin{align*}
0 & = (z^{-1} \cdot f - f)(z_+^{-j} \ubar z_+^j) \\
& = (z^{-1} \cdot f_0 - f_0)(z_+^{-j} \ubar z_+^j)+ \sum_{i \geq 1}(z^{-1} \cdot f_\infty - f_\infty)(z_+^{i-j} \ubar z_+^{j-i}) \\
& = \sum_{i \geq 1} f_\infty(z_+^{i-j}(z \ubar z^{-1})z_+^{j-i}) - \sum_{i \geq 1} f_\infty(z_+^{i-j} \ubar z_+^{j-i}) \\
& = f_\infty(z \ubar z^{-1}) - f_\infty(\ubar),
\end{align*}
the last equality resulting from the fact that $\ubar \in \Ubar_0 \setminus z_+ \Ubar_0 z_+^{-1}$ and $z \ubar z^{-1} \in \Ubar_0 \setminus z_+ \Ubar_0 z_+^{-1}$ whereas $\supp(f_\infty) \subset \Ubar_0 \setminus z_+ \Ubar_0 z_+^{-1}$, so that $f_\infty(z_+^{i-j} (z \ubar z^{-1}) z_+^{j-i}) = f_\infty(z_+^{i-j} \ubar z_+^{j-i}) = 0$ if $j \neq i$.
Therefore,
\begin{equation} \label{key}
f_\infty(z \ubar z^{-1}) = f_\infty(\ubar).
\end{equation}
In particular, we deduce that $f_\infty$ is constant on $\Ubar_0 \setminus z_+ \Ubar_0 z_+^{-1}$ if the adjoint action of $\Zc$ on $\Ubar \setminus \{1\}$ is transitive.

Now let $\alpha \in \Delta$.
We let $\Ubar_{\alpha, 0} = \Ubar_\alpha \cap \Ubar_0$ and $\Nbar_{\alpha, 0} = \Nbar_\alpha \cap \Ubar_0$.
Replacing $\Ubar_0$ by $\Ubar_{\alpha, 0} \Nbar_{\alpha, 0}$, we can assume that $\Ubar_0 = \Ubar_{\alpha, 0} \Nbar_{\alpha, 0}$, hence $z_+ \Ubar_0 z_+^{-1} = (z_+ \Ubar_{\alpha, 0} z_+^{-1}) (z_+ \Nbar_{\alpha, 0} z_+^{-1})$, and we fix an element $z_{\alpha, +} \in Z_\alpha$ strictly contracting $\Nbar_{\alpha, 0}$.
Let $\ubar \in \Ubar_0 \setminus z_+ \Ubar_0 z_+^{-1}$.
We write $\ubar = \ubar_\alpha \nbar_\alpha$ with $\ubar_\alpha \in \Ubar_{\alpha, 0}$ and $\nbar_\alpha \in \Nbar_{\alpha, 0}$.
Let $i \geq 0$ such that $z_+^{-i} \ubar_\alpha z_+^i \in \Ubar_{\alpha, 0} \setminus z_+ \Ubar_{\alpha, 0} z_+^{-1}$.
For any $j \geq 0$ large enough so that $z_{\alpha, +}^j z_+^{-i} \nbar_\alpha z_+^i z_{\alpha, +}^{-j} \in \Nbar_{\alpha, 0}$, we have
\begin{equation*}
z_{\alpha, +}^j z_+^{-i} \ubar z_+^i z_{\alpha, +}^{-j} = (z_+^{-i} \ubar_\alpha z_+^i)(z_{\alpha, +}^j z_+^{-i} \nbar_\alpha z_+^i z_{\alpha, +}^{-j}) \in \Ubar_0 \setminus z_+ \Ubar_0 z_+^{-1},
\end{equation*}
hence using \eqref{key} we obtain
\begin{equation*}
f_\infty(\ubar) = f_\infty((z_+^{-i} \ubar_\alpha z_+^i)(z_{\alpha, +}^j z_+^{-i} \nbar_\alpha z_+^i z_{\alpha, +}^{-j})).
\end{equation*}
Since $f_\infty$ is locally constant, we deduce by passing to the limit as $j \to +\infty$ that
\begin{equation} \label{alpha}
f_\infty(\ubar) = f_\infty(z_+^{-i} \ubar_\alpha z_+^i).
\end{equation}
Therefore, $f_\infty$ is constant on $\Ubar_0 \setminus z_+ \Ubar_0 z_+^{-1}$ if and only if it is constant on $\Ubar_{\alpha, 0} \setminus z_+ \Ubar_{\alpha, 0} z_+^{-1}$.

Finally, assume that $|\Delta| \geq 2$.
Let $\nbar_\alpha \in \Nbar_{\alpha, 0} \setminus z_+ \Nbar_{\alpha, 0} z_+^{-1}$.
For any $\ubar_\alpha \in \Ubar_{\alpha, 0} \setminus z_+ \Ubar_{\alpha, 0} z_+^{-1}$ and for any $i \geq 0$, we have $z_+^i \ubar_\alpha z_+^{-i} \nbar_\alpha \in \Ubar_0 \setminus z_+ \Ubar_0 z_+^{-1}$, hence using \eqref{alpha} we obtain
\begin{equation*}
f_\infty(z_+^i \ubar_\alpha z_+^{-i} \nbar_\alpha) = f_\infty(\ubar_\alpha).
\end{equation*}
Since $f_\infty$ is locally constant, we deduce by passing to the limit as $i \to +\infty$ that
\begin{equation*}
f_\infty(\nbar_\alpha) = f_\infty(\ubar_\alpha).
\end{equation*}
Therefore, $f_\infty$ is constant on $\Ubar_{\alpha, 0} \setminus z_+ \Ubar_{\alpha, 0} z_+^{-1}$.
\end{proof}

Using the exact sequence \eqref{LES}, we deduce from Lemma \ref{lemm:inj} the following result.

\begin{proposition} \label{prop:inj}
Assume that the adjoint action of $\Zc$ on $\Ubar \setminus \{1\}$ is transitive or that $|\Delta| \geq 2$.
Then \eqref{comp} is injective.
\end{proposition}

\begin{remark}
If $\Delta = \{\alpha\}$ and the adjoint action of $\Zc$ on $\Ubar \setminus \{1\}$ is not transitive, then \eqref{ev1} need not be injective.
For example, if $G = \SL_2(F)$, then the $R$-module in Lemma \ref{lemm:inj} is free of rank $|F^* / {F^*}^2| \geq 4$ so that \eqref{ev1} is not injective.
However, we do not know whether \eqref{comp} remains injective in general.
\end{remark}

Using Proposition \ref{prop:inj}, we can now complete the proof of Lemma \ref{lemm:isogeny}.

\begin{proof}[End of the proof of Lemma \ref{lemm:isogeny}]
Assume that $|\Delta| \geq 2$.
It remains to show that the action of $G / \varphi(G')$ on $H^1(\varphi(G'), \St)$ is trivial.
Note that $G = \Zc \varphi(G')$ (since $\varphi(G')$ contains $U$ and $\Ubar$) and that $\Zc \cap \varphi(G') = \varphi(\Zc')$, hence an isomorphism $\Zc / \varphi(\Zc') \iso G / \varphi(G')$.
Now \eqref{iso-varphi} composed with \eqref{comp-coh} with $\Gb'$ instead of $\Gb$ induces a map
\begin{equation} \label{comp-coh-varphi}
H^1(\varphi(G'), \St) \to H^1(\varphi(\Zc'), 1) \simeq \Hom(\varphi(\Zc'), R)
\end{equation}
which is injective by Proposition \ref{prop:inj} with $\Gb'$ instead of $\Gb$.
The $R$-module $\Hom(\varphi(\Zc'), R)$ is endowed with the action of $\Zc / \varphi(\Zc')$ given by $z \cdot \phi : z' \mapsto \phi(z^{-1} z' z)$ and \eqref{comp-coh-varphi} commutes with the action of $\Zc / \varphi(\Zc')$ on its source and target.
Therefore, it is enough to prove that the action of $\Zc / \varphi(\Zc')$ on $\Hom(\varphi(\Zc'), R)$ is trivial, i.e.\ that any continuous group homomorphism $\phi : \varphi(\Zc') \to R$ is trivial on the commutator subgroup $[\Zc, \varphi(\Zc')]$.

Let $\Zcb_\ssc$ denote the simply connected covering of the derived subgroup $\Zcb_\der$ of $\Zcb$ and $\iota$ denote the natural homomorphism $\Zcb_\ssc \to \Zcb$.
Then $[\Zc, \Zc] \subset \iota(\Zc_\ssc)$ (this holds true for any connected reductive algebraic $F$-group).
Moreover, $\iota$ factors through $\Zcb'_\der \twoheadrightarrow \Zcb_\der$ by the universal property of the simply connected covering, hence $\iota(\Zc_\ssc) \subset \varphi(\Zc'_\der) \subset \varphi(\Zc')$.
In conclusion, we have inclusions of subgroups $[\Zc, \varphi(\Zc')] \subset [\Zc, \Zc] \subset \iota(\Zc_\ssc) \subset \varphi(\Zc')$.
Finally, the abelianization of $\Zc_\ssc$ is a finite group of prime-to-$p$ order (this will be explained in detail in Subsection \ref{devissage}) so that any group homomorphism $\Zc_\ssc \to R$ is trivial.
In particular, any continuous group homomorphism $\phi : \varphi(\Zc') \to R$ is trivial on $\iota(\Zc_\ssc)$, hence on $[\Zc, \varphi(\Zc')]$.
\end{proof}

\subsubsection{Image of the map}

We begin with a preliminary result on the image of \eqref{ev1}.

\begin{lemma} \label{lemm:imev1}
\begin{enumerate}
\item Any extension in the image of \eqref{ev1} is trivial on $Z_\alpha$ for all $\alpha \in \Delta$.
\item If $|\Delta| \geq 3$ or $\Delta = \{\alpha, \beta\}$ with $\alpha \perp \beta$, then \eqref{ev1} is zero.
\end{enumerate}
\end{lemma}

\begin{proof}
Let $I \subset \Delta$ be non-empty and $\St_{M_I}$ denote the Steinberg representation of $M_I$.
Since $\ev_1$ factors through the $\Bbar_I$-equivariant surjection $\LCc(\Ubar) \twoheadrightarrow \LCc(\Ubar_I)$ given by the restriction to $\Ubar_I$, \eqref{ev1} factors through the map
\begin{equation} \label{ev1I}
\Ext_{\Bbar_I}^1(1, \St_{M_I}) \to \Ext_\Zc^1(1, 1)
\end{equation}
induced by the $\Bbar_I$-equivariant isomorphism $\St_{M_I} \iso \LCc(\Ubar_I)$ and the $\Zc$-equivariant surjection $\LCc(\Ubar_I) \twoheadrightarrow 1$ given by the evaluation at $1 \in \Ubar_I$.
Since $\Hom_{\Bbar_I}(1, \St_{M_I}) = 0$ (as can be seen from the previous isomorphism), any extension in the source of \eqref{ev1I} is trivial on $Z_I$, hence (1) by taking $I = \{\alpha\}$.
Now assume that there exist $\alpha, \beta \in \Delta$ such that $\alpha \perp \beta$ (e.g.\ $|\Delta| \geq 3$).
Proceeding as in the proofs of Lemmas \ref{lemm:product}(1) and \ref{lemm:isogeny}(1), we see that the source of \eqref{ev1I} with $I = \{\alpha, \beta\}$ is zero, hence (2).
\end{proof}

Now we turn to the image of \eqref{comp}.
The $R$-module $\Ext_\Zc^1(1, 1)$ is endowed with the action of $W$ induced by the action of $\Nc$ on $\Zc$ by conjugation.
The corresponding action of $W$ on the $R$-module $\Hom(\Zc, R)$ is given by $w \cdot \phi : z \mapsto \phi(n^{-1}zn)$, where $n \in \Nc$ is any representative of $w$.
We also write $\phi^w = w^{-1} \cdot \phi$.
When $w = s_\alpha$ for some $\alpha \in \Delta$, we rather write $\phi^\alpha$.
We let $\epsilon : W \to \{\pm1\}$ be the character defined by $\epsilon(w) = (-1)^{l_w}$, where $l_w$ denotes the length of $w$, and we write $\Ext_\Zc^1(1, 1)^\epsilon$ for the $\epsilon$-isotypic component of $\Ext_\Zc^1(1, 1)$, that is the $R$-submodule consisting of those $\phi$ such that $w \cdot \phi = \epsilon(w) \phi$ for all $w \in W$.
When $\Delta = \{\alpha\}$, we rather write $\Ext_\Zc^1(1, 1)^{s_\alpha = -1}$.

\begin{lemma} \label{lemm:imeps}
The image of \eqref{comp} lies in $\Ext_\Zc^1(1, 1)^\epsilon$.
\end{lemma}

\begin{proof}
Recall the explicit description \eqref{comp-coh} of the map \eqref{comp} in terms of continuous cochains.
It is enough to prove that $s_\alpha$ acts by multiplication by $-1$ on the image of \eqref{comp-coh} for all $\alpha \in \Delta$.

Let $\alpha \in \Delta$ and $n_\alpha \in \Nc$ be a representative of $s_\alpha$.
If $\phi : G \to \St$ is a continuous cochain, then $\phi$ and $n_\alpha \cdot \phi : g \mapsto n_\alpha \cdot \phi(n_\alpha^{-1}gn_\alpha)$ are cohomologous since $n_\alpha \in G$, thus their images in $\Hom(\Zc, R)$ are equal, i.e.\ $\jmath^{-1}(\phi(z))(1) = \jmath^{-1}(n_\alpha \cdot \phi(n_\alpha^{-1} z n_\alpha))(1)$ for all $z \in \Zc$, where $\jmath$ is the composite \eqref{jmath}.
Therefore, it is enough to prove that $\jmath^{-1}(n_\alpha \cdot v)(1) = -\jmath^{-1}(v)(1)$ for all $v \in \St$.

Let $v \in \St$.
We let $f = \jmath^{-1}(v)$ and $f_\alpha = \jmath^{-1}(n_\alpha \cdot v)$; that is, $f$ (resp. $f_\alpha$) is the unique lift of $v$ (resp.\ $n_\alpha \cdot v$) in $i_B^G$ with support in $B\Ubar$.
Since $f_\alpha$ and $n_\alpha \cdot f$ have the same image in $\St$ (which is $n_\alpha \cdot v$), we have $f_\alpha - n_\alpha \cdot f \in \sum_{I \neq \emptyset} i_{P_I}^G$.
Moreover, since $n_\alpha \in \Pbar_\alpha$, we have $\supp(f_\alpha - n_\alpha \cdot f) \subset B \Pbar_\alpha = P_\alpha \Bbar$.
We let $\ci_B^{B \Pbar_\alpha}$ (resp.\ $\ci_{P_\alpha}^{P_\alpha \Bbar}$) denote the $\Pbar_\alpha$-subrepresentation of $i_B^G$ (resp.\ $i_{P_\alpha}^G$) consisting of those functions with support in $B \Pbar_\alpha = P_\alpha \Bbar$, so that $f_\alpha - n_\alpha \cdot f \in \ci_B^{B \Pbar_\alpha}$.
By \cite[Lemma 7.8]{AHV}, we have
\begin{equation*}
(\sum_{I \neq \emptyset} i_{P_I}^G) \cap \ci_B^{B \Pbar_\alpha} = \sum_{I \neq \emptyset}(i_{P_I}^G \cap \ci_B^{B \Pbar_\alpha})
\end{equation*}
and we deduce from \cite[Prop.\ 7.11]{AHV} that for $I \neq \emptyset$, we have
\begin{equation*}
i_{P_I}^G \cap \ci_B^{B \Pbar_\alpha} =
\begin{cases}
\ci_{P_\alpha}^{P_\alpha \Bbar} & \text{ if $I = \{\alpha\}$,} \\
0 & \text{otherwise.}
\end{cases}
\end{equation*}
Therefore, $f_\alpha - n_\alpha \cdot f \in \ci_{P_\alpha}^{P_\alpha \Bbar}$.
Now we compute:
\begin{equation*}
\jmath^{-1}(n_\alpha \cdot v)(1) = f_\alpha(1) = (f_\alpha - n_\alpha \cdot f)(1) = (f_\alpha - n_\alpha \cdot f)(n_\alpha^{-1}) = - f(1) = -\jmath^{-1}(v)(1).
\end{equation*}
(The first and last equalities follow from the definition of $f$ and $f_\alpha$, the second and next to last ones result from the fact that $n_\alpha \notin B\Ubar$ and $n_\alpha^{-1} \notin B\Ubar$ whereas $\supp(f) \subset B\Ubar$ and $\supp(f_\alpha) \subset B\Ubar$, and the middle equality is a consequence of the fact that $n_\alpha^{-1} \in P_\alpha$ and $f_\alpha - n_\alpha \cdot f \in i_{P_\alpha}^G$.)
\end{proof}

In order to determine the image of \eqref{comp} in the case $\Delta = \{\alpha\}$, we generalize the proof of \cite[Prop.\ 4.3.22(2)]{Em2}.
We consider the composite
\begin{equation} \label{projeps}
\Ext_\Zc^1(1, 1) \hookrightarrow \Ext_G^1(i_B^G, i_B^G) \to \Ext_G^1(1, \St) \xrightarrow{\eqref{comp}} \Ext_\Zc^1(1, 1)^\epsilon
\end{equation}
where the first map is induced by the exact functor $\Ind_B^G$, the second map is induced by the $G$-equivariant injection $1 \hookrightarrow i_B^G$ and the $G$-equivariant surjection $i_B^G \twoheadrightarrow \St$, and the third map is \eqref{comp} (which takes values in $\Ext_\Zc^1(1, 1)^\epsilon$ by Lemma \ref{lemm:imeps}).

\begin{lemma} \label{lemm:projalpha}
Assume that $\Delta = \{\alpha\}$.
Then \eqref{projeps} is given by $\phi \mapsto \phi - \phi^\alpha$.
\end{lemma}

\begin{proof}
First, we consider the $\Zc$-equivariant surjection
\begin{equation*}
\pi : i_B^G \twoheadrightarrow \St \xrightarrow{\jmath^{-1}} \LCc(\Ubar) \xtwoheadrightarrow{\ev_1} 1
\end{equation*}
where the first map is the $G$-equivariant surjection, the second map is the $G$-equivariant isomorphism $\jmath^{-1}$ (which is the inverse of the composite \eqref{jmath}), and the third map is the evaluation at $1 \in \Ubar$ (which is $\Zc$-equivariant).
Let $n_\alpha \in \Nc$ be a representative of $s_\alpha$.
For any $f \in i_B^G$, the function $f - f(n_\alpha) \in i_B^G$ has the same image as $f$ in $\St$ and $\supp(f - f(n_\alpha)) \subset B\Ubar$, hence $\pi(f) = f(1) - f(n_\alpha)$.

Now, using $\pi$, we give an explicit description of the map \eqref{projeps} in terms of continuous cochains.
Let $\phi : \Zc \to R$ be a continuous group homomorphism.
We let $E$ be the corresponding extension of $1$ by $1$, i.e.\ $E = R e_1 \oplus R e_2$ as an $R$-module and the smooth action of $\Zc$ is given by $z \cdot e_1 = e_1$ and $z \cdot e_2 = e_2 + \phi(z) e_1$.
We consider the $G$-representation $\Ind_B^G(E)$, which is an extension of $i_B^G$ by $i_B^G$.
Any $f \in \Ind_B^G(E)$ can be uniquely written $f = f_1 e_1 + f_2 e_2$ with $f_1, f_2 \in \LC(G)$ satisfying $f_1(zug) = f_1(g) + \phi(z) f_2(g)$ and $f_2(zug) = f_2(g)$ for all $z \in \Zc$, $u \in U$, and $g \in G$.
We let $I$ denote the extension of $1$ by $i_B^G$ obtained from $\Ind_B^G(E)$ by pullback along the $G$-equivariant injection $1 \hookrightarrow i_B^G$, i.e.\ $I$ is the $G$-subrepresentation of $\Ind_B^G(E)$ consisting of those functions $f = f_1 e_1 + f_2 e_2$ with $f_2$ constant.
The image of $E$ under \eqref{projeps} is the extension $E'$ of $1$ by $1$ obtained from $I$ by pushforward along the $\Zc$-equivariant surjection $\pi$.
We fix $f \in I$ such that $f = f_1 e_1 + e_2$, so that $f_1(zug) = f_1(g) + \phi(z)$ for all $z \in \Zc$, $u \in U$, and $g \in G$.
Therefore, the continuous group homomorphism $\phi' : \Zc \to R$ corresponding to $E'$ is given by
\begin{align*}
\phi'(z) & = \pi(z \cdot f_1 - f_1) \\
& = (z \cdot f_1 - f_1)(1) - (z \cdot f_1 - f_1)(n_\alpha) \\
& = (f_1(z) - f_1(1)) - (f_1((n_\alpha z n_\alpha^{-1})n_\alpha) - f_1(n_\alpha)) \\
& = \phi(z) - \phi(n_\alpha z n_\alpha^{-1}),
\end{align*}
hence $\phi' = \phi - \phi^\alpha$.
\end{proof}

\begin{remark}
In general, one can prove that $\pi(f) = \sum_{w \in W} \epsilon(w) f(n_w)$ for all $f \in i_B^G$, where $n_w \in \Nc$ is any representative of $w$.
Therefore, \eqref{projeps} is given by $\phi \mapsto \sum_{w \in W} \epsilon(w) \phi^w$.
\end{remark}

We deduce from Lemma \ref{lemm:projalpha} that \eqref{projeps} is surjective if $\Delta = \{\alpha\}$ and $p \neq 2$.
Since \eqref{projeps} factors through \eqref{comp}, we obtain the following result.

\begin{proposition} \label{prop:imalpha}
Assume that $\Delta = \{\alpha\}$ and $p \neq 2$.
Then the image of \eqref{comp} is $\Ext_\Zc^1(1, 1)^{s_\alpha = -1}$.
\end{proposition}

\begin{remark} \label{rema:imalpha}
Assume that $\Delta = \{\alpha\}$ and $p = 2$.
Using Lemma \ref{lemm:projalpha}, we see that \eqref{projeps} is surjective if $G = \GL_2(D)$, so that the result holds true in this case too.
On the other hand, we see that \eqref{projeps} is zero if $G = \SL_2(F)$ and $\car(R) = 2$.
\end{remark}

\subsection{A vanishing result} \label{vanishing}

This section is devoted to the fairly technical proof of the following vanishing result.

\begin{proposition} \label{prop:im}
Assume that $\Gb$ is absolutely almost-simple, $|\Delta| = 2$, and $p \neq 2$.
If $p \nmid |\Zb_{\Gb}|$, then \eqref{comp} is zero.
\end{proposition}

Thanks to Lemma \ref{lemm:imeps}, this result is a consequence of the short exact sequence \eqref{HomZceps} and of Lemmas \ref{lemm:ZZc} and \ref{lemm:HZc} below.

\subsubsection{Devissage} \label{devissage}

Let $\Zcb_\ssc$ be the simply connected covering of the derived subgroup $\Zcb_\der$ of $\Zcb$.
The action of $\Ncb$ on $\Zcb$ by conjugation stabilizes $\Zcb_\der$ and $\Zb_{\Zcb}$.
The induced action of $\Ncb$ on $\Zcb_\der$ lifts uniquely to $\Zcb_\ssc$ by the universal property of the simply connected covering, and stabilizes $\Zb_{\Zcb_\ssc}$.
Moreover, the morphisms of the canonical short exact sequence of algebraic $F$-groups
\begin{equation} \label{Zc}
1 \to \Zb_{\Zcb_\ssc} \to \Zcb_\ssc \times \Zb_{\Zcb} \to \Zcb \to 1
\end{equation}
are $\Ncb$-equivariant.
Note that the action of $\Ncb$ on $\Zb_{\Zcb}$ and $\Zb_{\Zcb_\ssc}$ factors through $W$.

Recall that $\Zcb_\ssc$ is the direct product of its almost-simple factors, each of which is an anisotropic simply connected almost-simple algebraic $F$-groups, hence
\begin{equation*}
\Zc_\ssc \simeq \SL_1(D_1) \times \cdots \times \SL_1(D_r)
\end{equation*}
for some finite-dimensional division algebras $D_1, \ldots, D_r$ over $F$ and $H^1(F, \Zcb_\ssc) = 1$.
These results are due to Kneser \cite{KneI, KneII} when $\car(F) = 0$ and Bruhat-Tits \cite{BTIII} in general.
Thus passing to $F$-points in \eqref{Zc} yields an exact sequence of topological $\Nc$-modules
\begin{equation} \label{ZcF}
1 \to Z_{\Zc_\ssc} \to \Zc_\ssc \times Z_\Zc \to \Zc \to H^1(F, \Zb_{\Zcb_\ssc}) \to H^1(F, \Zb_{\Zcb}).
\end{equation}
By a result of Riehm \cite{Rie}, the abelianization of $\SL_1(D_i)$ is a finite group of prime-to-$p$ order (see the corollary to Theorem 21 and Theorem 7(iii) in loc.\ cit.), hence $\Hom(\Zc_\ssc, R) = 0$.
Therefore, \eqref{ZcF} induces an exact sequence
\begin{equation} \label{HomZc}
0 \to \Hom(H_\Zc, R) \to \Hom(\Zc, R) \to \Hom(Z_\Zc, R)
\end{equation}
where $H_\Zc$ is the topological abelian group defined by
\begin{equation*}
H_\Zc = \ker(H^1(F, \Zb_{\Zcb_\ssc}) \to H^1(F, \Zb_{\Zcb})).
\end{equation*}
Note that $Z_\Zc$ and $H_\Zc$ are both topological $W$-modules.
Taking the $\epsilon$-isotypic components in \eqref{HomZc} yields an exact sequence
\begin{equation} \label{HomZceps}
0 \to \Hom(H_\Zc, R)^\epsilon \to \Hom(\Zc, R)^\epsilon \to \Hom(Z_\Zc, R)^\epsilon.
\end{equation}

\subsubsection{Vanishing}

From now on we assume that $\Gb$ is absolutely almost-simple, $|\Delta| = 2$, and $p \neq 2$.
Moreover, we will write $\Delta = \{\alpha, \beta\}$.
In Table \ref{tab:rank2} below, we list the possible Tits indices \cite{Tits66} for $\Gb$ (extracted from Table 2 in loc.\ cit.) with the corresponding groups $(\Zb_{\Gb_\ssc})_{\Fbar}$ and $(\Zb_{\Zcb_\ssc})_{\Fbar}$.

\begin{table}[ht]
\centering
\caption{Tits indices of relative rank $2$ over a non-archimedean local field}
\label{tab:rank2}
\includegraphics{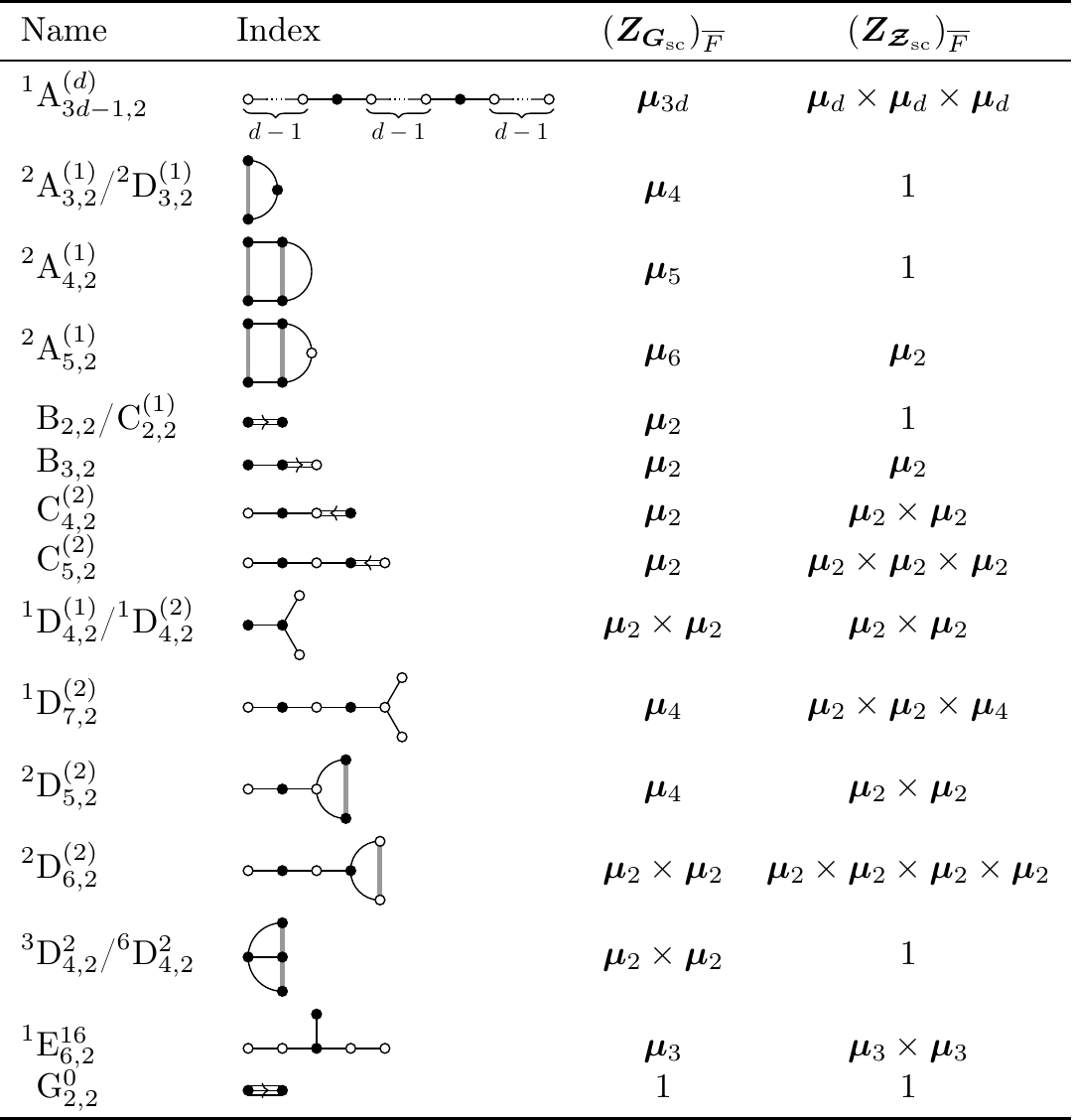}
\end{table} 

\begin{lemma} \label{lemm:ZZc}
Assume that $p \nmid |\Zb_{\Gb}|$.
Then $\Hom(Z_\Zc, R)^\epsilon = 0$.
\end{lemma}

\begin{proof}
We claim that there is a short exact sequence of algebraic $F$-groups of multiplicative type
\begin{equation} \label{ZZc}
1 \to \Zb_{\Gb} \to \Zb_\alpha \times \Zb_\beta \to \Zb_{\Zcb} \to 1.
\end{equation}
Let $X$ be the group of characters of a maximal torus of $\Gb_{\Fbar}$ containing $\Sb_{\Fbar}$ and $\Delta^\abs$ be the set of absolute simple roots in $X$ relative to a Borel subgroup of $\Gb_{\Fbar}$ contained in $\Bb_{\Fbar}$.
The restriction map
\begin{equation*}
r : X \to X^*(\Sb_{\Fbar}) \simeq X^*(\Sb)
\end{equation*}
induces a map
\begin{equation*}
r : \Delta^\abs \to \{\alpha, \beta, 0\}.
\end{equation*}
We let $\Delta_\alpha^\abs = r^{-1}(\{\alpha, 0\})$, $\Delta_\beta^\abs = r^{-1}(\{\beta, 0\})$, and $\Delta_0^\abs = r^{-1}(\{0\})$.
These are the subsets of absolute simple roots corresponding to the Levi subgroups $(\Mb_{\alpha})_{\Fbar}$, $(\Mb_{\beta})_{\Fbar}$, and $\Zcb_{\Fbar}$ respectively.
We let $X_\alpha = \Z[\Delta_\alpha^\abs]$ and $X_\beta = \Z[\Delta_\beta^\abs]$.
Note that $X_\alpha \cap X_\beta = \Z[\Delta_0^\abs]$ and $X_\alpha + X_\beta = \Z[\Delta^\abs]$.
Therefore, \eqref{ZZc} corresponds to the short exact sequence of $\sg_F$-modules
\begin{equation*}
0 \to X / (X_\alpha \cap X_\beta) \to (X / X_\alpha) \oplus (X / X_\beta) \to X / (X_\alpha + X_\beta) \to 0
\end{equation*}
via the duality between algebraic $F$-groups of multiplicative type and finitely generated abelian groups endowed with a continuous action of $\sg_F$ (see \cite[§12g]{AG}).

Now let $\phi : Z_\Zc \to R$ be a continuous group homomorphism such that $s_\alpha \cdot \phi = s_\beta \cdot \phi = -\phi$.
Passing to $F$-points in \eqref{ZZc} yields an exact sequence of topological abelian groups
\begin{equation} \label{ZZcF}
1 \to Z_G \to Z_\alpha \times Z_\beta \to Z_\Zc \to H^1(F, \Zb_{\Gb}) \to H^1(F, \Zb_\alpha) \times H^1(F, \Zb_\beta).
\end{equation}
The actions of $s_\alpha$ and $s_\beta$ on $Z_\alpha$ and $Z_\beta$ respectively are trivial.
Since $p \neq 2$, we deduce that $\phi$ is trivial on $Z_\alpha$ and $Z_\beta$.
Therefore, $\phi$ factors through $\ker(H^1(F, \Zb_{\Gb}) \to H^1(F, \Zb_\alpha) \times H^1(F, \Zb_\beta))$, which is $|\Zb_{\Gb}|$-torsion.
We conclude that $\phi = 0$ since $p \nmid |\Zb_{\Gb}|$.
\end{proof}

\begin{remark} \label{rema:SL3-ZZc}
Assume that $G = \SL_3(D)$ and let $d$ denote the reduced degree of $D$ over $F$.
Then $Z_\Zc \simeq \{(x, y, z) \in (F^*)^3 \mid  (xyz)^d = 1\}$.
A simple computation shows that this isomorphism composed with the map $(F^*)^3 \to F^*$ given by  $(x, y, z) \mapsto xz^{-1}$ induces an isomorphism $Z_\Zc / Z_\alpha Z_\beta \simeq F^* / \mu_d {F^*}^3$.
\end{remark}

\begin{lemma} \label{lemm:HZc}
If $p = 3$ and $\Gb$ is of type $\prescript{1}{}{\Arm}_{3d-1,2}^{(d)}$ assume that $3 \nmid |\Zb_{\Gb}|$.
Then $\Hom(H_\Zc, R)^\epsilon = 0$.
\end{lemma}

\begin{proof}
Since $H_\Zc$ is $|\Zb_{\Zcb_\ssc}|$-torsion (as $H_\Zc \subset H^1(F, \Zb_{\Zcb_\ssc})$), we have $\Hom(H_\Zc, R) = 0$ if $p \nmid |\Zb_{\Zcb_\ssc}|$.
Looking at Table \ref{tab:rank2}, we see that $p \mid |\Zb_{\Zcb_\ssc}|$ if and only if either $\Gb$ is of type $\prescript{1}{}{\Arm}_{3d-1,r}^{(d)}$ with $p \mid d$, or $\Gb$ is of type $\prescript{1}{}{\Erm}_{6,2}^{16}$ and $p = 3$ (recall that $p \neq 2$).

\begin{enumerate}[wide, label=\emph{(\alph*}), labelindent=0pt]

\item \label{case:SL3-HZc}
Assume that $G = \SL_3(D)$ and let $d$ denote the reduced degree of $D$ over $F$.
Then $\Zb_{\Zcb_\ssc} \simeq (\mub_d)^3$ and $\Zb_{\Zcb} \simeq \ker((\Gm)^3 \xrightarrow{\scriptscriptstyle\prod} \Gm \xrightarrow{\scriptscriptstyle(-)^d} \Gm)$, with $W$ acting via the isomorphism $W \simeq S_3$.
We let $\Tb = \ker((\Gm)^3 \xrightarrow{\scriptscriptstyle\prod} \Gm)$ be the maximal subtorus of $\Zb_{\Zcb}$.
There is a short exact sequence of algebraic $F$-groups of multiplicative type
\begin{equation*}
1 \to \Tb \to \Zb_{\Zcb} \xrightarrow{\scriptscriptstyle\prod} \mub_d \to 1.
\end{equation*}
Since $\Tb$ is a split torus, the product map induces a $W$-equivariant isomorphism
\begin{equation*}
H^1(F, \Zb_{\Zcb}) \iso H^1(F, \mub_d).
\end{equation*}
We deduce that the action of $W$ is trivial on $H^1(F, \Zb_{\Zcb})$ and we obtain a $W$-equivariant isomorphism
\begin{equation} \label{HZc}
H_\Zc \simeq \{(x, y, z) \in (F^* / {F^*}^d)^3 \mid xyz = 1\}
\end{equation}
with $W$ acting via the isomorphism $W \simeq S_3$.
Now there is an isomorphism
\begin{equation*}
\Hom(H_\Zc, R)^\epsilon \simeq \Hom(H_\Zc / H_\Zc^\epsilon, R)
\end{equation*}
where $H_\Zc^\epsilon$ is the subgroup of $H_\Zc$ generated by the elements of the form $h(s_\alpha \cdot h)$ and $h(s_\beta \cdot h)$ with $h \in H_\Zc$.
The subgroup of $(F^* / {F^*}^d)^3$ corresponding to $H_\Zc^\epsilon$ via \eqref{HZc} is generated by the elements of the form $(x, x, x^{-2})$ and $(x^{-2}, x, x)$ with $x \in F^* / {F^*}^d$.
A simple computation shows that \eqref{HZc} composed with the map $(F^* / {F^*}^d)^3 \to F^* / {F^*}^d$ given by $(x, y, z) \mapsto xz^{-1}$ induces an isomorphism
\begin{equation*}
H_\Zc / H_\Zc^\epsilon \simeq (F^* / {F^*}^d) / ({F^*}^3 / {F^*}^{3d}) \simeq
\begin{cases}
F^* / {F^*}^3 & \text{if $3 \mid d$,} \\
1 & \text{otherwise.}
\end{cases}
\end{equation*}
Therefore, $\Hom(H_\Zc, R)^\epsilon = 0$ if $3 \nmid d$ or $p \neq 3$.

\item
Assume that $\Gb$ is of type $\prescript{1}{}{\Arm}_{3d-1,2}^{(d)}$, i.e.\ $G_\ssc = \SL_3(D)$ with $D$ of reduced degree $d$ over $F$.
Let $\Zcb'$ denote the inverse image of $\Zcb$ in $\Gb_\ssc$.
Since the homomorphism $\Zcb_\ssc \to \Zcb$ factors through $\Zcb'$, there is an inclusion $H_{\Zc'} \subset H_\Zc$, hence an exact sequence
\begin{equation*}
0 \to \Hom(H_\Zc / H_{\Zc'}, R)^\epsilon \to \Hom(H_\Zc, R)^\epsilon \to \Hom(H_{\Zc'}, R)^\epsilon.
\end{equation*}
By \ref{case:SL3-HZc}, we know that $\Hom(H_{\Zc'}, R)^\epsilon = 0$ if $3 \nmid d$ or $p \neq 3$, and that $W$ acts trivially on $H^1(F, \Zb_{\Zcb'})$, so that $W$ acts trivially on the subquotient $H_\Zc / H_{\Zc'}$, hence $\Hom(H_\Zc / H_{\Zc'}, R)^\epsilon = 0$ since $p \neq 2$.
Therefore, $\Hom(H_\Zc, R)^\epsilon = 0$ if $3 \nmid d$ or $p \neq 3$.

\item
Assume now that $p = 3$ and write $d = 3^r d'$ with $3 \nmid d'$, so that $\Zb_{\Gb_\ssc} \simeq \mub_{3^{r+1}} \times \mub_{d'}$.
Since $3 \nmid |\Zb_{\Gb}|$ by assumption, the homomorphism $\Gb_\ssc \to \Gb$ factors through $\Gb_\ssc / \mub_{3^{r+1}}$.
Let $\Zcb''$ denote the inverse image of $\Zcb$ in $\Gb_\ssc / \mub_{3^{r+1}}$.
There are inclusions $H_{\Zc'} \subset H_{\Zc''} \subset H_\Zc$, hence an exact sequence
\begin{equation*}
0 \to \Hom(H_\Zc / H_{\Zc''}, R)^\epsilon \to \Hom(H_\Zc, R)^\epsilon \to \Hom(H_{\Zc''}, R)^\epsilon.
\end{equation*}
By \ref{case:SL3-HZc}, we know that $W$ acts trivially on $H^1(F, \Zb_{\Zcb'})$.
Thus $W$ acts trivially on the subquotient $H_\Zc / H_{\Zc''}$, hence $\Hom(H_\Zc / H_{\Zc''}, R)^\epsilon = 0$ since $p \neq 2$.
Let $\Tb'$ (resp. $\Tb''$) denote the maximal subtorus of $\Zb_{\Zcb'}$ (resp. $\Zb_{\Zcb''}$).
There is a commutative diagram of algebraic $F$-groups of multiplicative type
\begin{equation*}
\xymatrix{
\mub_3 \ar@{^(->}[r] \ar@{^(->}[d] & \mub_{3^{r+1}} \ar@{->>}[r] \ar@{^(->}[d] & \mub_{3^r} \ar@{^(->}[d] \\
\Tb' \ar@{^(->}[r] \ar@{->>}[d] & \Zb_{\Zcb'} \ar@{->>}[r] \ar@{->>}[d] & \mub_d \ar@{->>}[d] \\
\Tb'' \ar@{^(->}[r] & \Zb_{\Zcb''} \ar@{->>}[r] & \mub_{d'}
}
\end{equation*}
whose rows and columns are exact.
Since $\Tb''$ is a split torus, the surjection $\Zb_{\Zcb''} \twoheadrightarrow \mub_{d'}$ induces an isomorphism
\begin{equation*}
H^1(F, \Zb_{\Zcb''}) \iso H^1(F, \mub_{d'}).
\end{equation*}
Thus $H^1(F, \Zb_{\Zcb''})$ is $3$-torsion-free.
We deduce that the $3$-torsion subgroup $H_{\Zc''}[3^\infty]$ of $H_{\Zc''}$ is equal to that of $H^1(F, \Zb_{\Zcb_\ssc})$, hence a $W$-equivariant isomorphism
\begin{equation*}
H_{\Zc''}[3^\infty] \simeq (F^* / {F^*}^{3^r})^3
\end{equation*}
with $W$ acting on the right-hand side via the isomorphism $W \simeq S_3$.
A simple computation as in \ref{case:SL3-HZc} shows that $\Hom(H_{\Zc''}[3^\infty], R)^\epsilon = 0$ (using the fact that $p \neq 2$).
Since $H_{\Zc''} / H_{\Zc''}[3^\infty]$ is $d'$-torsion, we deduce that $\Hom(H_{\Zc''}, R)^\epsilon = 0$.
Therefore, $\Hom(H_\Zc, R)^\epsilon = 0$.

\item
Assume that $\Gb$ is of type $\prescript{1}{}{\Erm}_{6,2}^{16}$ and $p = 3$.
Let $\alpha$ denote the relative simple root corresponding to the distinguished vertex of the Tits index of $\Gb$ with $3$ neighbors and let $\beta$ denote the other relative simple root (see Table \ref{tab:rank2}).
Then $\Zb_{\Zcb_\ssc} \simeq \mub_3 \times \mub_3$ and the action of $s_\alpha$ permutes the two copies of $\mub_3$ whereas $s_\beta$ acts trivially.
Thus $s_\beta$ acts trivially on $H_\Zc$ so that $\Hom(H_\Zc, R)^\epsilon = 0$ since $p \neq 2$.
\qedhere

\end{enumerate}
\end{proof}

\begin{remark} \label{rema:im}
Keep the assumptions on $\Gb$ but assume that $p = 2$.
\begin{enumerate}
\item The image of \eqref{comp} is zero in the following cases: $G = \SL_3(D)$; $\Gb$ is quasi-split and $2 \nmid |\Zb_{\Gb}|$; and $\Gb$ is of type $\prescript{1}{}{\Erm}_{6,2}^{16}$.
Indeed, we see that any extension in the image of \eqref{comp} is trivial on $Z_\Zc$ using Lemma \ref{lemm:imev1}(1) and the fact that $Z_\Zc / Z_\alpha Z_\beta$ is $2$-torsion-free (see Remark \ref{rema:SL3-ZZc} in the case $G = \SL_3(D)$ and use \eqref{ZZcF} together with the fact that $2 \nmid |\Zb_{\Gb}|$ in the other cases), and we have $\Hom(H_\Zc, R)^\epsilon = 0$ (see \ref{case:SL3-HZc} in the proof of Lemma \ref{lemm:HZc} in the case $G = \SL_3(D)$ and use the fact that $H_\Zc$ is $2$-torsion-free since $2 \nmid |\Zb_{\Zcb_\ssc}|$ in the other cases).
\item In general, we deduce from \eqref{HomZceps} and the proofs of Lemmas \ref{lemm:ZZc} and \ref{lemm:HZc} that the image of \eqref{comp} is killed by the highest power of $2$ dividing the product of the exponents of the groups $\Zb_{\Gb}$ and $\Zb_{\Zcb_\ssc}$.
Looking at Table \ref{tab:rank2}, we see that if $\Gb$ is not of type $\prescript{1}{}{\Arm}_{3d-1,2}^{(d)}$, then this power is at most $16$ (and even $4$ if $\Gb$ is adjoint).
\end{enumerate}
\end{remark}

\subsection{Conclusion} \label{conclusion}

Finally, we compute the $R$-module $\Ext_G^1(1, \St)$.

\begin{proposition} \label{prop:basealpha}
Assume that $\Delta = \{\alpha\}$ and $p \neq 2$.
If the adjoint action of $\Zc$ on $\Ubar \setminus \{1\}$ is transitive, then there is an isomorphism
\begin{equation*}
\Ext_G^1(1, \St) \iso \Ext_{\Zc}^1(1, 1)^{s_\alpha = -1}.
\end{equation*}
\end{proposition}

\begin{proof}
This follows from Propositions \ref{prop:inj} and \ref{prop:imalpha}.
\end{proof}

\begin{remark} \label{rema:basealpha}
If $G = \GL_2(D)$, then the assumptions of Proposition \ref{prop:basealpha} are satisfied when $p \neq 2$, but the result holds true when $p = 2$ (see Remark \ref{rema:imalpha}).
Since the derived subgroup of $D^*$ is $\SL_1(D)$, the reduced norm $\operatorname{Nrd} : D^* \twoheadrightarrow F^*$ induces an isomorphism $\Hom(F^*, R) \iso \Hom(D^*, R)$.
Using Lemma \ref{lemm:product}(2), we conclude that if $G \simeq \GL_2(D) \times (D^*)^r$, then there is an isomorphism
\begin{equation*}
\Ext_G^1(1, \St) \simeq \Hom(F^*, R).
\end{equation*}
When $D = \Q_p$, this result is originally due to one of us (see \cite[Th.\ VII.4.18]{Col}).
\end{remark}

\begin{proposition} \label{prop:base}
Assume that $|\Delta| \geq 2$ and $p \neq 2$.
Then $\Ext_G^1(1, \St) = 0$.
\end{proposition}

\begin{proof}
If $|\Delta| \geq 3$ or $\Delta = \{\alpha, \beta\}$ with $\alpha \perp \beta$, then \eqref{comp} is zero by Lemma \ref{lemm:imev1}(2) and the result follows from Proposition \ref{prop:inj}.
Assume that $\Delta = \{\alpha, \beta\}$ with $\alpha$ and $\beta$ non-orthogonal.
Let $\Gb_\ssc$ be the simply connected covering of the derived subgroup $\Gb_\der$ of $\Gb$.
Recall that $\Gb_\ssc$ is the direct product of its almost-simple factors.
Thus $\Gb_\ssc = \Gb_\is \times \Gb_\anis$ where $\Gb_\is$ (resp.\ $\Gb_\anis$) is the direct product of the isotropic (resp.\ anisotropic) almost-simple factors of $\Gb_\ssc$.
By assumption, $\Gb_\is$ is almost-simple of relative rank $2$.
Using Lemmas \ref{lemm:product}(2) and \ref{lemm:isogeny}(1) with the canonical isogeny $\Gb_\is \times \Gb_\anis \times \Zb_{\Gb} \twoheadrightarrow \Gb$, we can reduce to the case where $\Gb = \Gb_\is$.
Thus we can assume that $\Gb$ is almost-simple of relative rank $2$.
Replacing $F$ by a finite separable extension, we can assume that $\Gb$ is absolutely almost-simple of relative rank $2$.
Finally, using Lemma \ref{lemm:isogeny}(2) with the canonical isogeny $\Gb \twoheadrightarrow \Gb_\mathrm{ad}$, we can assume that $\Gb$ is adjoint, i.e.\ $\Zb_{\Gb} = 1$.
Then the result follows from Propositions \ref{prop:inj} and \ref{prop:im}.
\end{proof}

\begin{remark} \label{rema:base}
Assume that $|\Delta| \geq 2$ and $p = 2$.
Proceeding as in the proof of Proposition \ref{prop:base} and using Remark \ref{rema:im}, we see that $\Ext_G^1(1, \St) = 0$ if $G$ is a Levi subgroup of $\GL_n(D)$ or $\Gb$ is quasi-split.
Moreover, we see that $\Ext_G^1(1, \St)$ is $4$-torsion in general.
\end{remark}

\section{The Kottwitz set $B(G)$}\label{BG}

This section is devoted to a brief review of basic facts concerning the theory of $\sigma$-conjugacy classes, due to Kottwitz.
We fix in the sequel a finite extension $F$ of $\Qp$ and an algebraic closure 
$ \overline{F}$ of $F$. Let $C$ be the completion of $ \overline{F}$, let $F^{\rm nr}\subset  \overline{F}$ be the maximal unramified extension of $F$ and finally let $\overline{\breve{F}}\simeq \breve{F}\otimes_{F^{\rm nr}} \overline{F}$ be the algebraic closure of 
$\breve{F}$ inside $C$. Let $\sigma$ be the (relative to $F$) Frobenius automorphism on 
$F^{\rm nr}$ and $\breve{F}$. Let 
$\sg_F={\rm Gal}( \overline{F}/F)$, with inertia subgroup $I_F=\sg_{F^{\rm nr}}\simeq \sg_{\breve{F}}:={\rm Gal}(\overline{\breve{F}}/\breve{F})$ and Weil subgroup $
W_F$. Finally, let $G$ be a connected reductive group defined over $F$.

\subsection{Various incarnations of $B(G)$} 
  In  \cite{K1} and \cite{K2} Kottwitz attached to $G$ a pointed set $B(G)$. This set  has now several incarnations which nicely complement each other. We will review them briefly. 
  
   \subsubsection{Via Galois cohomology}  
     The first definition is a cohomological one \cite[1.7]{K1}, \cite[1.1-1.5]{K2}: $$B(G):=H^1(\sigma^{\Z}, G(\breve{F}))$$
      is 
     the (pointed) set of {\it $\sigma$-equivalence classes} in $G(\breve{F})$, two elements 
     $b,b'\in G(\breve{F})$ being in the same equivalence class if $b'=gb\sigma(g)^{-1}$ for some 
     $g\in G(\breve{F})$, or equivalently if the elements $b\sigma, b^{\prime}\sigma\in G(\breve{F})\rtimes\langle\sigma\rangle$ are conjugate under $G(\breve{F})$. 
     We will write $[b]=\{gb\sigma(g)^{-1}|\, g\in G(\breve{F})\}\in B(G)$ for  the 
     $\sigma$-equivalence class of $b\in G(\breve{F})$. 
     
      We have a natural isomorphism\footnote{Here and below, to simplify the notation, we will write 
        $H^1(\breve{F}, G)$ etc.\ for the Galois cohomology $H^1(\sg_{\breve{F}}, G(\overline{\breve{F}}))$.} 
       $$B(G)\stackrel{\sim}{\to}H^1(W_F, G(\overline{\breve{F}}))$$
       induced by the inflation-restriction sequence associated to
       the exact sequence of topological groups 
     $
     1\to I_F\to W_{F}\to \sigma^{\Z}\to 1
     $ and the vanishing $H^1({\breve{F}}, G)=0$, the latter being a consequence of Steinberg's theorem
 \cite[Th.\ 1.9]{RS}, since $G$ is connected and 
        $\breve{F}$ has cohomological dimension $1$.
      The natural map, obtained using the restriction $W_F\to \sg_F$ and the inclusion $G(\overline{F})\subset G(\overline{\breve{F}})$,
       $$H^1(F,G)\to H^1(W_F, G(\overline{\breve{F}}))\stackrel{\sim}{\leftarrow}B(G)$$ is injective. 
   
     \begin{example}[Torus] \label{Newton}
     A fundamental result of Kottwitz \cite[Sec.\ 2]{K1} gives a 
     functorial isomorphism $X_*(T)_{\sg_F}\stackrel{\sim}{\to} B(T)$ for 
     $F$-tori $T$, uniquely pinned down by the requirement that in the induced isomorphism 
     $B({\mathbb G}_m)\simeq \Z$ the element $1\in \Z$ corresponds to the $\sigma$-conjugacy class in $\breve{F}^*$ consisting of elements with normalized valuation $1$. Concretely, the isomorphism $X_*(T)_{\sg_F}\stackrel{\sim}{\to} B(T)$ sends the class of $\mu\in X_*(T)$ to 
     $[N_{E/E_0}(\mu(\pi_E))]$, where
     $E$ is a finite Galois extension of $F$ inside $\overline{F}$ splitting 
     $T$, $\pi_E$ is a uniformizer of $E$, 
     $E_0$ is the maximal subfield of $E$ unramified over $F$ and finally 
     $N_{E/E_0}: T(E)\to T(E_0)$ is the norm map.
     
          \end{example}
          
 \subsubsection{Via $G$-isocrystals}    
     The second incarnation (see \cite[Sec.\ 3]{K1}) of $B(G)$ is as the set of isomorphism classes of {\it $G$-isocrystals} (relative to $\breve{F}/F$), i.e.,  exact, faithful
     $F$-linear tensor functors\footnote{$G$-isocrystals  can be defined for any linear algebraic group $G$ over $F$. In that case one adds  an assumption that the defining functor  is  strictly compatible with the fiber functors. If  the group $G$ is connected, as it is the case in this paper,  this assumption is not necessary by the vanishing theorem of Steinberg \cite[Lemma 9.1.5]{DOR}. } ${\rm Rep}_F(G)\to {\rm Mod}_{\breve{F}}(\varphi)$, where 
${\rm Rep}_F(G)$ is the category of finite dimensional algebraic
     $F$-rep\-re\-sen\-ta\-tions of $G$ and ${\rm Mod}_{\breve{F}}(\varphi)$ is the category of isocrystals relative to 
     $\breve{F}/F$.      
     
       There is an equivalence (even an  isomorphism) of categories between the groupoid of $G$-isocrystals and the groupoid having as objects 
     the set $G(\breve{F})$, the set of morphisms $b\to b'$ being the set of $g\in G(\breve{F})$ with 
     $gb\sigma(g)^{-1}=b'$. Specifically, every $b\in G(\breve{F})$ yields a 
     $G$-isocrystal $N_b$:
   $$ N_b: {\rm Rep}_F(G)\to  {\rm Mod}_{\breve{F}}(\varphi),\quad 
    (V,\rho)\mapsto (V\otimes_F \breve{F}, b\sigma:=\rho(b) ({\rm id}_V\otimes \sigma)),
    $$ whose isomorphism class depends only on 
     $[b]$. Sending $[b]$ to the isomorphism class of $N_b$ yields a bijection between 
     $B(G)$ and the set of isomorphism classes of $G$-isocrystals.
     For instance, $B(\mathbb{GL}_n)$ is identified with the set of isomorphism classes of $n$-dimensional isocrystals \cite[Rem.\ 3.4 (ii)]{RR}. 
     
        \subsubsection{Via the Fargues-Fontaine curve} 
        This incarnation of $B(G)$, due to Fargues \cite{FG}, will not be used in the rest of the paper, but is particularly appealing and we feel that it gives a better understanding of many of the constructions to come. We fix an embedding of the residue field $\mathbf{F}_q$ of $F$ into the residue field
      $\overline{\mathbf{F}}_p$ of $C$ and consider the Fargues-Fontaine curve $X=X_{F,\overline{\mathbf{F}}_p}$ over ${\rm Spec}(F)$ attached to $F$ and $\overline{\mathbf{F}}_p$. 
        There exists an exact, faithful, $F$-linear tensor functor 
      $$\mathcal{E}: {\rm Mod}_{\breve{F}}(\varphi)\to {\rm Bun}_X,$$
      where ${\rm Bun}_X$ is the category of vector bundles on $X$. While 
      {\it not} fully faithful, this functor is essentially surjective and induces an equivalence of categories between 
      the isoclinic isocrystals and the semistable vector bundles on $X$ as well as a bijection
      $$\mathcal{E}: |{\rm Mod}_{\breve{F}}(\varphi)|\simeq |{\rm Bun}_X|$$ between the sets of 
        isomorphism classes of the corresponding objects\footnote{Despite its innocuous-looking character, this is one of the most difficult results in the book of Fargues and Fontaine \cite{FF0}.}.
          Every $b\in G(\breve{F})$ yields therefore an exact, faithful, $F$-linear tensor functor 
      $$\mathcal{E}_b=\mathcal{E}\circ N_b: {\rm Rep}_F(G)\to {\rm Bun}_X,$$ i.e., a 
      $G$-bundle on $X$ in the Tannakian sense\footnote{There is a natural equivalence
       between the category of $G$-bundles on $X$ 
      and 
      the category of $G$-torsors on $X$ locally trivial for the \'etale topology: if $Y$ is $G$-torsor \'etale locally trivial,
      we obtain a $G$-bundle by sending $(V,\rho)\in {\rm Rep}_F(G)$ to $Y\times_{G,\rho} V$; conversely, each $G$-bundle $\omega$ yields a locally trivial $G$-torsor ${\rm Isom}^{\otimes}(\omega_{\rm can},\omega)$, where $\omega_{\rm can}(V,\rho)=V\otimes_F \mathcal{O}_X$.}
       and we have  the following beautiful result:
                  \begin{theorem}[Fargues, {\cite[Th.\ 5.1]{FG}}]
       The construction $b\to \mathcal{E}_b$ yields bijections of pointed sets 
       $$B(G)\simeq |{\rm Bun}_G|\simeq H^1_{\eet}(X,G),$$
       where $|{\rm Bun}_G|$ is the set of isomorphism classes of $G$-bundles on $X$.
      \end{theorem}

          \subsection{The structure of  $B(G)$}
        In order to describe $B(G)$ Kottwitz defined two rather technical but very important maps, which we briefly review. 
        
  \subsubsection{The Newton map}     
          Let $b\in G(\breve{F})$ and let $N_b$ be the associated $G$-isocrystal. If 
          $\mathbb{D}$ is the pro-torus over $F$ with character group $\Q$ (the ``slope torus'')
          there is a unique morphism $$\nu_b: \mathbb{D}_{\breve{F}}\to G_{\breve{F}}$$ such that for all 
          $(V,\rho)\in {\rm Rep}_F(G)$ the composition $\rho\circ \nu_b: \mathbb{D}_{\breve{F}}\to {\rm GL}(V\otimes_F \breve{F})$ corresponds to the (Dieudonn\'e-Manin) slope decomposition of $N_b(V)$, considered as a $\Q$-grading on $V\otimes_F \breve{F}$. The homomorphism $\nu_b$ is the {\it Newton or slope map attached to $b$} \cite[Sec.\ 4]{K1}. It satisfies $$\nu_{\sigma(b)}=\sigma(\nu_b), \, \nu_b={\rm Int}(b)\circ\sigma(\nu_b),\,
          \nu_{gb\sigma(g)^{-1}}={\rm Int}(g)\circ\nu_b=g\nu_b g^{-1},
          $$
          hence the $G(\breve{F})$-conjugacy class $[\nu_b]$ of $\nu_b$ depends only on $[b]\in B(G)$ and is 
          $\sigma$-invariant, thus
           $$[\nu_b]\in\mathcal{N}(G):=[{\rm Hom}_{\breve{F}}(\mathbb{D}_{\breve{F}}, G_{\breve{F}})/{\rm Int } \,G(\breve{F}) ]^{\sigma=1}.$$
           The elements of $\mathcal{N}(G)$ are called 
          {\em Newton vectors}. 
          
          \begin{remark} By choosing 
          a maximal torus $T$ of $G$ defined over $F$, we obtain a more concrete description
          $\mathcal{N}(G)\simeq (X_*(T)_{\Q}/W)^{\sg_F}$, where 
          $W$ is the absolute Weyl group of $G$ with respect to $T$. 
          If we also choose a basis $\Delta$ for the root system, 
          each element of 
          $\mathcal{N}(G)$ has a unique representative 
          in $X_*(T)_{\Q}^+:=\{\nu \in X_*(T)_{\Q}|\, \langle \nu, \alpha\rangle\geq 0\, \forall\, 
          \alpha\in \Delta\}$. One infers from this a partial order on 
          $\mathcal{N}(G)$, induced from the natural partial order on 
          $X_*(T)_{\Q}^+$ (for which $\nu_1\leq \nu_2$ if 
          $\nu_2-\nu_1\in \sum_{\alpha\in \Delta} \Q_{\geq 0} \alpha^{\vee}$).
          Pulling back along the Newton map we deduce a partial order on 
          $B(G)$. 
          \end{remark}

          Sending 
          $[b]\in B(G)$ to $[\nu_b]\in \mathcal{N}(G)$ yields the 
           {\it Newton map} 
          $$\nu: B(G)\to \mathcal{N}(G).$$ 
                    Kottwitz proves \cite[4.5]{K1} that 
          $\nu_b$ is trivial if and only if $[b]\in H^1(F,G)\subset B(G)$.
          
          \begin{example} If $G=T$ is an $F$-torus then 
          $\mathcal{N}(T)\simeq\Hom_{F}(\mathbb{D},T)\simeq X_*(T)_{\Q}^{\sg_F}$ and 
          $\nu_b$ (for $b\in T(\breve{F})$) is the unique element of $X_*(T)^{\sg_F}_{\Q}$ such that 
          $\langle \lambda, \nu_b\rangle=v_F(\lambda(b))$
          for $\lambda\in X^*(T)^{\sg_F}$, where $v_F$ is the normalized valuation on $\breve{F}$. 
                    The  composite map\footnote{Cf.\ Example \ref{Newton} for the isomorphism in the displayed formula.} $$X_*(T)\to X_*(T)_{\sg_F}\stackrel{\sim}{\to} B(T)\stackrel{\nu}{\to} X_*(T)^{\sg_F}_{\Q}$$ sends 
          $\mu\in X_*(T)$ to its Galois average $\mu^{\diamond}$.
                   We have an exact sequence 
          $
          0\to  H^1(F,G)\to B(T) \stackrel{\nu}{\to} X_*(T)^{\sg_F}_{\Q}.
          $
               
          \end{example}

          \begin{example} 
           For $G=\mathbb{GL}_n$ the partially ordered set 
          $\sn(G)$ is naturally in bijection with 
            $$(\mathbb{Q}^n)_+:=\{(x_1,\dots, x_n)\in \mathbb{Q}^n|\, x_1\geq\dots\geq x_n\},$$ endowed with the usual dominance order, for which 
            $(y_1,\dots,y_n)\leq (x_1,\dots,x_n)$ if and only if 
            $x_1+\cdots+x_i\geq y_1+\cdots+y_i$ for all $i$, with equality for $i=n$.
           An element $x=(x_1^{(n_1)},\dots, x_r^{(n_r)})\in\sn(G)$ is thus given by rational numbers 
           $
           x_1 > x_2 >\cdots >x_r
           $
           and multiplicities $n_1,\ldots, n_r\in \Z_{>0}$ such that $
           \sum_{i=1}^{r}n_i=n.
           $ 
           If $[b]\in B(G)$ and $x_1,\dots,x_r$ are the slopes of the isocrystal $N_b$ (in decreasing order), then
           $\nu([b])=(x_1^{(n_1)},\dots, x_r^{(n_r)})$, where $n_i$ are the dimensions of the isotypic parts of $N_b$. The Newton map 
            $\nu$ is injective with image consisting of those $x=(x_1^{(n_1)},\dots,x_r^{(n_r)})\in(\mathbb{Q}^n)_+$ such that $
           n_ix_i\in\Z,  i=1,\ldots,r 
           $ (see \cite[Ex.\ 1.19]{RR}).           \end{example}

              \subsubsection{The Kottwitz map $\kappa$}        
             Let $\pi_1(G)$ 
          be the (Borovoi) algebraic fundamental group of $G$ \cite{Bor}, \cite[1.13]{RR}. 
          It is a finitely generated discrete $\sg_F$-module, functorial and exact in $G$, and isomorphic (as abelian group) to 
          $X_*(T)/\sum_{\alpha\in\Phi(G,T)}\Z\alpha^{\vee}$, where $T$ is a maximal torus of $G_{\overline{F}}$ and $\Phi(G,T)$ is the set of roots of $T$, $\alpha^{\vee}$ being the co-root corresponding to 
          $\alpha\in \Phi(G,T)$. For example, 
          $\pi_1({\mathbb{GL}}_n)\simeq {\Z}$ and $\pi_1(T)\simeq X_*(T)$ for an 
          $F$-torus $T$. We also note \cite[1.14]{RR} that $\pi_1(G)\simeq \pi_1(G^{\prime})$ for any inner form $G^{\prime}$ of $G$, and that $
           \pi_1(G)\stackrel{\sim}{\to} \pi_1(G_{\rm ab}:=G/G_{\rm der})$ whenever 
           $G_{\rm der}$ is simply connected. 
          
        By a theorem of Kottwitz (see \cite[Sec.\ 6]{K2}), \cite[Th.\ 1.15]{RR})
        there is a unique natural  map\footnote{Kottwitz formulated his theorem in \cite{K1}, \cite{K2}  in terms of the center $Z(\wh{G})$ of the Langlands dual group. The formulation we present here in terms of the algebraic fundamental group is due to  Rapoport-Richartz \cite{RR}. It has better functoriality properties than the original one.} $$\kappa: B(G)\to \pi_1(G)_{\sg_F}$$ making the following diagram commute:
       $$
       \xymatrix{\breve{F}^{*} \ar[r]^{v_F}\ar[d] & \Z\ar@{=}[d]\\
       B(\mathbb{G}_m) \ar[r]^-{\kappa}&  \pi_1(\mathbb{G}_m)_{\sg_F}.
       }
       $$
   For instance for $G=\mathbb{GL}_n$ the map $\kappa$ sends $ [b]\in B(G)$ to $v_F(\det b)$.
    In general, the induced map $(\nu, \kappa): B(G)\to \mathcal{N}(G)\times \pi_1(G)_{\sg_F}$ is injective and 
there exists a natural map of exact sequences
   $$
   \xymatrix{
  H^1(F,G)\ar@{^(->}[r] \ar[d]^{\kappa}_{\wr}&  B(G)\ar[r]^{\nu}\ar[d]^{\kappa}  & \sn(G)\ar[d]^{\delta}\\
 \pi_1(G)_{\sg_F, {\rm tors}} \ar@{^(->}[r] &  \pi_1(G)_{\sg_F} \ar[r]^-{\mu\mapsto\mu^{\diamond}}   & \pi_1(G)^{\sg_F}_{\Q}
   }
   $$

         Fargues \cite[Sec.\ 8]{FG}, inspired by Labesse's \cite{Lab} reinterpretation of the constructions of Kottwitz, also gave a geometric interpretation of the Kottwitz map $\kappa$ using the Fargues-Fontaine curve $X$:
        \begin{enumerate}
        \item Using cohomology theory of  crossed modules \cite{Lab}, he  defined a map 
        $\kappa^F: H^1_{\eet}(X,G)\to \pi_1(G)_{\sg_F}
        $ which agrees with the Kottwitz map $\kappa$ after the identification $B(G)\stackrel{\sim}{\to} H^1_{\eet}(X,G)$. 
        \item Then, using the  universal $G$-torsor and  cohomology of cross modules again, he defined 
       a Chern class map $c_1^G:  H^1_{\eet}(X,G)\to \widehat{\pi_1(G)}_{\sg_F}$ (the profinite completion) and showed that
       $$\kappa^F(b)=-c_1^G(\mathcal{E}_{b}),\quad b\in G(\breve{F}).$$ 
       \end{enumerate}

          \subsubsection{Automorphism groups}         
           Let $b\in G(\breve{F})$ and let $J_b$ be the automorphism group of the $G$-isocrystal $N_b$, i.e., the connected reductive group over $F$ such that for any 
           $F$-algebra $R$ 
              $$J_b(R):=\{g\in G(\breve{F}\otimes_FR)| \, gb\sigma(g)^{-1}=b\}.$$ 
              Clearly, if $b=1$ then $J_b=G$. 
              Suppose that $G$ is quasi-split and let $M_{b}$ be the centralizer of $\nu_{b}$, a Levi component of an $F$-parabolic of $G$. By results of Kottwitz \cite[Sec.\ 6]{K1}
              the group $J_b$ is an inner form of $M_b$.

           \subsubsection{Basic and decent elements}    
                     The subset $B(G)_{\rm basic}$ of {\it basic elements of $B(G)$} consists of those
          $[b]\in B(G)$ for which $\nu_b$ factors through $Z_{\breve{F}}$, where $Z$ is the center of $G$. If $[b]\in B(G)_{\rm basic}$ the Newton map $\nu_b: {\mathbb D}_{\breve{F}}\to Z_{\breve{F}}$ is defined over $F$, since 
          its conjugacy class is $\sigma$-invariant. It follows from results of Kottwitz (see 
          \cite{K1} and \cite{RR}) that an element $[b]\in B(G)$ is in $B(G)_{\rm basic}$ if and only if it satisfies one of the following equivalent conditions:
                      
            $\bullet$ The automorphism group $J_b$ is an inner form of $G$. 
            
            $\bullet$ $[b]$ is a minimal element for the partial order on $B(G)$.
            
            $\bullet$ The associated $G$-bundle $\mathcal{E}_b$ on the Fargues-Fontaine curve $X$ is semistable (i.e., $\mathcal{E}_b({\rm Lie}(G), {\rm Ad})$ is a semistable vector bundle on $X$, or equivalently\footnote{This uses the deep fact that semistable vector bundles on $X$ are stable under tensor product.} $\mathcal{E}_b(V,\rho)$ is semistable for any homogeneous representation $\rho$ of $G$). 
          For $G=\mathbb{GL}_n$ the set $B(G)_{\rm basic}$ corresponds to that of isoclinic isocrystals of dimension $n$.
            
          Kottwitz \cite[Ch.\ 5]{K1} proved that the Kottwitz map $\kappa$ induces a bijection
       $$\kappa: B(G)_{\rm basic}\simeq \pi_1(G)_{\sg_F}.$$
       Since $\kappa$ identifies $H^1(F,G)$ and $\pi_1(G)_{\sg_F, {\rm tors}}$, it follows that $H^1(F,G)\subset B(G)_{\rm basic}$. It also follows that when $G_{\rm der}$ is simply connected the map $G\to G_{\rm ab}:=G/G_{\rm der}$ induces a bijection $B(G)_{\rm basic}\stackrel{\sim}{\to} B(G_{\rm ab})$. In particular, $B(G)_{\rm basic}$ is trivial if $G$ is semisimple and simply connected.

         Let now $F=\qp$ and let $s$ be a positive integer. We say that 
         $b\in G(\breve{\Q}_p)$ is {\it $s$-decent} if $s{\nu}_b$ factors through the quotient $\mathbb{G}_{m \breve{\Q}_p}$ of 
         $\mathbb{D}_{\breve{\Q}_p}$ and we have an equality in $G(\breve{\Q}_p)\rtimes \sigma^{\mathbf{Z}}$:
         $$(b\sigma)^{s}=(s{\nu}_b)(p)\sigma^s.$$
        By \cite[Cor.\ 1.9]{RZ} this implies that ${\nu}_b$ is defined over $\Q_{p^s}$ and 
         $b\in G(\Q_{p^s})$.  We say that $[b]\in B(G)$ is {\it decent} if it contains an $s$-decent element for some positive integer $s$.
                  Since $G$ is connected and the residue field of $\breve{\Q}_p$ is algebraically closed
         \cite[Lemma 9.1.33]{DOR} implies that any $[b]\in B(G)$ is decent (see also \cite{K1}).
         Moreover, by \cite[Lemma 9.6.19]{DOR}, if $G$ is quasi-split each $[b]\in B(G)$ contains an element
         $b$ which is $s$-decent for some $s$ and such that ${\nu}_b$ is defined over $\Q_{p^s}$.

                       \section{Period domains}   
    The period domains we are interested in classify weakly admissible filtrations on isocrystals. We briefly review here the relevant facts. A beautiful reference for everything that follows is \cite[Ch.\ I, IV, V, VI, VIII]{DOR}. 
  
    \subsection{Filtrations}We start by reviewing some basic facts concerning filtrations of Tannaka fiber functors, following \cite{DOR}. 
  
\subsubsection{Filtrations on vector spaces}

   If $K$ is a field we denote by ${\rm Vec}_K$ the category of finite dimensional $K$-vector spaces. 
   An $\mathbf{R}$-filtration $FV$ on $V\in {\rm Vec}_K$ is 
   a decreasing, exhaustive, separated filtration $(F^xV)_{x\in \mathbf{R}}$ by $K$-vector subspaces of $V$, such that $F^xV_K=\cap_{y<x} F^yV_K$ for all $x$. We denote by
   ${\rm gr}^x(FV)=F^xV/\sum_{y>x} F^yV$ and ${\rm gr}(FV)=\oplus_{x\in \mathbf{R}} {\rm gr}^x(FV)$ the associated 
   $\mathbf{R}$-graded vector space. 
      The {\it type of $FV$} is the nonincreasing sequence 
   $\nu(FV)=(x_1^{(n_1)},\dots,x_r^{(n_r)})\in \mathbf{R}^n$,
   where $x_1>\dots>x_r$ are the jumps of the filtration (i.e., 
   those $x$ for which ${\rm gr}^x(FV)\ne 0$) and 
      $x_i$ is repeated $n_i:=\dim {\rm gr}^{x_i}(FV)$ times. 
      We say that $FV$ is a {\it $\Q$-filtration} if 
  $\nu(FV)\in \Q^n$.

Let $K/k$ be a field extension and let ${\rm Fil}_k^K$ be the 
    category of pairs $(V,FV_K)$, where $V\in {\rm Vec}_k$ and $FV_K$ is an $\mathbf{R}$-filtration on the $K$-vector space $V_K:=V\otimes_k K$. This is a quasi-abelian category endowed with the degree, rank and slope  functions, defined, for $V:=(V,FV_K)$, by 
 $$ \deg(V)=\sum_{x} x\dim_K {\rm gr}^x(V_K),\, {\rm rk}(V)=\dim_k V, \,
  \mu(V)=\frac{\deg(V)}{{\rm rk}(V)}, \,V\ne 0.$$
 
  One has a good Harder-Narasimhan formalism with respect to this slope function. In particular, there is a notion of semistability for objects of 
  ${\rm Fil}_k^K$ and the tensor product of semistable objects is semistable when $K/k$ is separable, thanks to
  the theorems of Faltings \cite{Fal} and Totaro \cite{Tot} (this fails when $K/k$ is not separable). Moreover, one can characterize the semistable objects in terms of the inner product (this will be a recurrent theme in the sequel) as follows. Recall that, if 
      $FV, F'V$ are $\mathbf{R}$-filtrations on $V\in {\rm Vec}_K$, their {\em inner product} is defined by
      $$\langle FV, F'V\rangle:=\sum_{x,y\in \mathbf{R}} xy \dim_K {\rm gr}_{FV}^x({\rm gr}_{F'V}^y),$$
         where ${\rm gr}^x_{FV}({\rm gr}^y_{F'V})$ is the $x$th graded piece of the 
      $K$-vector space ${\rm gr}^y(F'V)$ endowed with the filtration induced by $FV$. The semistability criterion is then:

   \begin{proposition}[{\cite[Cor.\ 1.2.6]{DOR}}] \label{dublin1}
 The pair $(V, FV_K)$ is semistable if and only if  $$\langle FV_K,F^{\prime}V_K\rangle\leq \mu(V,FV_K)\deg(V,F^{\prime}V_K)$$  for all $\Z$-filtrations $F^{\prime}V_K$ of $V_K$. Moreover, it suffices to check this inequality when 
  $\deg(V, F^{\prime}V_K)=0$. 
  
  \end{proposition}

      Let $\Q-{\rm Fil}_k^K$
  be the full subcategory of those $(V,FV_K)$ for which $FV_K$ is a $\Q$-filtration on $V_K$ and let 
  $\Q-{\rm Grad}_K$ be the category of $\Q$-graded finite dimensional $K$-vector spaces.
  There are natural functors, the first being $(V,FV)\to {\rm gr}(FV)$ and the second being the forgetful functor
  $$\Q-{\rm Fil}_k^K\to \Q-{\rm Grad}_K\to {\rm Vec}_K,$$
  as well as a functor $\Q-{\rm Grad}_K\to \Q-{\rm Fil}_k^K$ sending $V=\oplus_{a\in \Q} V_a$ to the filtration $FV_K$ such that $F^xV_K=\sum_{a\geq x} V_a$.

 \subsubsection{Filtrations on ${\rm Rep}_k(G)$} 
 
 Let $K/k$ be a field extension and let $G$ be a linear algebraic group over $k$. Let 
   $\omega^G: {\rm Rep}_k(G)\to {\rm Vec}_k$ be the standard fiber functor, sending $(V,\rho)$ to $V$. 
  Recall that 
  $\mathbb{D}_K$ is the pro-torus over $K$ with character group $\Q$. The functor 
  sending $(V,\rho)\in {\rm Rep}_K(\mathbb{D}_K)$ to $V=\oplus_{a\in \Q} V_a$, where 
   $V_a$ is the weight space of $V$ corresponding to $a\in \Q=X^*(\mathbb{D}_K)$, induces an equivalence of neutral Tannakian categories over $K$ 
    $${\rm Rep}_K (\mathbb{D}_K)\simeq \Q-{\rm Grad}_K.$$
    Any morphism of $K$-group schemes $\mathbb{D}_K\to G_K$ induces therefore a 
  {\it $\Q$-grading of $\omega^G$ over $K$}, i.e., a tensor functor 
      $F: {\rm Rep}_k(G)\to \Q{\rm-Grad}_K$ whose composition with the forgetful  functor 
      $\Q-{\rm Grad}_K\to {\rm Vec}_K$ is $(V,\rho)\to V\otimes_k K$. All 
      $\Q$-gradings of $\omega^G$ over $K$ are thus obtained, thanks to the Tannakian formalism. If
  $G=\mathbb{GL}(V)$ with $V\in {\rm Vec}_k$, giving $F$ comes down to giving a $\Q$-grading of $V\otimes_k K$.   
   
    A {\it $\Q$-filtration of $\omega^G$ over $K$} is a tensor functor 
  $F: {\rm Rep}_k(G)\to \Q-{\rm Fil}_k^K$ whose composition with $\Q-{\rm Fil}_k^K\to \Q-{\rm Grad}_K$ is exact\footnote{The exactness condition is imposed so that filtrations can be described using gradings.} and whose composition with the forgetful functor
  $\Q-{\rm Fil}_k^K\to {\rm Vec}_k$ is $\omega^G$. When 
  $G=\mathbb{GL}(V)$ with $V\in {\rm Vec}_k$,  giving $F$ comes down to giving a $\Q$-filtration of $V\otimes_k K$ \cite[Rem.\ 4.2.11]{DOR}.  
  
   Let ${\rm Grad}_K(\omega^G)$ (resp.\ ${\rm Fil}_K(\omega^G)$) be the set of $\Q$-gradings (reps. $\Q$-filtrations) of $\omega^G$ over $K$. 
  Composition with the natural functor $\Q-{\rm Grad}_K\to \Q-{\rm Fil}_k^K$ yields a natural map
  $${\rm Hom}_K(\mathbb{D}_K, G_K)\simeq {\rm Grad}_K(\omega^G)\to 
  {\rm Fil}_K(\omega^G),$$
 which is surjective when $G$ is reductive or $k$ has characteristic $0$ (see \cite[Th.\ 4.2.13]{DOR}). 
   
    {\it We assume in the sequel that $G$ is reductive}. Two morphisms 
     $\lambda,\lambda':\mathbb{D}_K\to G_K$ are 
 in the same fiber of the map ${\rm Hom}_K(\mathbb{D}_K, G_K)\to 
  {\rm Fil}_K(\omega^G)$ (in which case we say that they are {\it par-equivalent}) if and only if 
 $\lambda'={\rm Int}(g)\lambda$ for some 
 $g\in P(\lambda)(K)$, where 
 $P(\lambda)$ is the parabolic subgroup of $G$ defined over $K$ consisting of those $g\in G$ for which $\lim_{t\to 0} {\rm Int}(\lambda(t)) g$ exists. If $U(\lambda)$ is the unipotent radical of $P(\lambda)$,   
 $\lambda$ and $\lambda'$ are par-equivalent if and only if there is a unique 
    $g\in U(\lambda)(K)$ such that $\lambda'={\rm Int}(g)\lambda$, and then $P(\lambda')=P(\lambda)$
    (however the latter equality does not imply that $\lambda, \lambda'$ are 
           par-equivalent).
    To summarize, we obtain maps
    $$ {\rm Hom}_K(\mathbb{D}_K, G_K)/{\rm par-equivalence}\simeq 
    {\rm Fil}_K(\omega^G) \twoheadrightarrow K-{\rm Par}(G),$$
    where $K-{\rm Par}(G)$ is the set of parabolic subgroups of $G$ defined over $K$.
    
     \subsubsection{Filtrations on isocrystals}    \label{lyon-11} Consider the setup described in the introduction to Section \ref{BG}: $F$ is a finite extension of 
     $\Q_p$, etc.    
     Let $K$ be an extension of $\breve{F}$. An object of the category of 
     filtered isocrystals over $K$ (relative to $\breve{F}/F$) consists of an isocrystal 
     $(V,\varphi)$ over $\breve{F}$ together with a $\Q$-filtration $FV_K$ on
     $V_K$. This category has a good Harder-Narasimhan formalism for the slope function 
     $$\mu(V, \varphi, FV_K)=\frac{1}{\dim V}(\sum_{i \in \mathbf{Z}} i \dim_K {\rm gr}^i_F(V_K)-v_{F}(\det \varphi)),$$
     and its semistable objects of slope $0$ are the weakly admissible filtered isocrystals introduced by Fontaine, and they are stable under tensor products by the theorem of Faltings and Totaro. 
   
     The above slope function can also be described as 
     $$\mu(V,\varphi,FV_K)=\mu(V,FV_K)+\mu(V,F_0V),
$$
where $F_0V$ is the slope filtration (defined by $F^{x}_0V=\bigoplus_{\alpha\leq -x}V_{\alpha}$) associated to the slope decomposition 
 $(V,\varphi)=\bigoplus_{\alpha\in\Q}V_{\alpha}$ of the isocrystal $(V,\varphi)$. 
     
           \subsection{Period domains} We will define now the period domains we will be working with. Let the notation be as at the beginning of Section \ref{BG}.

\subsubsection{Flag varieties and  period domains} 

   Consider  a triple $(G, b, \{\mu\})$, where $G$ is a connected reductive group over $F$, 
        $b\in G(\breve{F})$, and $\{\mu\}$ is a conjugacy class of cocharacters of $G$ over $\overline{F}$. Let $E=E(G,\{\mu\})\subset \overline{F}$ be the associated reflex field\footnote{Recall that $E$ is simply the field of definition of $\{\mu\}$, a finite extension of $F$.}. Let  
        $\mathcal{F}=\mathcal{F}(G, \{\mu\})$ be the associated flag variety, a smooth projective variety over $E$, homogeneous under $G_E$ and whose 
        $\overline{E}= \overline{F}$-points are the par-equivalence classes of elements in 
        $\{\mu\}$. If $G$ is quasi-split, which will be the case in our applications, a result of Kottwitz \cite[Lemma 1.1.3]{K2} shows that $\{\mu\}$ contains elements $\mu$ defined over $E$, and then $\mathcal{F}=G_E/P(\mu)$ (recall that $P(\mu)$ is the associated parabolic subgroup of $G_E$).
       
       Let $\mu'$ be an element of $\{\mu\}$ defined over an extension $K$ of 
       $E$. We say that the {\em pair $(b,\mu^{\prime})$ is weakly admissible}
       if the filtered isocrystal
     $(N_b(V), F_{\rho\mu'}(V_K))$ is weakly admissible for all $(V,\rho)\in {\rm Rep}_F(G)$. By the tensor product theorem it suffices to check this for a single faithful 
     representation $(V,\rho)$.
  
       Let $\breve{E}=E\breve{F}$ and let $\breve{\mathcal{F}}$ be the adic  space attached to 
        $\mathcal{F}\otimes_{E} \breve{E}$. By \cite[Prop.\ 1.36]{RZ} there is a partially proper open subset $\breve{\mathcal{F}}^{\rm wa}=\breve{\mathcal{F}}(G, b, \{\mu\})^{\rm wa}$ of $\breve{\mathcal{F}}$ 
        such that $\breve{\mathcal{F}}^{\rm wa}(K)$  is the set of weakly admissible points in $x\in \breve{\mathcal{F}}(K)$, i.e., points with associated cocharacter $\mu_x\in \{\mu\}$ defined over $K$ and for which the pair $(b, \mu_x)$ is weakly admissible.
  This is the 
        {\it period domain attached to} $(G, b, \{\mu\})$. We will see in the next section, and this will be crucial for the computation of its \'etale cohomology, that the period domain  is of the form 
        $$
       \breve{\sff}^{\rm wa}=\breve{\sff}\setminus \bigcup_{i\in I}J_b(F)Z_i,
        $$
        where $\{Z_i\}_{i\in I}$ is an {\em explicit} finite set of Schubert varieties.

        Up to isomorphism, $\breve{\mathcal{F}}(G, b, \{\mu\})^{\rm wa}$ depends only on 
        $[b]\in B(G)$ (sending $\mu$ to ${\rm Int}(g)(\mu)$ yields an isomorphism 
  $\breve{\mathcal{F}}_b^{\rm wa}\simeq \breve{\mathcal{F}}_{gb\sigma(g)^{-1}}^{\rm wa}$). The group $J_b(F)$ acts naturally on the flag variety $\breve{\mathcal{F}}$ over $\breve{E}$ (via $J_b(F)\subset G(\breve{F})$) 
      and this action restricts to an action on 
        $\breve{\mathcal{F}}^{\rm wa}$. Moreover, if $F=\Q_p$ and $b$ is $s$-decent, the period domain has a canonical model $\mathcal{F}^{\rm wa}\subset \mathcal{F}\otimes_E E_s$ over 
        $E_s:=E\Q_{p^s}$.

         \begin{example} \label{Drinfeld1}
    Take $G=\mathbb{GL}_n$, $[b]=[1]$, and 
  $\{\mu\}=(n-1, -1,-1,\dots,-1)$. The corresponding period domain is 
  $\mathbb{P}^{n-1}_{F}\setminus \cup_{H\in \mathcal{H}} H$, where $\mathcal{H}$ is the set of $F$-rational hyperplanes, i.e., the Drinfeld symmetric space of dimension 
  $n-1$.        
     \end{example}
\subsubsection{Existence of weakly admissible filtrations} 
       Fontaine and Rapoport in \cite[Th.\ 3]{FR} found a 
        simple criterion for the existence of weakly admissible filtrations on isocrystals; this result was extended in \cite[Th.\ 9.5.10]{DOR} (see also \cite[Prop.\ 3.1]{RV}).   
          To present it, we assume, for simplicity, that $G$ is quasi-split (see \cite[Sec.\ 3]{RV} for the general case) and we recall that the set of acceptable elements for $\{\mu\}$ \cite[2.2]{RV} is 
        $$A(G, \{\mu\}):=\{[b]\in B(G)| \nu_b\leq \mu^{\diamond}\},$$
    a finite nonempty set, intersecting nontrivially with $B(G)_{\rm basic}$ (see
 \cite[Lemma 2.5]{RV}). Here $\mu^{\diamond}$ (denoted $\bar{\mu}$ in \cite{RV}) is the average of 
 the cocharacters in the $\mathcal{\rm Gal}(\bar{F}/F)$-orbit of $\mu$.

        The following result is a generalization (and a group-theoretical reformulation)  of Mazur's  ``the Hodge polygon lies above the Newton polygon'' property of isocrystals.
       \begin{theorem}[Fontaine-Rapoport, {\cite[Th.\ 9.5.10]{DOR}}]
   The space $\breve{\mathcal{F}}(G,b,\{\mu\})^{\rm wa}$ is nonempty if and only if $[b]\in A(G, \{\mu\})$.
     \end{theorem}
      \subsubsection{Local Shtuka datum}    
  The following notion will be useful in the rest of the  paper. It is a modification of the notion of a local Shimura datum of Rapoport-Viehmann \cite[Def.\ 5.1]{RV}, where we drop the assumption that $\mu$ is minuscule and allow $[b]$ to be just acceptable (instead of neutral acceptable). 
 \begin{definition}
A {\em local Shtuka datum} over $F$ is a triple $(G,[b],\{\mu\})$ consisting of a connected reductive algebraic group $G$ over $F$, a $\sigma$-conjugacy class  $[b]\in B(G)$, and   a geometric conjugacy class $\{\mu\}$ of cocharacters of $G$ defined over $\overline{F}$. We assume that
  $[b]\in A(G,\{\mu\})$. 
 \end{definition} 
 As discussed above, associated to a local Shtuka datum are the following data:
 \begin{enumerate}
 \item the algebraic group $J=J_b$ over $F$, for $b\in [b]$,
 \item the reflex field $E=E(G,\{\mu\})$,
 \item the flag variety $\sff=\sff(G,\{\mu\})$ over $E$,
 \item the period domain $\breve{\sff}^{\rm wa}=\breve{\sff}(G,b,\{\mu\})^{\rm wa}$ over $\breve{E}$.
\end{enumerate}
 
    \begin{remark}
               As explained in \cite[Sec.\ IX.8]{DOR}, assuming $F=\Q_p$ is not really a restriction since Weil descent allows one to pass from a general $F$ to $\Q_p$ (this is similar to the situation for Shimura varieties, where restriction of scalars allows one to reduce the study to groups defined over $\Q$). In particular, we can get the Drinfeld symmetric space over general $F$ as a period domain in this context as well (see \cite[Example 9.8.9]{DOR}). {\em For simplicity, we will work from now on with $F=\Q_p$}.
               \end{remark}

   \section{The geometry of  complements of period domains} 
   Via the embeddings of  period domains into flag varieties, we will reduce the computation of the cohomology of a period domain to that of its complement in the flag variety. The period domain is a locus of semistability and its complement can be stratified by Schubert varieties given by the  degree to which this semistability fails\footnote{This stratification shares many properties with the Harder-Narasimhan stratification of the space of vector bundles over a Riemann surface.}. This section describes this stratification.
  \subsection{The Hilbert-Mumford criterion}\label{Hil-Mum}
 
    Suppose that a connected reductive group $G$ over a field $k$ acts on a proper
    algebraic variety $X$ over $k$. The Hilbert-Mumford  criterion \cite{GIT} describes the semistable points of this action via eigenvalues of $1$-parameter subgroups (1-{\rm PS}s, for short). We will review it briefly. 
    
    Let $\sll\in {\rm Pic}^G(X)$ be a $G$-equivariant  line bundle on $X$. 
   If $\lambda\in X_*(G)$ is defined over $k$, Mumford defined the {\em slope} of  $x\in X(k)$ with respect to $\lambda$  and $\sll$, denoted $\mu^{\sll}(x,\lambda)\in \Z$. If
   $x_0=\lim_{t\to 0} \lambda(t)x$ (this limit exists since, by properness of 
   $X$, the map $\mathbb{G}_m\to X, t\to \lambda(t)x$ extends to $\mathbb{A}^1$), then 
   $\mu:=\mu^{\sll}(x,\lambda)$ is characterized by the fact that $\lambda$ acts by the character 
   $t\to t^{-\mu}$ on the fiber $\sll_{x_0}$ (note that $x_0$ is fixed under the action of 
   $\mathbb{G}_m$, so this makes sense). 
   
      Since $$\mu^{\sll_1\otimes \sll_2}(x, \lambda^r)=r(\mu^{\sll_1}(x, \lambda)+\mu^{\sll_2}(x, \lambda))$$
      for $r\in \Z_{>0}$, 
 the previous definition extends 
   to $\sll\in {\rm Pic}^G(X)_{\Q}$ and $\lambda\in X_*(G)_{\Q}$. Moreover, we have 
      $\mu^{\sll}(gx, \lambda)=\mu^{\sll}(x, g^{-1}\lambda g)$ for $g\in G(k)$ and, most importantly  \cite[Prop.\ 2.7]{GIT}
$$\mu^{\sll} (\alpha x, \lambda)=\mu^{\sll}(x,\lambda),\quad \alpha\in P(\lambda)(k).$$
The construction has good functoriality properties: we have $
\mu^{f^*(\sm)} ( x, \lambda)=\mu^{\sm} (f(x), \lambda)
$ for a $G$-morphism $f:X\to Y$ of $G$-varieties and $\sm \in {\rm Pic}^{G}(Y)_{\Q}$. Moreover, 
 if $X=X_1\times\cdots\times X_m$ is a product of $G$-varieties and $\sll _i\in {\rm Pic}^G(X_i)_{\Q}$ then $$
\mu^{\sll_1\boxtimes\cdots\boxtimes\sll_m} ((x_1,\dots, x_m), \lambda)=\sum_{i=1}^m\mu^{\sll_i} (x_i, \lambda).
$$

      Let now $k$ be algebraically closed. Recall that the
      (open) semistable locus $X^{\rm ss}(\sll)$ in $X$ can be defined as the set of points $x\in X(k)$
   having an affine open neighborhood of the form $X_f=\{f\ne 0\}$ with $f\in \Gamma(X, \sll^{\otimes n})^G$ for some $n$.
   We have the  following Hilbert-Mumford  numerical criterion: 
   \begin{theorem}[Hilbert-Mumford]
   If $k$ is algebraically closed and $\sll$ is ample then, for $x\in X(k)$, 
   $$
   x\in X^{\rm ss}(\sll) \Longleftrightarrow \mu^{\sll}(x,\lambda)\geq 0 \text{ for all }
   \lambda\in X_*(G).
$$
   \end{theorem}
  \subsection{A Hilbert-Mumford criterion for weak admissibility }   Since weak admissibility was defined as a semistability condition in Section \ref{lyon-11} we can test it using a Hilbert-Mumford criterion once the right linearization of the group action is defined. We will now describe it.
        For the rest of this section  
            $G$ will be a reductive group over a perfect field $k$. 
           
           \subsubsection{Invariant inner products}  An {\it invariant inner product on $G$} (IIP for short) 
            is the data of a positive-definite bilinear form
           on $X_*(T)_{\Q}$ for each maximal torus $T$ of $G$ (defined over $\bar{k}$), compatible\footnote{ That is,  such that the maps 
           ${\rm Int}(g): X_*(T)_{\Q}\to X_*(gTg^{-1})_{\Q}$ and 
           $\tau: X_*(T)_{\Q}\to X_*(^{\tau}T)_{\Q}$, $^{\tau}T=\tau T\tau^{-1}$,  induced by any $g\in G(\bar{k})$ and
           $\tau\in \sg_k$ are isometries.} with the action of 
           $G(\bar{k})$ and $\sg_k$. These objects are standard in GIT and they seem to go back at least to Kempf's celebrated paper 
           \cite{Kempf}. 
                      
             One can describe IIP's on $G$ in terms of a fixed 
            maximal torus $T_0$ of 
           $G$ (defined over $\bar{k}$) as follows. Let $W=W(G,T_0)$ be the 
            associated Weyl group and consider the {\it L-action}\footnote{Explicitly, pick a Borel subgroup
           $B_0$ of $G$ defined over $\bar{k}$ and containing $T_0$; if $\tau\in \sg_k$, one can find 
           $g\in G(\bar{k})$, unique up to left translation by $T_0(\bar{k})$, such that $g^{\tau}T_0g^{-1}=T_0$ and 
           $g^{\tau}B_0g^{-1}=B_0$, and then ${\rm Int}(g) \sigma$ is an automorphism of $X_*(T_0)$ independent of the choice of $g$ and $B_0$ and this defines the $L$-action.} of 
           $\sg_k$ on $X_*(T_0)$ (if $G$ is quasi-split over $k$ and $T_0$ is defined over $k$, this is the usual action of $\sg_k$ on $X_*(T_0)$). 
           The conjugacy of maximal tori in $G$ implies that 
            IIP's on 
           $G$ correspond to $\sg_k\ltimes W$-invariant inner products on $X_*(T_0)_{\Q}$, 
           hence any $G$ has an IIP (the action of $\sg_k$ factors through a finite quotient).
           
            When $G$ is semisimple there is  a natural choice of an IIP on $G$  corresponding to the inner product on $X_*(T_0)_{\Q}$    given by the Killing form:      
           $$\spp(\lambda, \lambda')=\sum_{\chi} \langle \lambda, \chi\rangle \langle \lambda', \chi\rangle,$$
           the sum being taken over the roots of $T_0$. In general, \cite[6.2.4]{DOR} shows that any 
           IIP on $G$ is the orthogonal direct sum of IIP's on the torus $Z(G)^0$ and on $G_{\rm der}$. Moreover, on 
           a $k$-simple semisimple group any IIP is a positive multiple of the Killing form; on the other hand, any 
           $\sg_k$-invariant inner product on 
           $X_*(T)_{\Q}$ is an IIP on the torus $T$, so there is no canonical choice in this case.  
           
           \subsubsection{A $\Q$-linearization of the $G$-action: the Hodge filtration}
           
            We will explain now how to construct the ample equivariant line bundle needed to apply the Hilbert-Mumford criterion. We start with a simple but crucial case. 
           
           \begin{example}\label{general}Let $G=\mathbb{GL}_n$, with diagonal torus $T$ and the standard IIP $\spp$ on $X_*(T)_{\Q}=\Q^n$.
           Let 
           $\mu(t)=(t^{a_1},\dots,t^{a_1},\dots,t^{a_r},\dots,t^{a_r})$, $t^{a_i}$ appearing $n_i$ times and 
             $a_1>\cdots >a_r$ being integers. Then $\mathcal{F}(G, \{\mu\})=G/P$, $P$ being the upper triangular parabolic with Levi $M=\prod_{i=1}^r \mathbb{GL}_{n_i}$. 
             We define a $G_{\overline{k}}$-equivariant line bundle on $\sff_{\overline{k}}$ by
             $$\mathcal{L}_{G,  \{\mu\},\spp}=G\times^{P} {\mathbb{G}_{a, \lambda}},$$
             where
             $\lambda\in X^*(M)=X^*(P)$ is defined by $\lambda(g_1,\dots,g_r)=\prod_{i=1}^r (\det g_i)^{-a_i}$ for $(g_1,\dots,g_r)\in M$.
             The minus sign appears here to get an ample line bundle. We can interpret this geometrically as follows. Let 
             $V=k^n$. Then $\mathcal{F}:=\mathcal{F}(G, \{\mu\})$ is the variety (over $k$) of partial flags of $V$, of type $(n_1, n_1+n_2,\dots, n_1+\cdots+n_r=n)$ and so $\mathcal{F}$ is naturally a closed subvariety 
             $\sff\hookrightarrow \prod_{i=1}^r X_i$ 
             of a product of Grassmanians
             $X_i:= \mathbb{G}\mathbbm{r}_{n_1+\cdots+n_i}(V).
       $ Each $X_i$ has a natural very ample $\mathbb{GL}(V)$-equivariant line bundle $\mathcal{L}_i$ 
        obtained from $\so(1)$ via the Pl\"ucker embedding $X_i\hookrightarrow \mathbb{P}(\wedge^{n_1+\cdots+n_i}(V))$. 
        Then $\mathcal{L}_{G,  \{\mu\},\spp}$ is isomorphic (as equivariant line bundle) to the restriction of $\mathcal{L}_1^{\otimes (a_1-a_2)}\boxtimes \mathcal{L}_2^{\otimes (a_2-a_3)}\boxtimes\cdots\boxtimes \mathcal{L}_r^{\otimes a_r}$. 
              If 
              $x\in \mathcal{F}(K)$ is a point corresponding to a filtration $F_x$ of 
              $V_K$, then the fiber of $\mathcal{L}=\mathcal{L}_{G,  \{\mu\},\spp}$ at $x$ is 
              $\otimes_{i} \det({\rm gr}^{a_i}(F_x))^{-\otimes a_i}$. We deduce immediately from this 
           (see \cite[Lemma 2.2.1, Lemma 2.2.2]{DOR}) that for any point $x\in \breve{\sff}(K)$ ($K$ being an extension of $\breve{E}$), and any 
                           $\lambda\in X_*(G)^{\sg_F}$ with associated filtration $F_{\lambda}$, we have  
              $$\mu^{\mathcal{L}}(x,\lambda)=-\langle F_x, F_{\lambda}\rangle.$$
                In particular the Hilbert-Mumford criterion is nothing but Corollary \ref{dublin1} in this special case.

\end{example}
      
            \begin{proposition} \label{kolo2} Given a connected reductive $k$-group $G$, an element $\{\mu\}\in X_*(G)_{\Q}/G$ and 
        an IIP $\spp$ on $G$, one can naturally construct a $\Q$-line bundle
            $$\mathcal{L}_{G,  \{\mu\},\spp}\in {\rm Pic}^{G, \mathrm{ample}} (\mathcal{F}(G,\{\mu\}))_{\Q}$$
                            defined over the reflex field of $(G,\{\mu\})$ and such that: 
             \begin{enumerate}
                       \item   If $\iota: G\to G'$ is a closed immersion defined over $k$ and 
             $\spp$ is induced by an IIP $\spp'$ on $G'$, then $\mathcal{L}_{G,  \{\mu\},\spp}$ is the restriction of $\mathcal{L}_{G', \iota\{\mu\},\spp^{\prime}}$ via the closed immersion 
             $\sff(G,\{\mu\})\hookrightarrow \sff(G', \iota \{\mu\})\otimes_{E'}E$, where $E,E'$ are the corresponding reflex fields.
             \item If $G=\mathbb{GL}_n$ and $\spp$ is the standard invariant inner product, $\mathcal{L}_{G,  \{\mu\},\spp}$ is the one in Example \ref{general}.
                              \end{enumerate}      
                         \end{proposition}
                         
                         \begin{proof}
                          This follows from the discussion preceding 
                  \cite[Th.\ 6.2.8]{DOR} (for example,  \cite[Lemma 6.2.5]{DOR}). The key point is that an IIP $\spp$ on $G$ induces one on 
                           the Levi quotient $M$ of any parabolic subgroup $P$ of $G$ in a natural way (since any maximal torus of $M$ is an isomorphic image of a maximal torus of $G$ contained in $P$), which allows one to associate to 
                            any $\lambda\in X_*(G)$ an element $\lambda^*\in X^*(P(\lambda))_{\Q}$. The recipe is then induced by sending $\lambda$ to 
            $\mathcal{L}_{\lambda}=G\times^{P}{\mathbb{G}_{a, -\lambda^{*}}}$, the $G$-equivariant ${\Q}$-line bundle on 
            $G/P(\lambda)$ corresponding to $-\lambda^{*}$, the minus sign being chosen to ensure that 
            $\mathcal{L}_{\lambda}$ is ample.                     
                         \end{proof}
                         
                 \subsubsection{A $\Q$-linearization of the $J$-action: the Hodge and the slope filtrations}
               Let now $k=\Q_p$ and 
                consider a local shtuka datum $(G, [b], \{\mu\})$ over $F=\Q_p$.
            Fix an IIP $\spp$ on 
            $G$ and a {\em decent} $b\in G(\breve{\Q}_p)$, with associated automorphism group $J=J_b$ and slope morphism
            $\nu_b$.
            The element $\lambda_b:=-\nu_{b}\in {\rm Hom}_{\breve{\Q}_p}(\mathbb{D}_{\breve{\Q}_p}, G_{\breve{\Q}_p})$ gives rise to 
   an ample $G_{\breve{\Q}_p}$-equivariant $\Q$-line bundle $\sll_b:=\sll_{G,\{\lambda_b\},\spp}$ on $$\sff^b:=\sff(G_{\breve{\Q}_p},\{\lambda_b\})=G_{\breve{\Q}_p}/P(\lambda_b),$$
   which will be considered as a $J_{\breve{\Q}_p}$-equivariant line bundle via the natural map $J_{\breve{\Q}_p}\to G_{\breve{\Q}_p}$. By the same token we consider $\mathcal{L}_{G, \{\mu\}, \spp}$ as a $J_{\breve{E}}$-equivariant ${\Q}$-line bundle on 
       $\mathcal{F}_{\breve{E}}$ and define
       $$\mathcal{L}_{G, [b], \{\mu\}, \spp}:=i^*(\mathcal{L}_{G, \{\mu\}, \spp}\times\sll_b) \in {\rm Pic}^{J_{\breve{E}}, \mathrm{ample}}(\mathcal{F}_{\breve{E}})_{\Q},$$
      the closed embedding 
      $
      i:\sff_{\breve{E}}\hookrightarrow \sff_{\breve{E}}\times\sff^b_{\breve{E}}
      $ being 
      given by the identity on the first factor and by the $\breve{E}$-rational point $\lambda_b$ of $\sff^b_{\breve{E}}$ on the second factor. 
      The construction enjoys similar properties to the one from Proposition \ref{kolo2}, concerning functoriality with respect to closed embeddings $\iota: G\to G'$ (by taking $b'=\iota(b)$, of course).

  \subsubsection{A Hilbert-Mumford criterion for weak admissibility}  We keep the notations introduced in the previous paragraph and set         
            $\mathcal{L}=\mathcal{L}_{G,b, \{\mu\}, \spp}\in {\rm Pic}^{J_{\breve{E}}, \mathrm{ample}}(\breve{\mathcal{F}})_{\Q}$, where $\breve{\mathcal{F}}:=\mathcal{F}(G, \{\mu\})_{\breve{E}}$. 
            
                         \begin{theorem}[Totaro, {\cite[Th.\ 3]{Tot}}, {\cite[Th.\ 9.7.3]{DOR}}] \label{totaro1} Let $K/\breve{E}$ be a field extension and let 
                         $x\in \breve{\mathcal{F}}(K)$. Then $x\in \breve{\mathcal{F}}(G,b,\{\mu\})^{\rm wa}(K)$ 
           if and only if $\mu^{\mathcal{L}}(x,\lambda)\geq 0$ for all 
             $\lambda\in X_*(J)^{\sg_F}$. 
                          
             \end{theorem}
          
            \begin{example} Suppose that $G=T$ is a torus over $\Q_p$. 
            We then have $J=T$ and $$\mu^{\mathcal{L}}(x,\lambda)=-(\spp(\lambda, \lambda_x)+\spp(\lambda, \nu_b)),
              $$
              for any IIP $\spp$ on $T$, any $\lambda\in X_*(T)$, and $\lambda_x\in X_*(T)$ attached to 
              $x$. Hence the Hilbert-Mumford inequality takes the form:  $\spp(\lambda, \lambda_x)+\spp(\lambda, \nu_b)\leq 0$ for all $\Q_p$-rational $1$-{\rm PS} $\lambda$ of $T$. Since we may replace $\lambda$ with $-\lambda$ this condition is satisfied if and only if $\lambda_x+\nu_b$ is orthogonal to all $\Q_p$-rational $1$-{\rm PS} of $T$. Since the inner product is $\sg_{\Q_p}$-invariant, 
the $\Q_p$-rational characters of $T$ correspond to the $\Q_p$-rational $1$-{\rm PS} and the Hilbert-Mumford criterion becomes the familiar criterion:  $(\mu,b)$ is weakly admissible if and only if 
   $\mu+\nu_b$ is orthogonal to all $\Q_p$-rational characters of $T$.  

   \end{example}

      \begin{example} \label{kolo15}
      Consider the set-up from Example \ref{general}. The data of $b$ is equivalent to that of an isocrystal $N=(V, \varphi)$ of dimension $n$ over $\breve{\mathbf{Q}}_p$, with slope decomposition 
            $V=\oplus_{\alpha\in \Q} V_{\alpha}$. Let $\alpha_1>\cdots>\alpha_t$ be the different slopes and let $m_i=\dim V_{\alpha_i}$, so that the Newton vector of 
            $N$ is 
                        $\nu(N)=(\alpha_1^{(m_1)},\dots, \alpha_t^{(m_t)})\in (\Q^n)_+$. 
                        Consider the filtration $F_0$ defined by $F_0^{\beta}=\oplus_{\alpha\leq -\beta} V_{\alpha}$ (its type is thus
                  $\nu_0=(-\alpha_t^{(m_t)},\dots, -\alpha_1^{(m_1)})\in (\Q^n)_+$). An immediate computation  \cite[Lemma 8.4.2]{DOR}
 shows that if 
                  $K/\breve{\mathbf{Q}}_p$ is an extension and $x\in \breve{\mathcal{F}}(K)$ with corresponding filtration $F_x$ of 
              $V_K$, then for all $\lambda\in X_*(J)^{\sg_F}$
                            $$\mu^{\mathcal{L}}(x,\lambda)=-(\langle F_x, F_{\lambda}\rangle+
              \langle F_0, F_{\lambda}\rangle).$$
                The Hilbert-Mumford inequality $\mu^{\mathcal{L}}(x,\lambda)\geq 0$ is thus equivalent to $
                \langle F_x, F_{\lambda}\rangle+
              \langle F_0, F_{\lambda}\rangle\leq 0
                $.                 \end{example}

                   The computations in Example \ref{kolo15} and the basic properties of the slope function recalled in Section \ref{Hil-Mum} yield the first part of the following lemma. The second part is an immediate calculation. 
            
                             \begin{lemma}[Orlik, {\cite[Lemma 2.2]{Ol}}, {\cite[Lemma 2.2]{Of}}]
                  \label{explicit}Let $V$ be a faithful $F$-rational representation of $G$ (defined over $F$) and consider the IIP $\spp$ on 
                  $G$ induced by the standard IIP on $\mathbb{GL}(V)$. 
                 \begin{enumerate}
                 \item Let $K/\breve{E}$ be an extension, $x\in\breve{\sff}(K)$ and $\lambda\in X_*(J)^{\sg_F}$. Let  $F_x$,  $F_{\lambda}$, and $F_0$ denote the filtrations on $V_{\breve{E}}$ induced by $x$, $\lambda$ and $\lambda_b=-\nu_b$, respectively. Then 
                 $$
                 \mu^{\sll}(x,\lambda)=-(\langle F_x, F_{\lambda}\rangle+\langle F_0, F_{\lambda}\rangle).
                 $$
                 \item If $T\subset G$ is a maximal torus and $\lambda, \lambda^{\prime}\in X_*(T)_{\Q}$, then 
                 $
              \langle F_{\lambda}, F_{\lambda^{\prime}}\rangle  = \spp(\lambda,\lambda^{\prime}).
                 $
                 \end{enumerate}
                 \end{lemma}
\subsection{A simplified Hilbert-Mumford criterion} We will show here that in the Hilbert-Mumford criterion from Theorem \ref{totaro1} it suffices to test the $1$-{PS} associated to the relative simple roots and their conjugates. This will follow from the fact that the slope function, a priori just convex on every chamber of the spherical building, is, in fact, affine. 

 \subsubsection{The criterion} \label{notation1}          {\it We will assume from now on that $(G, [b], \{\mu\})$ is a local Shtuka datum over $F=\Q_p$ with $G$ quasi-split over $\Q_p$ and $b\in [b]$ 
 basic and $s$-decent.} Let $J=J_b$, $\nu=\nu_b$, let $E=E(G,\{\mu\})$ be the reflex field and let 
 $E_s=E\Q_{p^s}$. The associated period domain has a canonical model over $E_s$,invariant under the natural action of
   $J(\qp)$ on $\mathcal{F}$.

             We fix an invariant inner product $\spp$ on $G$ and we note that it gives rise to an invariant inner product on $J$. 
Fix a maximal $\qp$-split torus $S\subset J_{\rm der}$, of $\Q_p$-rank $d$, and a minimal parabolic subgroup $P_0$ of $J$ defined over $\Q_p$.   Let $\Delta=\{\alpha_1,\dots,\alpha_d\}\subset X^*(S)$ be the corresponding set of relative simple roots. Let $\omega_{\alpha_1},\dots,\omega_{\alpha_d}\in X_*(S)_{\Q}$ be the associated dual basis for the natural pairing between $X^*(S)_{\Q}$ and 
   $X_*(S)_{\Q}$. Fix a maximal torus $T$ of $G$ containing $S$ and such that $\mu, \nu\in {\rm Hom}_{\breve{\mathbf{Q}}_p} (\mathbb{D}, T)\simeq X_*(T)_{\Q}$.

\begin{proposition}[Orlik, {\cite[Cor.\ 2.4]{Of}}] \label{dublin4}
Let $x\in\breve{\sff}(K)$, for a field extension $K$ of $\breve{E}$. Then $x$ is not weakly admissible if and only if there exists an element $g\in J(\Q_p)$ and a simple root $\alpha\in\Delta$ such that $\mu^{\sll}(x,{\rm Int}(g)\omega_{\alpha}) <0$. 
\end{proposition}

 \subsubsection{The combinatorial and spherical buildings} \label{buildings1}
 
    If $G$ is a connected reductive group over a field $k$, one can associate to 
    $G$ two buildings, as follows. 
    
     The {\em combinatorial building} $\Delta(G)$ of $G$.  An abstract simplicial complex with 
    $G(k)$-action whose simplices are the proper $k$-parabolic subgroups of $G$ ordered by the opposite of inclusion (thus vertices are the  proper maximal $k$-parabolics and $(P_0,\dots,P_d)$ is a simplex if and only if $P_0\cap\dots\cap P_d$ is a parabolic subgroup).  If $n$ denotes the $k$-rank of the derived group of $G$ then $\Delta(G)$ has the homotopy type of a bouquet of $(n-1)$-spheres. 
         
     The {\em spherical building} $\sbb(G)$ of $G$. Unlike $\Delta(G)$ it takes into account the center of $G$.  If $m$ is the $k$-rank of $G$ the set $\sbb(G)$ is, in a $G(k)$-equivariant way, the $(m-n)$-fold suspension of $\Delta(G)$. Thus suitably topologized, it has the homotopy type of a bouquet of $(m-1)$-spheres. The building $\sbb(G)$ is  functorial for injective maps of reductive groups $f: G\to H$ over $k$: we have a natural embedding of topological spaces $\sbb(f): \sbb(G)\to \sbb(H)$.

     If $G=S$ is a split 
    $k$-torus, then $\sbb(S)$ is simply the sphere corresponding to half-lines in the vector space 
    $X_*(S)_{k,\mathbf{R}}$. In general, $\sbb(G)$ is obtained by gluing the different spheres $\sbb(S)$ (over all 
    maximal $k$-split tori $S\subset G$), where we identify $b$ and ${\rm Int}(g)b$ for 
    $b\in \sbb(S)$ and $g\in P(b)(k)$. Here $P(b)$ is a $k$-parabolic of $G$ naturally attached to 
    $b$ in a way compatible with the definition of $P(\lambda)$, for $\lambda\in X_*(S)_{\Q}$.
    There is a natural action of $G(k)$ on 
    $\sbb(G)$ and a natural map $b\to P(b)$ from $\sbb(G)$ to the set of $k$-parabolic subgroups of $G$ (we note  $P(b)(k)$ is the stabilizer of $b$ in $G(k)$). Moreover, if $S$ is a maximal 
    $k$-split torus of $G$, the map $\sbb(S)\to \sbb(G)$ is injective and consists precisely of 
    points $b$ of $\sbb(G)$ for which $S\subset P(b)$. We call the image the apartment attached to 
    $S$. Any two points of $\sbb(G)$ belong to a common apartment. 
    
      Assume that $G$ is semisimple.  Then, by a theorem of Curtis, Lehrer and Tits \cite[Prop.\ 6.1]{CLT},  there is a natural $G(k)$-equivariant bijection $\tau: |\Delta(G)|\to \sbb(G)$ between the geometric realization of $\Delta(G)$ and 
      $\sbb(G)$, such that for $b\in |\Delta(G)|$, the $k$-parabolic $P(\tau(b))$ is the one corresponding to the simplex of $\Delta(G)$ containing $b$ in its interior.  Hence the combinatorial building yields a triangulation of the spherical building. 
      
          For a $k$-rational parabolic subgroup $P\subset G$, we set
     $$
     \scc(P):=\{b\in \sbb(G):P(b)\supset P\}.
     $$
     We have $\scc(P)\subset \sbb(S)$, where $S$ is a maximal $k$-split torus contained in $P$.  We can think of $\scc(P)$ as parametrizing dominant $1$-{\rm PS} $\lambda$ of $P$ up to conjugation and ramification $\lambda\mapsto \lambda^n$. The map $\tau$ induces a homeomorphism between the closed simplex  of $|\Delta(G)|$ corresponding to $P$ and $\scc(P)$.
    If $P$ is a minimal $k$-parabolic than $\scc(P)$ is called a {\em chamber}; if $P$ is a proper maximal $k$-parabolic subgroup then $\scc(P)$ is called a vertex.  
    
\subsubsection{The slope function revisited}  
We return to the notation from Section \ref{notation1}. Let $S$ be a maximal $\Q_p$-split torus in $J_{\rm der}$. Fix $x\in\sff(K)$, for a field extension $K$ of $\breve{E}$. We can extend the slope function $\mu^{\sll}(x,\lambda)$  on $X_*({S})_{\Q}$ to a function on $X_*({S})_{\mathbf{R}}$ by using its description as  an infimum of values of certain rational  linear functionals \cite[Sec.\ 1]{RR1}. This description also implies  that the slope function $\mu^{\sll}(x,\lambda)$ is convex on the rational points of $\sbb(S)$ \cite[Cor.\ 2.15]{GIT}.

It turns out, as shown by Orlik, that the slope function  is, in fact (in a suitable sense), affine. To state  precisely what this means, for a chamber $\scc$ in $\sbb(J_{\rm der})$, we start with ``straightening it'', i.e., we deform it homeomorphically to the simplex
     $$
     \wtw{\scc}:=\{\sum_{\alpha\in\Delta}r_{\alpha}\lambda_{\alpha}| 0\leq r_{\alpha} \leq 1, \sum_{\alpha\in\Delta}r_{\alpha}=1\}\subset X_*(\wtw{S})_{\mathbf{R}},
     $$
    where the rational $1$-{\rm PS} $\lambda_{\alpha}\in X_*(\wtw{S})_{\Q}$, ${\alpha} \in\Delta$, for some maximal $\Q_p$-split torus $\wtw{S}\subset J_{\rm der}$, represent the vertices of $\scc$.   What we have gained doing this is  that the slope function $\mu^{\sll}(x,\lambda)$ on $X_*(\wtw{S})_{\mathbf{R}}$ while needing to be normalized to  $\nu^{\sll}(x,\lambda):=\mu^{\sll}(x,\lambda)/|\lambda|$ to descend to a function on $\sbb(\wtw{S})$ (after which it glues to a function on $\sbb(J_{\rm der})$) is defined on   $\wtw{\scc}$.  We say that $\mu^{\sll}(x,-)$ is {\em affine} on $\scc$ if it is affine on $\wtw{\scc}$, i.e., we have:
$$
\mu^{\sll}(x,\sum_{\alpha\in\Delta}r_{\alpha}\lambda_{\alpha})=\sum_{\alpha\in\Delta}r_{\alpha}\mu^{\sll}(x,\lambda_{\alpha}),\quad \text{ for all } \sum_{\alpha\in\Delta}r_{\alpha}\lambda_{\alpha}\in\wtw{\scc}.
$$
The following lemma is now a simple consequence of Lemma \ref{explicit}  that describes the slope function via the invariant inner product (which is linear in each variable):
\begin{lemma}[Orlik, {\cite[Prop.\ 2.3]{Of}}] \label{deszcz1}
Let $x\in \sff(K)$, for a field extension $K$ of $\breve{E}$. The slope function $\mu^{\sll}(x,-)$ is affine on each chamber of $\sbb(J_{\rm der})$.
\end{lemma}

     \subsubsection{Proof of Proposition \ref{dublin4}}
\begin{proof}
    Among the chambers of $\sbb(J_{{\rm der}})$ we will distinguish the {\em base chamber}
    $\scc_0:=\scc(P_0)$. We note here that $\Delta(J)=\Delta(J_{\rm der})$, hence $\sbb(J_{\rm der})$ is homeomorphic to $|\Delta(J)|$. The vertices of this chamber are
      the rational $1$-{\rm PS}  $\omega_{\alpha}$, $\alpha\in\Delta$; the corresponding parabolic subgroups
    \footnote{We have $P^J(\omega_{\alpha})(\overline{\Q}_p)=P(\omega_{\alpha})(\overline{\Q}_p)\cap J(\overline{\Q}_p)$.}  $P^J(\omega_{\alpha})$, $\alpha\in\Delta$, are the maximal $\Q_p$-rational parabolic subgroups that contain $P_0$.  If  $\scc=\scc(P)$ is a chamber in $\sbb(J_{\rm der})$ there exists a $g\in J(\Q_p)$,  unique up to multiplication by an element of $P_0(\Q_p)$ from the right, 
such that the conjugated $1$-{\rm PS} ${\rm Int}(g)\omega_{\alpha}$, $\alpha\in \Delta$,  are the vertices of $\scc$.      

  Proposition \ref{dublin4} follows now easily from Lemma \ref{deszcz1}.
\end{proof}
        
\subsubsection{Contractibility of a subcomplex of the spherical building}\label{contractibility}    Let $$Y:=\mathcal{F}\otimes_E E_s\setminus \mathcal{F}^{\rm wa}$$ be the complement of $\sff^{\rm wa}$. Let $x\in Y(K)$, for a field extension $K$ of $\breve{E}$. Consider the  subcomplex $T_x$ of the spherical buidling $\sbb(J_{\rm der})$ corresponding to the following set of vertices:
      $$
      \{gP(\omega_{\alpha})g^{-1}|g\in J(\Q_p),\alpha\in\Delta \mbox{ such that } \mu^{\sll}(x,{\rm Int}(g)\omega_{\alpha}) <0\}.
      $$
      We will need to know  that this subcomplex is contractible.  This follows from two facts:
      \begin{enumerate}
      \item Let 
      $$
      C_x:=\{\lambda\in \sbb(J_{\rm der})|\nu^{\sll}(x,\lambda) <0\}.
      $$
      This set is convex and the intersection of $C_x$ with each chamber in $\sbb(J_{\rm der})$ is convex \cite[Cor.\ 2.16]{GIT}. 
      \item Because the slope function $\mu^{\sll}(x,\lambda)$  is affine on every chamber of $\sbb(J_{\rm der})$ we have an inclusion $T_x\hookrightarrow C_x$. This is a deformation retract \cite[Lemma 3.4]{Of} (in fact,  we deform to a projection).
      \end{enumerate}              
                        
       \subsection{Stratification of the complement of $\mathcal{F}^{\rm wa}$} 
       The complement $Y$ of $\mathcal{F}^{\rm wa}$ in $\sff_{E_s}$ can be stratified using Schubert varieties, an essential result for the computation of its cohomology. In order to describe this stratification we will fully use the results presented above (we keep  the notation from Section \ref{notation1}).

    Each $\lambda\in X_*(J)_{\Q}$ determines a closed subvariety
 $$
 Y_{\lambda}:=\{x\in \sff|\,\mu^{\sll}(x,\lambda) <0\}
 $$
 of $\sff$, defined over $E_s$, consisting of points where $\lambda$ damages the semistability condition.
       In particular, each subset $I\subsetneq \Delta$ gives rise to a closed subvariety of 
       $\sff$, defined over $E_s$
 $$
 Y_I:=\bigcap_{\alpha\notin I} Y_{\omega_{\alpha}}
 $$
 and the properties of the slope function imply that 
the natural action of $J(\Q_p)$ on $\sff_{E_s}$ restricts to an action of $P_I(\Q_p)$ on $Y_I$ \cite[Lemma 3.1]{Ol},
 where
    $P_I=\cap_{\alpha\notin I} P^J(\omega_{\alpha})$. For example $P_{\Delta}=J$ and $P_{\emptyset}=P_0$, a fixed minimal $\qp$-parabolic subgroup of $J$ containing $S$.

   Let 
    $$X_I=J(\Q_p)/P_I(\Q_p),$$
    a compact $p$-adic analytic manifold, and let\footnote{The definition makes sense since 
     $P_I(\Q_p)$ preserves $Y_I$.}
     $$Z_I^{X_I}=\bigcup_{t\in X_I} tY_I^{\rm ad},$$
     a closed pseudo-adic subspace\footnote{We refer the reader to the appendix for the formalism of pseudo-adic spaces, due to Huber \cite{H1}.}  of $Y$ (since $X_I$ is compact, see \cite[Lemma 3.2]{Ol}). 
    Now, by Proposition \ref{dublin4} and the properties of the slope function $\mu^{\sll}(-,-)$ (see Section \ref{Hil-Mum}), we get the decomposition
     $$Y=\bigcup_{|\Delta\setminus I|=1} Z_I^{X_I}.$$
    
    \subsection{The cohomology of Schubert varieties} \label{Schubert1}
    
    Orlik used the results from GIT described in the previous sections to show (see \cite[Prop.\ 4.1]{Ol}, \cite[Prop.\ 4.1]{Of}) that  
    the varieties $Y_I$ are {\it Schubert varieties}, hence their cohomology is not difficult to compute.
            
    More precisely, fix 
    a Borel subgroup $B$ of $G$ contained in all the parabolics $P(\omega_{\alpha})$ and such that 
 $\mu$ belongs to the positive Weyl chamber with respect to $B$. Let $W=N_G(T)/T$ be the Weyl group, let $W_{\mu}\subset W$ be the stabilizer of $\mu$ and let $W^{\mu}$ be the set of Kostant representatives for $W/W_{\mu}$, i.e., representatives of shortest length in their cosets. The action of 
     $\sg_{E_s}$ on $W$ preserves $W^{\mu}$ (since $\mu$ is defined over $E_s$). 
     For $w\in W$, let $[w]\in W^{\mu}/\sg_{E_s}$ be its orbit. 
   
   Define ($\nu:=\nu_b$)
           $$\Omega_I=\{[w]\in W^{\mu}/\sg_{E_s}| \, \spp(w\mu-\nu, \omega_{\alpha})>0, \,\, \forall \alpha\notin I\},$$
            where $\spp$ is the fixed invariant inner product on $G$. Then $\Omega_{I\cap J}=\Omega_J\cap\Omega_J$ and a similar property holds for the varieties 
            $Y_I$. The Bruhat decomposition of the flag variety $\sff_{E_s}$ combined with the fact that, by Lemma \ref{explicit}, the semistability criterion can be expressed using the invariant inner product, yield the decomposition of $Y_I$ in terms of the Bruhat cells of $G$ with respect to $P(\mu)$:
 \begin{equation}\label{cells}
 Y_I=\bigcup_{[w]\in\Omega_I}BwP(\mu)/P(\mu).
 \end{equation}

     To each $\sg_{E_s}$-orbit 
  $[w]\in W^{\mu}/\sg_{E_s}$ we associate the following objects:

  $\bullet$ An integer $l_{[w]}$, the length of any element of $[w]$. 
  
  $\bullet$ The induced $\Z/p^n[\sg_{E_s}]$-module $\rho_{[w]}(\Z/p^n)$ of $\Z/p^n$-valued functions on the finite set
  $[w]$, with the natural $\sg_{E_s}$-action twisted (\`a la Tate) by $-l_{[w]}$.  Similarly, we define a  $\Z_p[\sg_{E_s}]$-module $\rho_{[w]}(\Z_p)$. 
     
     $\bullet$  The set
       $$I_{[w]}=\{\alpha\in \Delta| \, \, \spp(w\mu-\nu, \omega_{\alpha})\leq 0\}.$$
      It is the minimal subset of $\Delta$ such that $[w]\in\Omega_{I_{[w]}}$, thus we have 
      $[\omega]\in \Omega_I$ if and only if $I_{[w]}\subset I$.
 
          The Bruhat decomposition (\ref{cells}) yields the following computation of the \'etale  cohomology of the varieties $Y_I$. 
  \begin{corollary}\label{computation-cor}
 We have 
 \begin{align}
 \label{computation}
  & H^*_{\eet}(Y_{I,C},\Z/p^n)\simeq \bigoplus_{[w]\in\Omega_I}\rho_{[w]}(\Z/p^n)[-2l_{[w]}].
 \end{align}
 \end{corollary}
 \begin{proof}The computation in the $\ell$-adic setting in 
  \cite[Prop.\ 4.2]{Of} \cite[Prop.\ 7.1]{OS0} carries over to the $p$-adic setting, the key element being the standard form of the compactly supported cohomology of the affine space.
 
     The computation  goes as follows. For $0\leq i\leq m_I:=\max\{l_{[w]}|[w]\in\Omega_I\}$, set
 $$
 Y^i_I:=\bigcup_{[w]\in\Omega_I, l(w)\leq i}BwP(\mu)/P(\mu). $$
  We have a filtration by closed subvarieties
 $$
 Y_I=Y_I^{m_I}\supset Y_I^{m_I-1}\supset \cdots\supset Y^0_I\supset Y^{-1}_I:=\emptyset
 $$
such that 
 \begin{equation}
 \label{Bruhat}
 Y^i_I\setminus Y^{i-1}_I=\bigsqcup _{[w]\in\Omega_I, l(w)=i}BwP(\mu)/P(\mu).
 \end{equation}
  The Bruhat cell $BwP(\mu)/P(\mu)$ above is isomorphic to the affine space $\mathbb{A}^{i}_{E_s}$. Recall that
 $$
 H^j_{\eet, c}(\mathbb{A}^{i}_{C},\Z/p^n)\simeq\begin{cases}\Z/p^n(-i) & \mbox{for }  j=2i,\\
0 & \mbox{otherwise}. 
 \end{cases}
 $$
 The corollary follows now easily from  the long exact sequences (the induced Galois representations arise from the Bruhat decomposition (\ref{Bruhat}))
\begin{align*}
\cdots & \to H^{i-1}_{\eet, c}(Y^{j-1}_{I,C},\Z/p^n) \to H^i_{\eet, c}(Y^j_{I,C}\setminus Y^{j-1}_{I,C},\Z/p^n) \to H^i_{\eet, c}(Y^j_{I,C},\Z/p^n) \\
& \to H^i_{\eet, c}(Y^{j-1}_{I,C},\Z/p^n) \to \cdots \qedhere
\end{align*}
 
 \end{proof}
\begin{remark}\label{Bruhat2}The following observation will be useful later (see \cite[Sec.\ 4]{Of}). 
For  $[w]\in W^{\mu}/\sg_{E_s}$ and $I\subset \Delta$, let   $H(Y_I,[w])$ denote the part of  the direct sum (\ref{computation}), which comes from $[w]$, i.e.,
$$
H(Y_I,[w])=\begin{cases}
\rho_{[w]}(\Z/p^n)[-2l_{[w]}] &\mbox { if  } [w]\in \Omega_I,\\
0 &  \mbox{ if  } [w]\notin \Omega_I.
\end{cases}
$$
We have
$$
H^*_{\eet}(Y_{I,C},\Z/p^n)\simeq \bigoplus_{[w]\in W^{\mu}/\sg_{E_s}}H(Y_I,[w]).
$$

Let $I\subset J\subset \Delta$. Consider the projections
$$
p_{I,J}: H^*_{\eet}(Y_{J,C}, \Z/p^n)\to H^*_{\eet}(Y_{I,C}, \Z/p^n)
$$
induced by the closed embeddings $Y_I\hookrightarrow Y_J$.   The  proof of Corollary \ref{computation-cor} shows that they decompose into the direct sums:
$$
p_{I,J}\simeq \bigoplus_{([w],[w^{\prime}])\in (W^{\mu}/\sg_{E_s})^2}p_{[w],[w^{\prime}]}:\bigoplus_{[w]\in W^{\mu}/\sg_{E_s}}H(Y_J,[w])\to \bigoplus_{[w^{\prime}]\in W^{\mu}/\sg_{E_s}}H(Y_I,[w^{\prime}])
$$
with $
p_{[w],[w^{\prime}]}$ equal to the identity for 
$[w]=[w^{\prime}] $ and to zero otherwise.
\end{remark}

        \section{The main result} We are now ready to formulate and to prove the main result of this paper.
 
 \subsection{The setup} 
 Consider a local Shtuka datum $(G, [b], \{\mu\})$ over $F=\Q_p$, where 
$G$ is a {\it quasi-split} reductive group over $\Q_p$, 
$[b]\in B(G)$ is the $\sigma$-conjugacy class of a {\it basic} and $s$-decent\footnote{Recall that the hypothesis that $b$ is decent is harmless, since any $\sigma$-conjugacy class in $G(\breve{\Q}_p)$ contains a decent element.}         
 element 
$b\in G(\breveqp)$, and $\{\mu\}$ is a conjugacy class of cocharacters of $G$, with field of definition $E$ and associated flag variety $\mathcal{F}=\mathcal{F}(G, \{\mu\})$ defined over $E$. 

Letting $E_s=E\Q_{p^s}\subset \overline{\mathbf{Q}}_p$, the period domain  $\mathcal{F}^{\rm wa}=\mathcal{F}^{\rm wa}(G, [b], \{\mu\})$ 
is a nonempty\footnote{By the definition of a local Shtuka datum!} partially proper 
open subset of the adic space $\mathcal{F}_{\breve{E}}$, stable under the natural action of 
$J(\qp)\subset G(\breve{\Q}_p)$ on $\mathcal{F}_{\breve{E}}$ and having a canonical model over 
$E_s$. Thus the compactly  supported \'etale  cohomologies 
$H^*_{\eet,c}(\mathcal{F}^{\rm wa}_{C}, \Z/p^n)$ and  $H^*_{\eet,{\rm c}}(\mathcal{F}^{\rm wa}_{C}, \Z_p)$ are naturally 
 $J(\qp)\times \sg_{E_s}$-modules. Our goal is to describe these representations.

\subsubsection{Generalized Steinberg representations}  Fix a maximal $\Q_p$-split torus $S$ in $J=J_b$ and a minimal parabolic subgroup $B$ of $J$ defined over $\Q_p$ and having Levi component the centralizer of $S$.
This choice induces a parametrization $I \mapsto P_I$ of the standard parabolic subgroups of $J$ by subsets of $\Delta$, the set of simple roots associated to $B$.
In particular $P_{\emptyset}=B$ and $P_{\Delta}=J$.

We slightly change the notation introduced in Section \ref{extensions} and write simply
\begin{equation*}
X_I = J(\Q_p) / P_I(\Q_p), \quad i_I(A) = \LC(X_I, A), \quad  A = \Z/p^n, \Z_p, \Q_p,
\end{equation*}
where $\LC(X_I, A)$ is the space of locally constant (automatically with compact support since $X_I$ is compact) functions on $X_I$ with values in $A$.
We also let
\begin{equation*}
i^{\cont}_I(\Z_p) := \scc(X_I, \Z_p) = \varprojlim_{n} i_I(\Z/p^n),
\end{equation*}
where $\scc(X_I, \Z_p)$ is the space of continuous functions on $X_I$ with values in $\Z_p$.
The associated generalized smooth and continuous Steinberg representations of $J(\qp)$ will be denoted
\begin{gather*}
v_I(A) := i_I(A) / \sum_{I \subsetneq I'} i_{I'}(A), \quad A = \Z/p^n, \Z_p, \Q_p, \\
v^{\cont}_I(\Z_p) := i^{\cont}_I(\Z_p) / \sum_{I \subsetneq I'} i^{\cont}_{I'}(\Z_p) = \varprojlim_{n} v_I(\Z/p^n),
\end{gather*}
where the last equality follows from \cite[Lemme 3.3.3]{JHC} (in loc.\ cit.\ $J$ is split, but the results extend verbatim to any $J$).
Note that $i_{\Delta}=v_{\Delta}$ is simply the trivial representation of $J(\Q_p)$.

For two subsets $I \subset I' \subset \Delta$ with $|I' \setminus I| = 1$, we let $p_{I,I'} : i_{I'}(\Z/p^n) \to i_I(\Z/p^n)$ be the natural map induced by the surjection $X_I \to X_{I'}.$
For arbitrary subsets $I, I' \subset \Delta$ with $|I'| - |I| = 1$, we fix a numbering $I' = \{\beta_1,\ldots,\beta_r\}$ and we let
\begin{equation*}
d_{I,I'} =
\begin{cases}
(-1)^i p_{I,I'} & \text{if $I' = I \cup \{\beta_i\}$,} \\
0 & \text{if $I \not\subset I'$.}
\end{cases}
\end{equation*}
The following result is standard for complex coefficients.

\begin{proposition} \label{complex2}
Let $I \subset \Delta$ and consider the complex whose first term is in degree $-1$
\begin{equation*}
C_I(\Z/p^n) : i_{\Delta}(\Z/p^n) \to \bigoplus_{\substack{I \subset K \subset \Delta \\ |\Delta\setminus K| = 1}} i_K(\Z/p^n) \to \bigoplus_{\substack{I \subset K \subset \Delta \\ |\Delta\setminus K| = 2}} i_K(\Z/p^n) \to \cdots \to \bigoplus_{\substack{I \subset K \subset \Delta \\ |K\setminus I| = 1}} i_K(\Z/p^n) \to i_I(\Z/p^n)
\end{equation*}
with differentials induced by the $d_{K,K'}$.
Then $C_I(\Z/p^n)$ is a resolution of $v_{I}(\Z/p^n)$ by $J(\Q_p)$-modules.
\end{proposition}

\begin{proof}
The proof for complex coefficients goes over in our setup thanks to a multiplicity one result of Grosse-Kl\"onne \cite{GK}, Herzig \cite{Her}, and Ly \cite{LySt}, as we explain below.

Recall first the following simple fact from \cite[Sec.\ 2, Prop.\ 6]{SS}: if $(A_1,\ldots, A_m)$ is a family of subgroups of an abelian group $A$ such that
\begin{equation}
\label{condition1}
\big(\sum_{i \in T} A_i \big) \cap \big( \bigcap_{j \in S} A_j \big) = \sum_{i \in T} \big( A_i \cap \big( \bigcap_{j \in S} A_j \big) \big)
\end{equation}
for all subsets $T, S\subset \{1,\dots, m\}$, then the natural complex
\begin{equation*}
\cdots \to \bigoplus_{1 \leq i < j < k \leq m} A_i \cap A_j \cap A_k \to \bigoplus_{1 \leq i < j \leq m} A_i \cap A_j \to \bigoplus_{1 \leq i \leq m} A_i \to A
\end{equation*}
is a resolution of the subgroup $\sum_i A_i$ of $ A$.

Now assume that the abelian group $A$ has finite length and say that a subgroup $A' \subset A$ is \emph{isotypically closed} if each irreducible constituent of $A'$ has the same multiplicity in $A'$ and $A$.
Such a subgroup is uniquely determined by its irreducible constituents.
If two subgroups $A', A'' \subset A$ are isotypically closed, then $A' \cap A''$ and $A' + A''$ are isotypically closed too, and their irreducible constituents are, respectively, the intersection and the union of the irreducible constituents of $A'$ and $A''$.
We deduce that if $A_i$ is isotypically closed for all $i$, then the left and right sides of \eqref{condition1} are isotypically closed and have the same irreducible constituents (as union distributes over intersection), hence they are equal.

We apply this fact with $A = i_{\emptyset}(\Z/p^n)$ and the $J(\Q_p)$-submodules $i_K(\Z/p^n)$ for $I \subset K \subset \Delta$, $|\Delta \setminus K| = 1$.
The latter are isotypically closed in $i_{\emptyset}(\Z/p^n)$ by Proposition \ref{AHV} below.
\end{proof}

\begin{proposition}[Grosse-Kl\"onne, Herzig, Ly] \label{AHV}
The irreducible constituents of $i_K(\Z/p^n)$ are the representations $v_I(\Z/p)$ for $K \subset I \subset \Delta$, each occurring with multiplicity $n$.
\end{proposition}

\begin{proof}
By devissage, the result reduces to the case $n = 1$, which is due to  Grosse-Kl\"onne \cite{GK} and Herzig \cite{Her} when $\Gb$ is split, and to Ly \cite{LySt} in general.
\end{proof}

\subsubsection{The main result}
          
          The following theorem is the principal result of this paper.
                    
          \begin{theorem}\label{main}
           Let $(G,[b], \{\mu\})$ be a local Shtuka datum with $G/\qp$ reductive and quasi-split, $b\in G(\breve{\Q}_p)$ basic and $s$-decent. Suppose that $p \neq 2$.
           \begin{enumerate}
           \item There is an isomorphism of 
           $\sg_{E_s}\times J(\qp)$-modules 
          \begin{equation}
          \label{main-1}
 H^*_{\eet,c}(\sff^{\rm wa}_{C},\Z/p^n)\simeq \bigoplus_{[w]\in W^{\mu}/\sg_{E_s}} v_{{I_{[w]}}}(\Z/p^n)\otimes \rho_{[w]}(\Z/p^n)[-n_{[w]}],
 \end{equation}
 where $n_{[w]}=2l_{[w]}+|\Delta\setminus I_{[w]}|$.
  \item There is an isomorphism of 
           $\sg_{E_s}\times J(\qp)$-modules 
          \begin{equation}
          \label{main-3}
 H^*_{\eet,{\rm c}}(\sff^{\rm wa}_C,\Z_p)\simeq \bigoplus_{[w]\in W^{\mu}/\sg_{E_s}} v^{\cont}_{{I_{[w]}}}(\Z_p)\otimes \rho_{[w]}(\Z_p)[-n_{[w]}].
 \end{equation}
 \end{enumerate}
          \end{theorem}
          \begin{remark}
          
          \begin{enumerate}

         \item  Orlik shows in \cite[Th.\ 1.1]{Ol}  that the isomorphism (\ref{main-1}) 
          holds for  compactly supported \'etale cohomology with coefficients $\Z/\ell^n$, $\ell\neq p$ (there are additional assumptions on $\ell$ needed, see \cite[Sec.\ 1]{Ol}). This implies (by the same argument as we use below) the isomorphism (\ref{main-3}) for compactly supported \'etale cohomology with coefficients $\Z_{\ell}$. 
       \item Just as in \cite[Cor.\ 10.3.7]{DOR} one obtains strong vanishing results for 
       $ H^*_{\eet,{c}}(\sff^{\rm wa}_C,\Z/p^n)$ from the above theorem, as well as a simple description of the 
       top-degree cohomology. 
       
       \end{enumerate}
       \end{remark}
          
        \subsubsection{The case of the Drinfeld space} \label{main-ex} Let us discuss an example. Let ${\mathbb H}^d_{\Q_p}$ be  the Drinfeld symmetric space of dimension 
  $d$ over $\Q_p$.     Recall that 
  ${\mathbb H}^d_{\Q_p}=\mathbb{P}^{d}_{\Q_p}\setminus \cup_{H\in \mathcal{H}} H$, where $\mathcal{H}$ is the set of $\Q_p$-rational hyperplanes. Set $G:=\mathbb{GL}_{d+1,\Q_p}$ (as we will see in the proof, Theorem \ref{main} holds true for $p=2$ in this case).
  \begin{corollary}
  \begin{enumerate}
   \item There is an isomorphism of 
           $\sg_{\Q_p}\times G(\qp)$-modules 
          \begin{equation*}
 H^*_{\eet,c}({\mathbb H}^d_{C},\Z/p^n)\simeq \bigoplus_{i=0}^{d} {\rm Sp}_{d-i}(\Z/p^n)(-i)[-d-i].
 \end{equation*}
  \item There is an isomorphism of 
           $\sg_{\Q_p}\times G(\qp)$-modules 
          \begin{equation*}
 H^*_{\eet,{\rm c}}({\mathbb H}^d_C,\Z_p)\simeq \bigoplus_{i=0}^{d} {\rm Sp}^{\cont}_{d-i}(\Z_p)(-i)[-d-i].
 \end{equation*}
 \end{enumerate}
 Here the generalized Steinberg representations ${\rm Sp}_{d-i}(\Z/p^n)$ and ${\rm Sp}^{\cont}_{d-i}(\Z_p)$ are  as defined in the proof.
          \end{corollary}
\begin{proof}By Example \ref{Drinfeld1},  ${\mathbb H}^d_{\Q_p}$ is the period domain corresponding to $b=1, \{\mu\}=(d, (-1)^d\}$. We have $E=E_s=\Q_p$ and $J=G$.

Let $B$ be the upper triangular Borel subgroup of $G$ and 
   let $\Delta=\{\alpha_1,\alpha_2,\dots,\alpha_d\}$ be the set of relative simple roots: $\alpha_i({\rm diag}(t_1,\ldots, t_{d+1}))=t_it_{i+1}^{-1}$. 
   We identify the Weyl group $W$ of $G$ with the group of permutations of $\{1,2,\dots,d+1\}$ and with the subgroup of permutation matrices
   in $G$. Then $W$ is generated by the elements $s_i=(i, i+1)$ for $0\leq i\leq d+1$.  
  The  set $W^{\mu}$ of Konstant representatives for $W/W_{\mu}$ consists of: 
  $$ w_0=1, w_1=s_1, w_2=s_2s_1, \dots, w_d=s_ds_{d-1}\cdots s_1,\quad l_{[w_i]}=i.
   $$
  We have 
 $\Delta\setminus I_{[w_i]}=\{\alpha_{i+1},\ldots, \alpha_d\}, |\Delta\setminus I_{[w_i]}|=d-i .$ And $\rho_{[w_i]}(\Z/p^n)=\Z/p^n(-i)$, $\rho_{[w_i]}(\Z_p)=\Z_p(-i)$.

Set ${\rm Sp}_{d-i}(\Z/p^n):=v^{G}_{I_{[w_i]}}(\Z/p^n)$ and  ${\rm Sp}^{\cont}_{d-i}(\Z_p):= v^{G,\cont}_{I_{[w_i]}}(\Z_p)$.
Then our corollary follows from Theorem \ref{main}.
\end{proof}
  We computed  that: 
  \begin{align*}
  H^i_{\eet,c}({\mathbb H}^d_{C},\Z/p^n) \simeq {\rm Sp}_{2d-i}(\Z/p^n)(d-i),\quad H^i_{\eet,{\rm c}}({\mathbb H}^d_{C},\Z_p) \simeq {\rm Sp}^{\cont}_{2d-i}(\Z_p)(d-i).
  \end{align*}
  In particular: 
  \begin{align*}
   & H^i_{\eet,c}({\mathbb H}^d_{C},\Z/p^n) \simeq H^i_{\eet,{\rm c}}({\mathbb H}^d_C,\Z_p)=0,\quad 0\leq i\leq d-1,\\
      & H^{2d}_{\eet,c}({\mathbb H}^d_{C},\Z/p^n) \simeq \Z/p^n(-d),\quad H^{2d}_{\eet,{\rm c}}({\mathbb H}^d_C,\Z_p)\simeq \Z_p(-d).
  \end{align*}
  Recall that in \cite[Th.\ 1.1]{CDN4} we have computed that:
   $$
   H^i_{\eet}({\mathbb H}^d_{C},\Z/p^n)\simeq  {\rm Sp}_i(\Z/p^n)^*(-i),\quad   H^{i}_{\eet}({\mathbb H}^d_{C},\Z_p)\simeq  {\rm Sp}^{\cont}_i(\Z_p)^*(-i).
   $$
      Hence (${}^*$ denotes the linear dual):
      
      \begin{corollary} There is a duality isomorphism of  $\sg_{\Q_p}\times G(\qp)$-modules
        $$
       H^i_{\eet}({\mathbb H}^d_{C},\Z/p^n)(d)\simeq  H^{2d-i}_{\eet,c}({\mathbb H}^d_{C},\Z/p^n)^*,\quad  H^i_{\eet}({\mathbb H}^d_{C},\Z_p)(d)\simeq  H^{2d-i}_{\eet,{\rm c}}({\mathbb H}^d_{C},\Z_p)^*.
        $$
        \end{corollary}

 \subsection{Orlik's fundamental complex}    In this section we will define a resolution of the constant sheaf $\Z/p^n$ on the complement of the period domain $\sff^{\rm wa}_C$ and compute the cohomology of its terms. The key geometric ideas are due to Orlik, see \cite{Ol}.
 
 \subsubsection{Stratification of the constant sheaf}\label{strat}
         Let ${\rm Sh}(Y_{C,\eet})$ denote the topos of \'etale sheaves on $Y_C$. We refer the reader to the appendix for the review of the formalism of \'etale cohomology of pseudo-adic spaces developed by Huber in 
         \cite{H1}. If $X$ is an algebraic variety over $C$ (for instance $Y_{I,C}$), let $X^{\rm ad}$ be the associated adic space over 
         ${\rm Spa}(C,\mathcal{O}_C)$.

                 If $Z$ is a closed pseudo-adic subspace of 
          $Y_C$ and $i: Z\to Y_C$ is the inclusion, define 
          $F_Z:=i_*i^*F\in {\rm Sh}(Y_{C,\eet})$ for any sheaf 
  $F\in  {\rm Sh}(Y_{C,\eet})$. 
   Fix a subset $I\subset \Delta=\{\alpha_1,\ldots,\alpha_d\}$. If $T$ is a compact open subset 
          of $X_I=J(\qp)/P_I(\qp)$, then $$Z^T_{I,C}:=\bigcup_{t\in T} tY_{I,C}^{\rm ad}$$ is a closed pseudo-adic subspace of $Y_C$ thanks to the compactness of $T$ (see \cite[Lemma 3.2]{Ol}) and for any $F\in  {\rm Sh}(Y_{C,\eet})$ we have a natural injection 
          $F_{Z^T_{I,C}}\subset \prod_{t\in T} F_{tY_{I,C}^{\rm ad}}$. Therefore any partition 
          $X_I=\coprod_{a\in A} T_a$ by (nonempty) compact open subsets\footnote{Note that $A$ is necessarily finite since 
          $X_I$ is compact.} induces an embedding $$\bigoplus_{a\in A} F_{Z^{T_a}_{I,C}}\subset
          \prod_{x\in X_{I}} F_{xY_{I,C}^{\rm ad}}.$$ 
          
          \begin{definition}
           If $F\in {\rm Sh}(Y_{C,\eet})$, define $F_I\in {\rm Sh}(Y_{C,\eet})$ as the {\em subsheaf of locally constant sections} of $\prod_{x\in X_I} F_{xY_{I,C}^{\rm ad}}$, i.e., 
           $$F_I=\varinjlim_{c\in \scc_I} F_c,$$
           the limit being taken over the (pseudo-filtered) category $\scc_I$ of compact open disjoint coverings of 
           $X_I$ ordered by refinement and $F_c$ for $c=\{T_j\}_{j\in A}\in \scc_I$ being the image of the natural embedding $
 \bigoplus_{j\in A}F_{Z_{I,C}^{T_j}}\hookrightarrow \prod_{x\in X_I} F_{xY_{I,C}^{\rm ad}}.
 $
            
          \end{definition}
          
          The following computation will be essential for us later.
           \begin{proposition} 
          \label{locally-constant}
\begin{enumerate}
\item  If $\bar{x}$ is a geometric point of $Y_C$ with support $x\in Y_C$, then, 
          for all $F\in {\rm Sh}(Y_{C,\eet})$, we have a natural isomorphism
          $$(F_I)_{\bar{x}}\simeq \LC(X_I(x), F_{\bar{x}}),\,\, \text{where}\,\, 
        X_I(x)=\{g\in X_I| \, x\in gY^{\rm ad}_{I,C}\}.$$
  \item Let $i\in\N$. We have
 $$
 H^i_{\eet}(Y_C,(\Z/p^n)_{I})\simeq\LC(X_I,H^i_{\eet}(Y_{I,C},\Z/p^n))\simeq i_{P_I}(\Z/p^n)\otimes H^i_{\eet}(Y_{I,C},\Z/p^n).
 $$

  \end{enumerate}        
           \end{proposition}
 \begin{proof} The first claim follows from the definition of $F_I$.
 Orlik's proof \cite[Prop.\ 4.3]{Ol} of the second claim for $\ell$-adic sheaves goes through in our case, the key points of commuting with inductive and projective limits as well as the comparison algebraic-analytic being also valid $p$-adically. We describe briefly the essential steps of the proof. Since $Y$ is quasi-compact and 
 $\scc_I$ is pseudo-filtered, \cite[2.3.13]{H1} yields an isomorphism
  $$
 H^i_{\eet}(Y_C, (\Z/p^n)_{I})\simeq\varinjlim_{c\in\scc_I}H^i_{\eet}(Y_C,(\Z/p^n)_c)\simeq \varinjlim_{c\in\scc_I}(\bigoplus_{T\in c}H^i_{\eet}(Z^{T}_{I,C}, (\Z/p^n))).
 $$

    We can write $Y_{I,C}^{\rm ad}=\bigcap_{s\in\N}Z_{I,C}^{T_s}$ \cite[Lemma 4.4]{Ol}, for a family of compact open neighborhoods of the point $[P_I]$ of $X_I$ such that $\cap_{s\in\N}T_s=[P_I]$. Hence, by \cite[2.4.6]{H1},
    $$
    \varinjlim_{s\in\N}H^i_{\eet}(Z_{I,C}^{T_s}, \Z/p^n)\simeq H^i_{\eet}(Y_{I,C}^{\rm ad}, \Z/p^n)\simeq H^i_{\eet}(Y_{I,C}, \Z/p^n),
    $$
    the last isomorphism being a consequence of Huber's comparison Theorem \cite[Th.\ 3.7.2]{H1}.
Combining the above we get
$$
H^i_{\eet}(Y_C,F_I)\simeq \varinjlim_{c\in\scc_I}(\bigoplus_{T\in c}H^i_{\eet}(Z^{T}_{I,C}, \Z/p^n))\simeq \LC(X_I,H^i_{\eet}(Y_{I,C},\Z/p^n)),
$$ as desired.
\end{proof}

          \subsubsection{Acyclicity of the fundamental complex}
            We will explain now how to create a complex out of the various $F_I$ for $I\subset \Delta$. 
            Choose an ordering on $\Delta$ and fix $F\in {\rm Sh}(Y_{C,\eet})$. We will construct (following Orlik) maps $d_{I,I'}: F_{I'}\to F_I$ for all subsets 
            $I,I'$ of $\Delta$ with $|I'|-|I|=1$, inducing 
            the {\it fundamental complex}\footnote{Check \cite[Ch.\ XI]{DOR} for a ``geometric'' construction of this complex.}: 
            \begin{equation}
            \label{complex1}
            C(F): \quad 0\to F\to \bigoplus_{|\Delta\setminus I|=1} F_I\to \bigoplus_{|\Delta\setminus I|=2} F_I\to\cdots
              \to \bigoplus_{|\Delta\setminus I|=|\Delta|-1} F_I\to F_{\emptyset}\to 0.
              \end{equation}
            We set $d_{I,I'}=0$ when  
            $I$ is not a subset of $I'$, so suppose that $I\subset I'$.
             We have a natural surjective map $p_{I,I'}: X_I\to X_{I'}$ (induced by $P_I\subset P_{I'}$), and for all 
             $x\in X_{I'}, y\in X_I$ we have a natural map $F_{xY_{I'}}\to F_{yY_I}$, namely the zero map if $p_{I,I'}(y)\ne x$ and the map induced by the closed embedding $yY_{I}\to 
             xY_{I'}$ otherwise. Unwinding  the definitions of $F_I$ and $F_{I'}$, we obtain a natural map $p_{I,I'}: F_{I'}\to F_I$, and we  set $d_{I,I'}=(-1)^i p_{I,I'}$ if $I'=\{\alpha_1<\dots<\alpha_r\}$ and $I=I'\setminus \{\alpha_i\}$.

         Recall that a sheaf $F$ on $Y_{C,\eet}$ is called {\it overconvergent} if the specialization morphism $F_{\xi_1}\to F_{\xi_2}$ is an isomorphism
               for any specialization of geometric points 
         $\xi_2\to \xi_1$. If $F$ is an overconvergent  sheaf, then it is easy to see that the terms of $C(F)$ are overconvergent \cite[Lemma 3.4]{Ol}.

             \begin{theorem} 
              If $F$ is an overconvergent sheaf on $Y_{C, \eet}$, the fundamental complex 
              $C(F)$ is  acyclic. 
                           \end{theorem}
             \begin{proof}
             The arguments from the proofs of \cite[Th.\ 3.3]{Ol} and \cite[Th.\ 2.1]{Oc}, for $\ell$-adic sheaves, also work in this setting. We sketch them briefly. By overconvergence, it suffices to check the acyclicity of the stalk $C(F_{\eta})$ of 
             $C(F)$ at a maximal geometric point 
             $\eta: {\rm Spa}(K, \mathcal{O}_K)\to \mathcal{F}_C$. If $x_{\eta}\in \sff(K)$ is the induced point, the complex 
             $C(F_{\eta})$ is given, thanks to Proposition \ref{locally-constant}, by 
             \begin{equation}
\label{kicius1}
0\to F_{x_{\eta}}\to \bigoplus_{|\Delta\setminus I|=1} \LC(X_I(x_{\eta}),F_{{x_{\eta}}})\to \cdots
              \to \bigoplus_{|\Delta\setminus I|=|\Delta|-1} \LC(X_I(x_{\eta}),F_{x_{\eta}})\to \LC(X_{\emptyset}(x_{\eta}),F_{{x_{\eta}}})\to 0
\end{equation}
where we recall that $X_I=J(\Q_p)/P_I(\Q_p)$, $X_I(x_{\eta})=\{g\in X_I| \, x_{\eta}\in g Y_I(K)\}$, and $P_{\emptyset}=P_0$. 
                     But this is a complex of locally constant functions (with values in $F_{x_{\eta}}$)   on  a subcomplex of the combinatoral  building of $J$, whose simplices are given by 
          $$
          \{gP_Ig^{-1}|g\in J(\Q_p), x_{\eta}\in gY_I(K), I \subsetneq \Delta\}.
          $$
          Its geometric realization (via the map $\tau$ from Section \ref{buildings1}) is the subcomplex $T_{x_\eta}$ of the spherical building $\sbb(J_{\rm der})$ from Section \ref{contractibility}. Since the complex $T_{x_\eta}$ is contractible, by \cite[Rem.\ 66]{SS}, the complex (\ref{kicius1}) is acyclic.
                       \end{proof}

\subsection{The key spectral sequence} We will now  evaluate the spectral sequence  $E_1$ induced by the acyclic complex (\ref{complex1}) for $F=\Z/p^n$: 
 $$
 E_1^{i,j}=H^j_{\eet}(Y_{C},\bigoplus_{|\Delta\setminus I|=i+1}(\Z/p^n)_{I})\Rightarrow H^{i+j}_{\eet}(Y_{C},\Z/p^n).
 $$
 In order to simplify some of the rather complicated formulae below we introduce the shorthands: 
$$i_{I,n}:=i_{I}(\Z/p^n),\,\, v_{I,n}:=v_I(\Z/p^n),\,\, \rho_{[w],n}:=\rho_{[w]}(\Z/p^n).$$

 \begin{lemma}\label{degeneration}
The above spectral sequence degenerates at $E_2$.
\end{lemma}
\begin{proof} 
 By Proposition \ref{locally-constant} we have
 $$
 E_1^{i,j}\simeq\bigoplus_{ |\Delta\setminus I|=i+1}\LC(X_I,H^{j}_{\eet}(Y_{I,C},\Z/p^n))\simeq \bigoplus_{ |\Delta\setminus I|=i+1}  i_{I,n}\otimes H^j_{\eet}(Y_{I,C},\Z/p^n).
 $$
From the Bruhat decomposition of $H^{j}_{\eet}(Y_{I,C},\Z/p^n)$ obtained  in Corollary \ref{computation-cor}, using Remark \ref{Bruhat2}, we obtain a decomposition of the spectral sequence $E_1$:
  $$
  E_1=\bigoplus_{[w]\in W^{\mu}/\sg_{E_s}}E_{1,[w]},
  $$
  where $E_{1,[w]}$ is the complex living just in row $2l_{[w]}$:
  $$ E_{1,[w]}=(\bigoplus_{\substack{I_{[w]}\subset I \\ |\Delta\setminus I |=1}}  i_{I,n}\otimes\rho_{[w],n}\to \bigoplus_{\substack{I_{[w]}\subset I \\ |\Delta\setminus I|=2}}i_{I,n}\otimes\rho_{[w],n}\to \cdots \to i_{{I_{[w]},n}}\otimes\rho_{[w],n})[-2l_{[w]}].$$
  We get an exact sequence of complexes\footnote{Recall that the complexes 
  $C_I(\Z/p^n)$ are defined in Proposition \ref{complex2}.}
  $$
  0\to i_{\Delta, n}\otimes\rho_{[w],n}[-2l_{[w]}+1]\to C_{I_{[w]}}(\Z/p^n)\otimes\rho_{[w],n}[-2l_{[w]}]\to E_{1,[w]}\to 0.
  $$
    
By Proposition \ref{complex2} and noting that $\rho_{[w],n}$ is a free $\Z/p^n$-module, the only nonzero terms of $ E_{2,[w]}$ are given as follows: for $|\Delta\setminus I_{[w]}|=1$ 
$$E_{2,[w]}^{0,2l_{[w]}} \simeq i_{{I_{[w]},n}}\otimes\rho_{[w],n}$$
and for $|\Delta\setminus I_{[w]}|>1$
$$E_{2,[w]}^{0,2l_{[w]}} \simeq i_{\Delta,n}\otimes\rho_{[w],n} \quad \text{and} \quad E_{2,[w]}^{i,2l_{[w]}}\simeq v_{{I_{[w]},n}}\otimes\rho_{[w],n}, \ i=|\Delta\setminus I_{[w]}|-1.$$

To see that $E_2=E_{\infty}$ note that the nontrivial differentials on $E_i$, $i\geq 2$, can only go from $E_{i,[w]}$ to $E_{i,[w^{\prime}]}$ for $[w]\neq [w^{\prime}]$, $l_{[w]}\neq l_{[w^{\prime}]}$. But such maps have to be trivial by a weight argument.
 Indeed,  all the terms of $E_2$ are free modules over  $\Z/p^n$ and are given by the tensor product of a $J(\Q_p)$-module and a $\sg_{E_s}$-module. We also have $E_2(\Z/p^n)\otimes_{\Z/p^n}\Z/p^m\simeq E_2(\Z/p^m)$ for any integer $n>m$. Thus the nontrivial  differentials of the spectral sequence $E_i$, $i\geq 2$, would yield nontrivial maps between projective systems $(\Z/p^n(a))_n$ and $(\Z/p^n(b))_n$, $a\neq b$. 
 And this is not possible (see \cite[Sec.\ 5]{Ol} for details).
\end{proof}

   The next proposition crucially uses the results of Section \ref{extensions}. 
   
 \begin{proposition}\label{cohomology}
We have
 \begin{align*}
 H^*_{\eet}(Y_C,\Z/p^n)\simeq & 
\big (\bigoplus_{|\Delta\setminus I_{[w]}|=1}i_{{I_{[w]},n}}\otimes\rho_{[w],n} [-2l_{[w]}]\big )\bigoplus\\
 & \bigoplus_{|\Delta\setminus I_{[w]}|>1}\big ((i_{\Delta,n}\otimes\rho_{[w],n}[-2l_{[w]}])
\oplus (v_{{I_{[w]},n}}\otimes\rho_{[w],n}[-2l_{[w]}-|\Delta\setminus I_{[w]}|+1]) \big ).
\end{align*}

 \end{proposition}
\begin{proof}
 Fix $j\in\N$. Using Lemma \ref{degeneration} and its proof we can compute the grading induced by the spectral sequence $E_1$:
 \begin{align*}
 & \gr^i(H^j_{\eet}(Y_C,\Z/p^n))=E^{i,j-i}_{\infty}=E^{i,j-i}_{2}\simeq \bigoplus_{[w]\in W^{\mu}/\sg_{E_s}}E^{i,j-i}_{2,[w]}\\
&  \simeq \begin{cases}\big (\bigoplus_{[w]\in T_1}i_{{I_{[w]},n}}\otimes\rho_{[w],n}\big )\bigoplus\big (\bigoplus_{[w]\in T_2}i_{\Delta,n}\otimes\rho_{[w],n}\big ) & \mbox{ if } i=0,\\
\bigoplus_{[w]\in T_3}v_{{I_{[w]},n}}\otimes\rho_{[w],n} & \mbox{ if } i>0,
 \end{cases}
 \end{align*}
 where we set 
 \begin{align*}
 & T_1:=\{[w]\in W^{\mu}/\sg_{E_s}| |\Delta\setminus I_{[w]}|=1, 2l_{[w]}=j\}, \quad T_2:=\{[w]\in W^{\mu}/\sg_{E_s}| |\Delta\setminus I_{[w]}|>1, 2l_{[w]}=j\},\\
 & T_3:=\{[w]\in W^{\mu}/\sg_{E_s}| 2l_{[w]}+|\Delta\setminus I_{[w]}|-1=j, i=|\Delta\setminus I_{[w]}|-1\}.
 \end{align*}

 It suffices to show that this grading splits. We start by proving the following:
 
 \begin{lemma}
  The cohomology groups $H^i_{\eet}(Y_C,\Z/p^n)$ are smooth $J(\Q_p)$-modules.
 \end{lemma}
 
 \begin{proof} We start by observing that 
  $Y_C$ is proper, being quasi-compact and closed in $\sff_C$, which is proper (see \cite[Lemma 5.3.3]{H1}). 
  It follows that $ \rg_{\eet,c}(Y_C,\Z/p^n)\simeq  \rg_{\eet}(Y_C,\Z/p^n)$ and the  
   distinguished triangle associated to the triple $(\sff^{\rm wa},\sff,Y)$ becomes 
 \begin{equation}
 \label{zima1}
  \rg_{\eet,c}(\sff^{\rm wa}_C,\Z/p^n)\lomapr{ } \rg_{\eet}(\sff_C,\Z/p^n)\lomapr{i^*} \rg_{\eet}(Y_C,\Z/p^n).
  \end{equation}

   Consider the induced long exact sequence of cohomology groups.  By a result of Berkovich \cite[Cor.\ 7.8]{B}, we know that 
   $H^i_{\eet,c}(\mathcal{F}_C^{\rm wa}, \Z/p^n)$ is a smooth $J(\Q_p)$-module. We note  that one cannot apply this result directly  to $Y$, which is only a pseudo-adic space. And, of course,  $H^i_{\eet}(\mathcal{F}_C, \Z/p^n)$ is a smooth   representation of $J(\Q_p)$, of finite type over $\Z/p^n$. Our lemma follows then by induction on $i$, using the next lemma.
    \end{proof}
 
 \begin{lemma}
  Let $G$ be a $p$-adic analytic group and let $\pi$ be a representation of $G$ over 
  $\Z/p^n$ living in an exact sequence 
  $$0\to \sigma\to \pi\to \tau\to 0,$$
  with $\sigma$ and $\tau$ smooth representations of $G$ over $\Z/p^n$ and 
  $\sigma$ finitely generated as $\Z/p^n$-module. Then 
  $\pi$ is a smooth representation of $G$.
 \end{lemma}
 
 \begin{proof}  Since $\sigma$ is finitely generated over 
 $\Z/p^n$ and smooth, there is an open subgroup $H$ of 
 $G$ acting trivially on $\sigma$. Replacing $G$ by $H$ we may thus assume that 
 $G$ acts trivially on $\sigma$. Let $v\in \pi$. Since 
 $\tau$ is smooth, there is an open subgroup $K$ of $G$ fixing the image of 
 $v$ in $\tau$. Shrinking $K$, we may assume that $K$ is a uniform pro-$p$ group. Since $kv-v\in \sigma$, for $k\in K$, we have 
 $(g-1)(k-1)v=0$ for $g\in G$ and $k\in K$. Using the binomial formula and the fact that 
 $p^n$ kills $\pi$, it follows that 
 $k^{p^n}v=v$ for $k\in K$. Since $K^{p^n}$ is open in $K$ and thus in $G$, we are done.  
 \end{proof}

 To show that the filtration for $i>0$ splits, consider the equation 
 $$
 2l_{[w]} + |\Delta\setminus I_{[w]}|-1=j= 2l_{[w^{\prime}]} + |\Delta\setminus I_{[w^{\prime}]}|-1
 $$
 with $[w],[w^{\prime}]\in W^{\mu}/\sg_{E_s}$. If $l_{[w]}\neq l_{[w^{\prime}]}$ this equation implies that  $|\Delta\setminus I_{[w]}|$ and $|\Delta\setminus I_{[w^{\prime}]}|$ differ by at least two. Hence $|I_{[w]}|$ and $|I_{[w^{\prime}]}|$ differ by at least two as well. Since $H^j_{\eet}(Y,\Z/p^n)$ are smooth $J(\Q_p)$-modules and the $\Ext$ group $\Ext^1_{J(\Q_p)}(v_{{I_{[w]}}},v_{{I_{[w^{\prime}]}}} )$
 in the category of smooth representations is trivial by Theorem \ref{vanishing-intro} (here we use the hypothesis $p \neq 2$), we have a splitting of $J(\Q_p)$-modules.
In the case $G = \mathbb{GL}_{n,\Q_p}$ the result holds true for $p=2$ as well (see Remark \ref{vanishing-GLn}).
  
 This splitting is automatically Galois equivariant: call the section $s$ and consider $gs$, $g\in \sg_{E_s}$. The map $t:=s-gs$ decomposes into a direct sum of maps between generalized Steinberg representations $v_{I}$ and $v_{{I^{\prime}}}$ with $I$ and $I^{\prime}$ differing by at least two elements. Hence, by Proposition \ref{prop:HomSt}, all these maps are trivial and thus  $t=0$, as wanted. 
 
 To include $i=0$, we start by showing that 
 \begin{gather}
\Ext^1_{J(\Q_p)}(i_{{I_{[w]},n}}, v_{{I_{[w^{\prime}],n}}})=0,\quad w\in T_1, w^{\prime}\in T_3, \label{vanishing1}\\
\Ext^1_{J(\Q_p)}(i_{\Delta,n}, v_{{I_{[w^{\prime}],n}}})=0,\quad   w^{\prime}\in T_3. \label{vanishing2}
 \end{gather}
 For the first equality, consider the exact sequence (recall that $|\Delta\setminus I_{[w]}|=1$)
 \begin{equation}
 \label{kolo-kwak}
 0\to i_{\Delta,n}\to i_{{I_{[w]},n}}\to v_{{I_{[w]},n}}\to 0.
 \end{equation}
 It yields the exact sequence
 $$
\Ext^1_{J(\Q_p)}(v_{{I_{[w]},n}},v_{{I_{[w^{\prime}],n}}})\to \Ext^1_{J(\Q_p)}( i_{{I_{[w]},n}}, v_{{I_{[w^{\prime}],n}}})\to \Ext^1_{J(\Q_p)}(i_{\Delta,n}, v_{{I_{[w^{\prime}],n}}}).
 $$
 Since
 $$
 |\Delta\setminus I_{[w^{\prime}]}|=-2l_{[w^{\prime}]}+1+j=-2l_{[w^{\prime}]}+ 2l_{[w]}+1\geq 3
 $$
 and $ |\Delta\setminus I_{[w]}|=1$,  $|I_{[w]}|$ and $|I_{[w^{\prime}]}|$ differ by at least two and the first term in the above sequence is zero by Theorem \ref{vanishing-intro}; similarly, since $|\Delta|$ and $|I_{[w^{\prime}]}|$ differ by at least two, the right term of the above sequence is zero by Theorem \ref{vanishing-intro}.
 Here we used again the hypothesis $p \neq 2$ (which is not needed in the case $G = \mathbb{GL}_{n,\Q_p}$, see Remark \ref{vanishing-GLn}).
 Hence we have obtained \eqref{vanishing1} and along the way we have shown \eqref{vanishing2} as well.
 
   Moreover, the $J(\Q_p)$-equivariant sections are automatically Galois equivariant: one argues as above using in addition the exact sequence
   $$
   \Hom_{J(\Q_p)}(v_{{I_{[w]},n}},v_{{I_{[w^{\prime}],n}}})\to \Hom_{J(\Q_p)}( i_{{I_{[w]},n}}, v_{{I_{[w^{\prime}],n}}})\to \Hom_{J(\Q_p)}(i_{\Delta,n}, v_{{I_{[w^{\prime}],n}}}),
$$
   induced by the exact sequence (\ref{kolo-kwak}), which shows that the middle term is trivial since so are the other two terms (by Proposition \ref{prop:HomSt}).
 \end{proof}
 \subsection{End of the proof of Theorem \ref{main}} 
 
  (1) {\em Torsion compactly supported \'etale cohomology.} In order to prove claim (1), we use the distinguished triangle (\ref{zima1}).   An argument based on the Bruhat decomposition of $\sff\simeq G_{E_s}/P(\mu)$, as in the proof of Corollary \ref{computation-cor}, shows that 
 $$
 \rg_{\eet}(\sff_C,\Z/p^n)\simeq  \bigoplus_{[w]\in W^{\mu}/\sg_{E_s}}\rho_{[w],n}[-2l_{[w]}] :=\wt{E}_1.$$

     The map $$i^*: \rg_{\eet}(\sff_C,\Z/p^n)\to \rg_{\eet}(Y_C,\Z/p^n)$$ can be represented by the map of complexes $ \wt{E}_1\to E_{1}$
 induced by the canonical maps
 $$
 \rho_{[w],n}=\iota_{\Delta,n}\otimes  \rho_{[w],n}\to \iota_{{I},n}\otimes \rho_{[w],n},\quad I_{[w]}\subset I, |\Delta\setminus I|=1.
 $$
 By Propositions \ref{complex2} and \ref{cohomology},
 $$
 \Cone(i^*)[-1]\simeq 
 \bigoplus_{[w]\in W^{\mu}/\sg_{E_s}}v_{{I_{[w]},n}}\otimes \rho_{[w],n}[-2l_{[w]}-|\Delta\setminus I_{[w]}|],
 $$
as wanted.

(2) {\em $p$-adic compactly supported \'etale cohomology.}  For claim (2), take the exact sequence (see Section \ref{compact2}):
$$
0\to\R^1 \varprojlim_n H^{i-1}_{\eet,c}(\sff^{\rm wa}_C,\Z/p^n)\to H^i_{\eet,c}(\sff^{\rm wa}_C,\Z_p)\to \varprojlim_n H^i_{\eet,c}(\sff^{\rm wa}_C,\Z/p^n)\to 0.
$$
Since 
$$
v_{{I_{[w]},n}}\simeq v_{{I_{[w]}, n+1}}\otimes_{\Z/p^{n+1}}\Z/p^n,\quad v^{\cont}_{{I_{[w]}}}(\Z_p)\simeq \varprojlim_nv_{{I_{[w]},n}},
$$
the pro-system $ \{H^{i-1}_{\eet,c}(\sff^{\rm wa}_C,\Z/p^n)\}_n$ is Mittag-Leffler hence 
the above exact sequence and claim (1) yield the isomorphism
$$
H^i_{\eet,c}(\sff^{\rm wa}_C,\Z_p)\stackrel{\sim}{\to} \varprojlim_n H^i_{\eet,c}(\sff^{\rm wa}_C,\Z/p^n)
$$
and claim (2).
\begin{remark} In \cite{Oc} Orlik computed $\ell$-adic compactly supported \'etale cohomology for $\ell\neq p$. In the context of Theorem \ref{main}  and for $\ell$ sufficiently generic with respect to $G$ he obtained an isomorphism of $\sg_{E_s}\times J(\qp)$-modules
\begin{align}
\label{kwaku-kwaku}
 H^*_{\eet,{\rm c, Hu}}(\sff^{\rm wa}_C,\Z_\ell) & \simeq \bigoplus_{[w]\in W^{\mu}/\sg_{E_s}} v^J_{{I_{[w]}}}(\Z_\ell)\otimes \rho_{[w]}(\Z_{\ell})[-n_{[w]}],
\end{align}
where   $H^*_{\eet,{\rm c, Hu}}$ denotes  Huber's compactly supported cohomology. The proof follows the proof for torsion coefficients with two main differences:
\begin{enumerate}
\item It starts with the  
   distinguished triangle associated to the triple $(\sff^{\rm wa},\sff,Y)$:
 $$
  \rg_{\eet,{\rm c, Hu}}(\sff^{\rm wa}_C,\Z_p)\lomapr{ } \rg_{\eet}(\sff_C,\Z_p)\lomapr{i^*} \rg_{\eet}(Y_C,\Lambda), \quad \Lambda=i^*\R\pi_*(\Z/p^n)_n.
$$
\item  It uses the fact (from \cite[Appendix B.2]{Dat}) that the representations 
$$
H^*_{\eet, {\rm c, Hu}}(\sff^{\rm wa}_C,\Z_p), \rg_{\eet}(Y_C,\Lambda), \ldots
$$
are smooth $J(\Q_p)$-modules.
\end{enumerate}
As we have mentioned in the introduction the $p$-adic analog of the isomorphism (\ref{kwaku-kwaku}) is false and the smoothness property mentioned above does not hold.
\end{remark}
  \appendix
 \section{Adic potpourri}
 We gather here, as a reference, some basic facts concerning pseudo-adic spaces and compactly supported \'etale cohomology. 
    \subsection{Pseudo-adic spaces} We start with pseudo-adic spaces. 
         Recall that, Huber defines in \cite{H1} the category PPA of pre-pseudo-adic spaces, consisting of 
            pairs $X=(\underline{X}, |X|)$, where $\underline{X}$ is an adic space and $|X|$ a subset of $\underline{X}$, morphisms 
            $X\to Y$ being morphisms of adic spaces $\underline{X}\to \underline{Y}$ that send $|X|$ into $|Y|$. 
            A morphism $f: X\to Y$ induces therefore a morphism of adic spaces $\underline{f}: \underline{X}\to 
            \underline{Y}$ and a map of topological spaces $|f|: |X|\to |Y|$ (we endow $|X|$ with the topology induced from 
            $\underline{X}$). We say that $f$ is {\em \'etale} if 
            $\underline{f}$ is \'etale and if $|X|$ is open in $\underline{f}^{-1}(|Y|)$ (this implies that 
            $|f|$ is an open map). The \'etale site $X_{\eet}$ of $X$ is the category of pre-pseudo-adic spaces 
            $Y$ \'etale over $X$ with the topology such that a family of morphisms $f_i: Y_i\to Y$ in 
            this category is a covering if $|Y|=\cup_{i} |f_i|(|Y_i|)$. 
            
            We mention the following properties of this construction, which we need, and we refer the reader to Huber's book \cite{H1} for the proofs and details (see especially Sections 1.10, 2.3):
            
            \begin{enumerate}
            \item The category PPA contains (as full subcategory) the category of adic spaces (via $X\mapsto (X, |X|)$)
                        and the \'etale topoi of $X$ and $(X,|X|)$ are equivalent.

              \item If $X$ is an adic space and $S\subset T$ are subsets of $X$, the natural morphism
                        $i: (X,S)\to (X,T)$ in PPA satisfies $i^*i_*F\simeq F$, for all $F\in {\rm Sh}((X,S)_{\eet})$, thus 
                        $i_*: {\rm Sh}((X,S)_{\eet})\to {\rm Sh}((X,T)_{\eet})$ is fully faithful. Moreover, if 
                        $S$ is closed in $T$, then $i_*$ is exact and identifies ${\rm Sh}((X,S)_{\eet})$ with the full subcategory of 
                        ${\rm Sh}((X,T)_{\eet})$ consisting of sheaves $F$ whose restriction to $(T-S)_{\eet}$ is the final object of 
                        ${\rm Sh}((T-S)_{\eet})$ (\cite[Lemma 2.3.11]{H1}).
                        
                       \item   Let 
            PA be the full subcategory of PPA consisting of {\it pseudo-adic spaces}, i.e., those $X$ for which $|X|$ is convex and locally pro-constructible in 
            $\underline{X}$. An object $X$ of PA is called quasi-compact/quasi-separated if $|X|$ is so, and a map 
            $f: X\to Y$ in PA is called quasi-compact/quasi-separated if $|f|$ is so. If 
            $f: X\to Y$ is a quasi-compact quasi-separated morphism in PA and if 
            $f$ is adic (i.e., $\underline{f}$ is adic), then $\R^nf_*$ commutes with pseudo-filtered inductive limits. 
            If $X\in PA$ is quasi-compact quasi-separated, then $H^n_{\eet}(X,-)$ commutes with pseudo-filtered inductive limits.
            (\cite[Lemma 2.3.13]{H1}). 
          
          \item If $x$ is a point of an adic space $X$ and if $K$ is the henselization of the residue class field
          $k(x)$ with respect to the valuation ring $k(x)^+$, there is a natural equivalence of categories 
          ${\rm Sh}((X,\{x\})_{\eet})\simeq {\rm Sh}({\rm Spec}(K)_{\eet})$ (\cite[Prop.\ 2.3.10]{H1}). 
            
            \item 
                      Let $P$ be one of the properties ``open, closed, locally closed''. 
            A $P$-subspace of $X\in PPA$ is 
            an object $Y\in PPA$ for which $\underline{Y}$ is a $P$-subspace of 
            $\underline{X}$ and $|Y|$ is a $P$-subspace of $|X|$. The notion of $P$-embedding in PPA is defined  in the obvious way. 
            If $i: X\to Y$ is a locally closed embedding in PPA then $i$ induces an equivalence 
            ${\rm Sh}(X_{\eet})\simeq {\rm Sh}((\underline{Y}, i(|X|))_{\eet})$ (\cite[Cor.\ 2.3.8]{H1}). In particular if 
            $i: X\to Y$ is a locally closed embedding of adic spaces then 
            ${\rm Sh}(X_{\eet})\simeq {\rm Sh}((Y, i(|X|))_{\eet})$.   
                        
           \item           
           A morphism $f: X\to Y$ in PPA is finite if $\underline{f}$ is finite and $|X|$ is closed 
            in $\underline{f}^{-1}(|Y|)$. If $f: X\to Y$ is a finite morphism in PA, then 
            $f_*: {\rm Sh}(X_{\eet})\to {\rm Sh}(Y_{\eet})$ is exact and commutes with any base change in PA
            $Y'\to Y$ (\cite[Prop.\ 2.6.3]{H1}).

            \item A {\it geometric point} (in the category PPA) is an object $S\in PA$ such that $\underline{S}$ is the adic spectrum
             of a separably algebraically closed affinoid field (\cite[1.1.5]{H1}) and $|S|=\{s\}$, where $s$ is the closed point of 
             $\underline{S}$ (\cite[1.1.6]{H1}). For a geometric point $S$, the functor $\Gamma(S,-)$ induces an equivalence 
             ${\rm Sh}(S_{\eet})\simeq {\rm Sets}$.
                A {\it geometric point of $X\in PPA$} is a morphism
             $u: S\to X$ in PPA, where $S$ is a geometric point. The stalk of 
             $F\in {\rm Sh}(X_{\eet})$ at $S$ is then $F_{S}=\Gamma(S, u^*F)$. 
             Somewhat more explicitly, $F_{S}\simeq \varinjlim_{(V,v)} F(V)$, the limit being over 
             the cofiltered category $C_S$ of pairs $(V,v)$, where $V$ is \'etale over $X$ and 
             $v: S\to V$ is an $X$-morphism.  
                                     
                      The support of 
             $u$ is by definition $u(|S|)\in |X|$. Two geometric points with the same support yield isomorphic stalk functors. Moreover, each $x\in X$ induces a geometric point 
             $\bar{x}$ of $X$ with support $x$ and the family of functors ${\rm Sh}(X_{\eet})\to {\rm Sets}, F\to F_{\bar{x}}$, for $x\in |X|$, is conservative (\cite[2.5.5]{H1}).             
             If $f: X\to Y$ is a morphism of analytic pseudo-adic spaces (i.e., 
            $\underline{X}, \underline{Y}$ are analytic adic spaces) and $f$ is of weakly finite type and quasi-separated, then, 
            for any maximal point $y$ of $|Y|$ and any 
            $F\in {\rm Sh}(X_{\eet})$, we have 
            a natural isomorphism (\cite[Th.\ 2.6.2]{H1})
            $$(\R^n f_* F)_{\bar{y}}\simeq H^n_{\eet}(X\times_Y \bar{y}, F).$$
            
            \item One can define (see \cite[Sec.\ 2.5]{H1}), for each geometric point $\xi$ of $X\in PPA$, the 
            {\it strict localization $X(\xi)$ of $X$ at $\xi$}. It comes with an $X$-morphism
            $\xi\to X(\xi)$, and the isomorphism class of $X(\xi)$ as $X$-space  
            depends only on 
            the support of $\xi$. If $X\in PA$ and $\xi,\xi'$ are geometric points of 
            $X$, a {\it specialization morphism} $u:\xi\to \xi'$ is an $X$-morphism in PA 
            $X(\xi)\to X(\xi')$. It induces functorial maps $u^*(F): F_{\xi'}\to F_{\xi}$ for 
            $F\in {\rm Sh}(X_{\eet})$, via the natural isomorphisms 
            $\Gamma(X(\xi), F|X(\xi))\simeq F_{\xi}$ and 
            $\Gamma(X(\xi'), F|X(\xi'))\simeq F_{\xi'}$.
        \end{enumerate}     
        \subsection{Compactly supported  cohomology} We survey Huber's compactly supported \'etale cohomology and introduce continuous compactly supported \'etale cohomology.
\subsubsection{Huber's compactly supported \'etale cohomology} \label{compact1}
Huber defined compactly supported \'etale cohomology of analytic pseudo-adic spaces in \cite[Ch.\ 5]{H1}; in \cite{H2} he extended this definition to $\ell$-adic sheaves. We will briefly recall its properties.

   Fix a prime $\ell$. Let $X$ be a taut separated pseudo-adic space locally of $^{+}$weakly finite type over $C$ (i.e., over $\Spa(C,\so_C)$).  For $ i\geq 0$, we set
   $$
   H^i_{\eet,{\rm c,Hu}}(X,\Z_{\ell}):=H^i\R\Gamma_{{\rm c,Hu}}(X_{\eet},(\Z/\ell^n)_n),\quad    H^i_{\eet,{\rm c,Hu}}(X,\Q_\ell)   :=H^i_{\eet,{\rm c,Hu}}(X,\Z_\ell)\otimes\Q_\ell,
   $$
where the functor $\R\Gamma_{{\rm c,Hu}}$ is defined in the following way. 
   
   If $X$ is partially proper, then it is the right derived functor of 
   $\Gamma_{{\rm c,Hu}}$, i.e., of the left exact functor
   $$
   \Gamma_{\rm c,Hu}: {\rm mod}(X_{\eet}-\Z_\ell^{\jcdot})\to {\rm mod}(\Z_\ell),\quad (F_n)_n\mapsto\Gamma_c(X_{\eet},\varprojlim_n F_n).
   $$
   Here ${\rm mod}(X_{\eet}-\Z^{\jcdot}_\ell)$ is the category of projective systems $(F_n)_n$ of $\Z_\ell$-modules on $X_{\eet}$ such that $p^nF_n=0$, $n\in\N$. Recall that, for an \'etale sheaf $F$,    $\Gamma_c(X_{\eet},F)$ denotes the abelian group of global sections whose support is proper.

      In general one sets
      $$
      \R\Gamma_{{\rm c,Hu}}(X_{\eet},(F_n)_n):=\R\Gamma_{{\rm c,Hu}}(\overline{X}_{\eet},(i_!F_n)_n),
      $$
      where $i:X\hookrightarrow \overline{X}$ is a locally closed embedding and $\overline{X}$ is partially proper. This definition is, of course, independent of the chosen partially proper compactification.  We have $$\Gamma_{\rm c,Hu}(X_{\eet},(F_n)_n)=\{(s_n)_n\in\varprojlim_n\Gamma(X_{\eet},F_n)|\overline{\cup_n{\rm supp}(s_n)} \text{ is proper}\}.   
   $$

   We list the following properties:
   \begin{enumerate}
   \item If $X$ is proper then $$ \R\Gamma_{{\rm c,Hu}}(X_{\eet},(F_n)_n)\simeq  \R\Gamma(X_{\eet},(F_n)_n).
   $$ 
   \item An isomorphism \cite[Lemma 2.3]{H2} of exact functors from $D^+({\rm mod}(X_{\eet}-\Z^{\jcdot}_\ell))$ to $D^+({\rm mod}(\Z_\ell))$:
   $$
   \R\Gamma_{\rm c,Hu}=\R\Gamma_{!}\circ\R\pi_*,
   $$
   for the {\em discretization} functor
   $$
   \pi_*: {\rm mod}(X_{\eet}-\Z^{\jcdot}_\ell)\to {\rm mod}(X_{\eet}-\Z_\ell),\quad (F_n)_n\mapsto \varprojlim_nF_n,
   $$
   and the functor
   $$
   \Gamma_{!}: {\rm mod}(X_{\eet}-\Z_\ell)\to {\rm mod}(\Z_\ell),\quad F\mapsto \Gamma_{c}(X_{\eet},F).
   $$
   \item If $X$ is quasi-compact, there is an exact sequence \cite[Cor.\ 2.4]{H2}
   $$
   0\to \R^1\varprojlim_nH^{i-1}_{\eet,c}(X,F_n)\to H^i_{\eet,{\rm c,Hu}}(X,(F_n)_n)\to \varprojlim_nH^i_{\eet,c}(X,F_n)\to 0
   $$
   \item Let $U$ be a taut open subspace of $X$, let $Z=X\setminus U$, and let $i: Z\hookrightarrow X$ be the inclusion. Assume that $X,U$ are partially proper. Then we have a distinguished triangle
   \begin{align*}
  \rg_{{\rm c,Hu}}(U_{\eet},(F_n|U)_n)  \to \rg_{{\rm c,Hu}}(X_{\eet}, (F_n)_n)\to \rg_{c}(Z_{\eet}, i^*\R\pi_*(F_n)_n)
   \end{align*}
   \item Let ${\mathbb U} $ be an open covering of $X$ such that every $U \in {\mathbb U} $ is taut and, for every
$U, V\in {\mathbb U}$, there exists a $W \in {\mathbb U}$ such that  $U \cup V\subset W$. Then,  the map
$$
\varinjlim_{U\in {\mathbb U}}
H^i_{\eet,{\rm c,Hu}} (U,(F_n|U)_n)\to H^i_{\eet, {\rm c,Hu}} (X,(F_n)_n),\quad i\geq 0,$$
is an isomorphism \cite[Prop.\ 2.1.]{H2}.
\item Let $X$ be adic and partially proper and let $G$ be a locally profinite group acting continuously on $X$. 
Then $H^i_{\eet,c}(X,\Z/{\ell}^n)$, $i\geq 0$, is a smooth $G$-module \cite[Cor.\ 7.8]{B}. 
\end{enumerate}
   
   Let us now distinguish two cases.\vspace{2mm}
   
\noindent   (i) {\bf The case $\ell\neq p$.} We can say more in this case. 
   \begin{enumerate}
   \item If $X$ is as at the beginning of this section and of finite type over $C$ and if  $(F_n)_n$ is a quasi-constructible $\Z_\ell^{\jcdot}$-module on $X_{\eet}$ then the natural map 
   $$
   H^i_{\eet,{\rm c,Hu}}(X,(F_n)_n)\to \varprojlim_n H^i_{\eet,c}(X,F_n)
   $$
   is a bijection. Moreover, the projective system $(H^i_{\eet,c}(X,F_n))_n$ is $\ell$-adic and every $H^i_{\eet,c}(X,F_n)$ is a finitely generated $\Z_{\ell}$-module (hence also $H^i_{\eet,{\rm c,Hu}}(X,(F_n)_n)$ is a finitely generated $\Z_{\ell}$-module).
   \item If $X=Y^{\ad}$, for a separated scheme $Y$ of finite type over $C$, and if   $(F_n)_n$ is a constructible $\Z_\ell^{\jcdot}$-module on $Y_{\eet}$, there is a natural isomorphism 
   $$
    H^i_{\eet,c}(Y,(F_n)_n)\stackrel{\sim}{\to}  H^i_{\eet,{\rm c,Hu}}(X,(F_n)_n).
   $$
   \item Let $X$ be adic and partially proper  and  let $G$ be a locally profinite group, with an open pro-$p$ subgroup, acting continuously on $X$. Let $(F_n)_n$ be a locally constant overconvergent $\Z_{\ell}^{\jcdot}$-module equipped with a compatible discrete $G$-action (see \cite[B.1.3]{Dat} for the definition). 
Then $H^i_{\eet,{\rm c,Hu}}(X,(F_n)_n)$, $i\geq 0$, is a smooth $G$-module \cite[Prop.\ B.2.5]{Dat}, \cite[4.1.19]{Fl}. 
\end{enumerate}
   
 \noindent    (ii) {\bf The case $\ell= p$.} In this case cohomology with compact support behaves very differently. We will discuss an example. 
\begin{example} Let ${\mathbb A}^1_C$ be the adic affine space of dimension $1$; 
this is a period domain, with $G:={\mathbb G}_{m,\Q_p}\times {\mathbb G}_{m,\Q_p}$ the relevant reductive group \cite[5.3.1, 4.2.2]{And}. 
We have the exact sequence
\begin{align*}
0 & \to H^1_{\eet}(x_{\infty},i^*_{\infty}\R\pi_*(\Z/p^n(1))_n) \to H^2_{\eet,{\rm c,Hu}}({\mathbb A}^1_C,\Z_p(1)) \to H^2_{\eet}({\mathbb P}^1_C,\Z_p(1)) \\
& \to H^2_{\eet}(x_{\infty},i^*_{\infty}\R\pi_*(\Z/p^n(1))_n),
\end{align*}
where $i_{\infty}: x_{\infty}\hookrightarrow {\mathbb P}^1_C$ is the point at infinity. Picking the fundamental neighborhoods of $x_{\infty}$ consisting of closed balls, we compute easily that
$$
i^*_{\infty}\R^i\pi_*(\Z/p^n(1))_n\simeq \begin{cases}\Z_p(1) & \mbox{ if } i=0,\\
\varinjlim_j (\varprojlim_nH^1_{\eet}(E(j),\Z/p^n(1))) & \mbox { if } i=1,\\
0 & \mbox{ if } i\geq 2,
\end{cases}
$$ 
where $E(j)$ is the closed ball centered at $x_{\infty}$ and of radius $p^{-j}$.  We used here the fact that  $H^i_{\eet}(E(j),\Z/p^n(1))=0, i\geq 2$. Since ${\rm Pic}(E(j))=0$, the Kummer exact sequence implies that 
$$H^1_{\eet}(E(j),\Z/p^n(1))\simeq C\{T^{-1}\}^*/C\{T^{-1}\}^{*p^n}.$$
Hence $H^2_{\eet}(x_{\infty},i^*_{\infty}\R\pi_*(\Z/p^n(1))_n)=0$. We also claim  that $$H^1_{\eet}(x_{\infty},i^*_{\infty}\R\pi_*(\Z/p^n(1))_n) \simeq\varinjlim_j C\{(p^jT)^{-1}\}^{*{\wedge}}.$$
Since $x_{\infty}$ is simply a geometric point (thus \'etale sheaves are acyclic on it), we have 
(using the local-global spectral sequence relating $H^i_{\eet}(x_{\infty}, H^j(K))$ and $H^{i+j}_{\eet}(x_{\infty}, K)$ for a complex of sheaves $K$, as well as the computation above)
\begin{align*}
H^1_{\eet}(x_{\infty},i^*_{\infty}\R\pi_*(\Z/p^n(1))_n) &\simeq H^0_{\eet}(x_{\infty},
i^*_{\infty}\R^1\pi_*(\Z/p^n(1))_n) 
\simeq \varinjlim_j (\varprojlim_n H^1_{\eet}(E(j),\Z/p^n(1))) \\
&\simeq \varinjlim_j (\varprojlim_n C\{T^{-1}\}^*/C\{T^{-1}\}^{*p^n}) 
\simeq 
\varinjlim_j C\{(p^jT)^{-1}\}^{*{\wedge}}.
\end{align*}
 It follows that 
     \begin{equation*}
    H^2_{\eet,{\rm c,Hu}}({\mathbb A}^1_C,\Z_p(1))\simeq  
(\varinjlim_n C\{(p^nT)^{-1}\}^{*{\wedge}} )\oplus \Z_p.
    \end{equation*}
 In the case $\ell\neq p$, the same computation gives $\Z_\ell$ as a result
since $C\{(p^n T)^{-1}\}^{*{\wedge}}$ is $\ell$-divisible.
Note that $C\{T^{-1}\}^*/C^*=1+T^{-1}{\mathfrak m}_C\{T^{-1}\}$ and that
its image by the logarithm satisfies
$$p T^{-1}{\mathfrak m}_C\{T^{-1}\}\subset \log\big(
1+T^{-1}{\mathfrak m}_C\{T^{-1}\}\big)
\subset (pT)^{-1}{\mathfrak m}_C\{(pT)^{-1}\}.$$
One gets the same inclusions for the $p$-adic completion.  Hence the above
inductive limit is isomorphic, via the logarithm, to the inductive
limit of the $(p^nT)^{-1}{\mathfrak m}_C\{(p^nT)^{-1}\}$ and so
$$H^2_{\eet,{\rm c,Hu}}({\mathbb A}^1_C,\Z_p(1))\simeq
\big(\so_{{\mathbb P}^1,\infty}/C\big)\oplus \Z_p.$$
     
Hence the $\ell$-adic compactly supported cohomology groups behave 
very differently in the cases $\ell=p$ and $\ell\neq p$, where $H^2_{\eet,{\rm c,Hu}}({\mathbb A}^1_C,\Z_{\ell}(1))\simeq \Z_{\ell}$.    
Note also that the action of $G(\Q_p)$ on 
$H^2_{\eet,{\rm c,Hu}}({\mathbb A}^1_C,\Z_p(1))$ is not smooth, 
contrary to the case $\ell\neq p$.
     \end{example}

      \subsubsection{Continuous compactly supported \'etale cohomology} \label{compact2} We will also study  a different version of Huber's compactly supported cohomology: For $X$ as in Section \ref{compact1}, we define its (continuous) compactly supported cohomology by: 
   $$
   \rg_{\rm c}(X_{\eet},(F_n)_n):= \R\varprojlim_n\rg_c(X_{\eet},F_n).
   $$
     We have $$\Gamma_{\rm c}(X_{\eet},(F_n)_n)=\{(s_n)_n\in\varprojlim_n\Gamma(X_{\eet},F_n)|{\rm supp}(s_n) \text{ is proper}\}.
   $$

   The following properties are obtained directly from the definition and the corresponding properties for the compactly supported cohomology of $F_n$'s.
     \begin{enumerate}
   \item There is an exact sequence 
   $$
   0\to \R^1\varprojlim_nH^{i-1}_{\eet,c}(X,F_n)\to H^i_{\eet,{\rm c}}(X,(F_n)_n)\to \varprojlim_nH^i_{\eet,c}(X,F_n)\to 0.
   $$
     \item Let $U$ be a taut open subspace of $X$, let $Z=X\setminus U$, and let $i: Z\hookrightarrow X$ be the inclusion. Then we have a distinguished triangle
   \begin{align*}
  \rg_{\rm c}(U_{\eet},(F_n|U)_n)  \to \rg_{\rm c}(X_{\eet}, (F_n)_n)\to \rg_{\rm c}(Z_{\eet}, i^*(F_n)_n).
   \end{align*}

\end{enumerate}
To lighten the notation, for $i\geq 0$, we will set 
$$
H^i_{\eet,{\rm c}}(X,\Z_p):=H^i_{\eet,{\rm c}}(X,(\Z/p^n)_n),\quad H^i_{\eet,{\rm c}}(X,\Q_p):=H^i_{\eet,{\rm c}}(X,\Z_p)\otimes_{\Z_p}\Q_p.
$$

\end{document}